\documentclass[11pt,leqno]{article}
\usepackage[english,activeacute]{babel}
\usepackage{epsfig}
\usepackage[latin1]{inputenc}
\usepackage[T1]{fontenc}
\usepackage{amsmath,amsfonts,amssymb,mathrsfs,amsthm}
\usepackage{graphicx}
\usepackage[mathscr]{eucal}
\usepackage{makeidx}
\usepackage{verbatim}
\usepackage{graphics,graphicx}
\usepackage{textcomp}

\usepackage{textcomp}

\usepackage{graphicx}
\usepackage{graphpap}

\usepackage{ifpdf}
\ifpdf
\usepackage{graphicx}  \DeclareGraphicsExtensions{.pdf,.mps}

\else
\usepackage{amsfonts}
\usepackage{graphicx}  \DeclareGraphicsExtensions{.ps,.eps}

\fi

\def\pasdegrille{\let\grille = \pasgrille}
\def\ecriture#1#2{\setbox1=\hbox{#1}
\dimen1= \wd1
\dimen2=\ht1
\dimen3=\dp1
\grille #2 \box1 }
\def\aat#1#2#3{
\divide \dimen1 by 48
\dimen3=\dimen1
\multiply \dimen1 by #1
\advance \dimen1 by -\dimen3
\divide \dimen1 by 101
\multiply \dimen1 by 100
\divide \dimen2 by \count11
\multiply \dimen2 by #2 
\setbox0=\hbox{#3}\ht0=0pt\dp0=0pt
  \rlap{\kern\dimen1 \vbox to0pt{\kern-\dimen2\box0\vss}}\dimen1= \wd1
\dimen2=\ht1}
\def\pasgrille{
\count12= \dimen1 
\divide \count12 by 50
\divide \dimen2 by \count12
\count11 =\dimen2
\ 
\divide \dimen1 by 48
\setlength{\unitlength}{\dimen1}
\smash{\rlap{\ }}
\dimen1= \wd1
\dimen2=\ht1
}
\def\grille{
\count12= \dimen1 
\divide \count12 by 50
\divide \dimen2 by \count12
\count11 =\dimen2
\ 
\divide \dimen1 by 48
\setlength{\unitlength}{\dimen1}
\smash{\rlap{\graphpaper[1](0,0)(50, \count11)}}
\dimen1= \wd1
\dimen2=\ht1
}

\pasdegrille

\def\e{{\varepsilon}}

\newcommand\R{\mathbb{R}}
\newcommand\C{\mathbb{C}}
\newcommand\N{\mathbb{N}}
\newcommand\Z{\mathbb{Z}}

\newcommand\bna{\begin{eqnarray*}}
\newcommand\ena{\end{eqnarray*}}

\newcommand\bnan{\begin{eqnarray}}
\newcommand\enan{\end{eqnarray}}

\newcommand\bnp{\begin{proof}}
\newcommand\enp{\end{proof}}

\newcommand\Hu{H^1}
\newcommand\Hup{\dot{H^1}}
\newcommand\HutL{\mathcal{E}}

\newcommand\nor[2]{\left\|#1\right\|_{#2}}
\newcommand\nort[1]{\left|\left|\left|#1\right|\right|\right|}

\newcommand\tend[2]{\underset{#1 \to #2}{\longrightarrow}}
\newcommand\tendweak[2]{\underset{#1 \to #2}{\rightharpoonup}}
\newcommand\limsu[2]{\underset{#1\to #2}{\varlimsup}}
\newcommand\limvar[2]{\underset{#1\to #2}{\lim}}

\newcommand\petito[1]{o(#1)}
\newcommand\grando[1]{\mathcal{O}(#1)}

\newtheorem{theorem}{Theorem}[section]
\newtheorem{remarque}{Remark}[section]
\newtheorem{lemme}{Lemma}[section]
\newtheorem{corollaire}{Corollary}[section]
\newtheorem{prop}{Proposition}[section]
\newtheorem{hypo}{Assumption}[section]
\newtheorem{definition}{Definition}[section]

\title{On stabilization and control for the critical Klein-Gordon equation on a 3-D compact manifold}
\author{Camille Laurent\thanks{Laboratoire de Math\'{e}matiques d'Orsay, UMR 8628 CNRS, Universit\'e Paris-Sud,  Orsay Cedex, F-91405 (camille.laurent@math.u-psud.fr).}
}

\begin{document}

\maketitle
\begin{abstract}
In this article, we study the internal stabilization and control of the critical nonlinear Klein-Gordon equation on 3-D compact manifolds. Under a geometric assumption slightly stronger than the classical geometric control condition, we prove exponential decay for some solutions bounded in the energy space but small in a lower norm. The proof combines profile decomposition and microlocal arguments. This profile decomposition, analogous to the one of Bahouri-Gérard \cite{BahouriGerard} on $\R^3$, is performed by taking care of possible geometric effects. It uses some results of S. Ibrahim \cite{ibrahim2004gon} on the behavior of concentrating waves on manifolds.\\
\end{abstract}
\vspace{0.2 cm}{\bf AMS subject classifications.} 93B05, 93D15, 35L70, 35Q93 
\tableofcontents
\section*{Introduction}
In this article, we study the internal stabilization and exact controllability for the defocusing critical nonlinear Klein-Gordon equation on some compact manifolds. 
\begin{eqnarray}
\label{eqncontrolnl}
\left\lbrace
\begin{array}{rcl}
\Box u=\partial^2_t u - \Delta u &=&-u-|u|^4u \quad \textnormal{on}\quad [0,+\infty [\times M\\
(u(0),\partial_t u(0))&=&(u_{0},u_1) \in \HutL.
\end{array}
\right.
\end{eqnarray}
where $\Delta$ is the Laplace-Beltrami operator on $M$ and $\HutL$ is the energy space $H^1(M)\times L^2(M)$. The solution displays a conserved energy
\bnan
\label{defnonmlinearNRJ}
E(t)=\frac{1}{2}\left(\int_M \left|\partial_t u \right|^2+\int_M \left|u \right|^2+\int_M \left|\nabla u \right|^2\right) +\frac{1}{6}\int_M \left|u\right|^6.
\enan
This problem was already treated in the subcritical case by B. Dehman, G. Lebeau and E. Zuazua \cite{DLZstabNLW}. The problem is posed in a different geometry but their proof could easily be transposed in our setting. Yet, their result fails to apply to the critical problem for two main reasons, as explained in their paper : \\
(a) The boot-strap argument they employ to improve the regularity of solutions vanishing
in the zone of control $\omega$ so that the existing results on unique continuation apply, does not
work for this critical exponent.\\
(b) They can not use the linearizability results by P. Gérard \cite{linearisationondePG} to deduce that the microlocal
defect measure for the nonlinear problem propagates as in the linear case.

In this paper, we propose a strategy to avoid the second difficulty at the cost of an additional condition for the subset $\omega$. It was already performed by B. Dehman and P. Gérard \cite{DehPGNLW} in the case of $\R^3$ with a flat metric. In fact, in that case, this defect of linearisability is described by the profile decomposition of H. Bahouri and P. Gérard \cite{BahouriGerard}. The purpose of this paper is to extend a part of this proof to the case of a manifold with a variable metric. This more complicated geometry leads to extra difficulties, in the profile decomposition and the stabilization argument. We also mention the recent result of L. Aloui, S. Ibrahim and K. Nakanishi \cite{dampingcriticAlIbrNak} for $\R^d$. Their method of proof is very different and uses Morawetz-type estimates. They obtain uniform exponential decay for a damping around spatial infinity for any nonlinearity, provided the solution exists globally. This result is stronger than ours, but their method does not seem to apply to the more complicated geometries we deal with. 

We will need some geometrical condition to prove controllability. The first one is the classical geometric control condition of Rauch and Taylor \cite{RauchTaylordecay} and Bardos Lebeau Rauch \cite{BLR}, while the second one is more restrictive. 
\begin{hypo}[Geometric Control Condition]\label{hypGCC}There exists $T_0>0$ such that every geodesic travelling at speed $1$ meets $\omega$ in a time $t<T_0$.
\end{hypo}
\begin{definition}
\label{defcouplefocus}
We say that $(x_1,x_2,t)\in M^2\times \R$ is a couple of focus at distance $t$ if the set
\bna
F_{x_1,x_2,t}:=\left\{\left.\xi \in S^*_{x_1}M\right| exp_{x_1}t\xi =x_2\right\}
\ena
of directions of geodesics stemming from $x_1$ and reaching $x_2$ in a time $t$, has a positive surface measure.

We denote $T_{focus}$ the infimum of the $t\in \R$ such that there exists a couple of focus at distance $t$. 
\end{definition}
If $M$ is compact, we have necessarily $T_{focus}>0$.

\begin{hypo}[Geometric control before refocusing]\label{hypGCCfocus}The open set $\omega$ satisfies the Geometric Control Condition in a time $T_0<T_{focus}$.
\end{hypo}
For example, for $\mathbb{T}^3$, there is no refocusing and the geometric assumption is the classical Geometric Control Condition. Yet, for the sphere $S^3$, our assumption is stronger. For example, it is fulfilled if $\omega$ is a neighborhood of $\left\{x_4=0\right\}$. We can imagine some geometric situations where the Geometric Control Condition is fulfilled while our condition is not, for example if we take only a a neighborhood of $\left\{x_4=0,x_3\geq0\right\}$ (see Remark \ref{rmkhyp} and Figure \ref{fig.sphere} for $S^2$). We do not know if the exponential decay is true in this case.

The main result of this article is the following theorem.
\begin{theorem}
\label{mainthm}
Let $R_0>0$ and $\omega$ satisfying Assumption \ref{hypGCCfocus}. Then, there exist $T>0$ and $\delta>0$ such that for any $(u_0,u_1)$ and $(\tilde{u}_0,\tilde{u}_1)$ in $H^1\times L^2$, with 
\bna
\nor{(u_0,u_1)}{H^1\times L^2} \leq R_0; &\quad &\nor{(\tilde{u}_0,\tilde{u}_1)}{H^1\times L^2}\leq R_0\\
\nor{(u_0,u_1)}{L^2\times H^{-1}} \leq \delta; &\quad &\nor{(\tilde{u}_0,\tilde{u}_1)}{L^2\times H^{-1}}\leq \delta
\ena
there exists $g\in L^{\infty}([0,T],L^2)$ supported in $[0,T]\times \omega$ such that the unique strong solution of 
\begin{eqnarray*}
\left\lbrace
\begin{array}{rcl}
\Box u+u+|u|^4u&=& g\quad \textnormal{on}\quad [0,T]\times M\\
(u(0),\partial_t u(0))&=&(u_0,u_1) .
\end{array}
\right.
\end{eqnarray*}
satisfies $(u(T),\partial_t u(T))=(\tilde{u}_0,\tilde{u}_1)$.
\end{theorem}
Let us discuss the assumptions on the size. In some sense, our theorem is a high frequency controllability result and expresses in a rough physical way that we can control some "small noisy data". In the subcritical case, two similar kind of results were proved : in Dehman Lebeau Zuazua \cite{DLZstabNLW} similar results were proved for the nonlinear wave equation but without the smallness assumption in $L^2\times H^{-1}$ while in Dehman Lebeau \cite{HUMDehLeb}, they obtain similar high frequency controllability results for the subcritical equation but in a uniform time which is actually the time of linear controllability (see also the work of the author \cite{LaurentNLSdim3} for the Schrödinger equation). Actually, this smallness assumption is made necessary in our proof because we are not able to prove the following  unique continuation result.

\medskip

\noindent \textbf{Missing theorem.} \textit{$u\equiv 0$ is the unique strong solution in the energy space of 
 \bna
 \left\lbrace
\begin{array}{rcl}
\Box u+u+|u|^4u&=& 0\quad \textnormal{on}\quad [0,T]\times M\\
\partial_t u&=&0\quad \textnormal{on}\quad [0,T]\times \omega.
\end{array}
\right.
 \ena
}
 
\medskip

In the subcritical case, this kind of theorem can be proved with Carleman estimates under some additional geometrical conditions and once the solution is known to be smooth. Yet, in the critical case, we are not able to prove this propagation of regularity. Note also that H. Koch and D. Tataru \cite{KochTataruCarlLp} managed to prove some unique continuation result in the critical case, but in the case $u=0$ on $\omega$ instead of $\partial_t u=0$. In the case of $\R^3$ with flat metric and $\omega$ the complementary of a ball, B. Dehman and P. Gérard \cite{DehPGNLW} prove this theorem using the existence of the scattering operator proved by K. Nakanishi \cite{NakanishiScattering}, which is not available on a manifold.  

Moreover, as in the subcritical case, we do not know if the time of controllability does depend on the size of the data. This is actually still an open problem for several nonlinear evolution equations such as nonlinear wave or Schrödinger equation (even in the subcritical case). Note that for certain nonlinear parabolic equations, it has been proved that we can not have controllability in arbitrary short time, see \cite{FernandZuazheat} or \cite{FernandezGuerBurger}. 

\bigskip 

The strategy for proving Theorem \ref{mainthm} consists in proving a stabilization result for a damped nonlinear Klein-Gordon equation and then, by a perturbative argument using the linear control, to bring the solution to zero once the energy of the solution is small enough. Namely, we prove
\begin{theorem}
\label{thmdecresexp}
Let $R_0>0$, $\omega$ satisfying Assumption \ref{hypGCCfocus} and $a \in C^{\infty}(M)$ satisfying $a(x)>\eta>0$ for all $x\in \omega$. Then, there exist $C, \gamma>0$ and $\delta>0$ such that for any $(u_0,u_1)$ in $H^1\times L^2$, with 
\bna
\nor{(u_0,u_1)}{H^1\times L^2} \leq R_0; &\quad &\nor{(u_0,u_1)}{L^2\times H^{-1}} \leq \delta; 
\ena
the unique strong solution of 
\begin{eqnarray}
\label{solutiondampedbis}
\left\lbrace
\begin{array}{rcl}
\Box u+u+|u|^4u+a(x)^2\partial_tu &=& 0\quad \textnormal{on}\quad [0,T]\times M\\
(u(0),\partial_t u(0))&=&(u_0,u_1) .
\end{array}
\right.
\end{eqnarray}
satisfies $E(u)(t)\leq Ce^{-\gamma t}E(u)(0)$.
\end{theorem}
This theorem is false for the classical nonlinear wave equation (see subsection \ref{subsectcontrexwave}) and it is why we have chosen the Klein-Gordon equation instead.

Let us now discuss the proof of Theorem \ref{thmdecresexp}, following B. Dehman and P. Gérard \cite{DehPGNLW} for the case of $\R^3$. We have the energy decay
\bna
E(u)(t)=E(u)(0)-\int_0^t \int_M |a(x)\partial_t u|^2.
\ena
So, the exponential decay is equivalent to an observability estimate for the nonlinear damped equation. We prove it by contradiction. We are led to proving the strong convergence to zero of a normalized sequence $u_n$ of solutions contradicting observability. In the subcritical case, the argument consisted in two steps 
\begin{itemize}
\item to prove that the limit is zero by a unique continuation argument
\item to prove that the convergence is actually strong by linearization and linear propagation of compactness thanks to microlocal defect measures of P. Gérard \cite{defectmeasure} and L. Tartar \cite{tartarh-measure}.
\end{itemize}
By linearization, we mean (according to the terminology of P. Gérard \cite{linearisationondePG}) that we have\\ $\nort{u_n-v_n}\tend{n}{\infty}0$ where $v_n$ is solution of the linear Klein-Gordon equation with same initial data :
\bna
\left\lbrace
\begin{array}{rcl}
\Box v_n+v_n&=& 0\quad \textnormal{on}\quad [0,T]\times M\\
(v_n(0),\partial_t v_n(0))&=&(u_n(0),\partial_t u_n(0)).
\end{array}
\right.
\ena 

In our case, the smallness assumption in the lower regularity $L^2\times H^{-1}$ makes that the limit is automatically zero, which allows to skip the first step. In the subcritical case, any sequence weakly convergent to zero is linearizable. Yet, for critical nonlinearity, there exists nonlinearizable sequences. Hopefully, in the case of $\R^3$, this defect can be precisely described. It is linked to the non compact action of the invariants of the equation : the dilations and translations. More precisely, the work of H. Bahouri and P. Gérard \cite{BahouriGerard} states that any bounded sequence $u_n$ of solutions to the nonlinear critical wave equation can be decomposed into an infinite sum of : the weak limit of $u_n$, a sequence of solutions to the free wave equation and an infinite sum of profiles which are translations-dilations of fixed nonlinear solutions. This decomposition was used by the authors of \cite{DehPGNLW} to get the expected result in $\R^3$. Therefore, we are led to make an analog of this profile deomposition for compact manifolds. We begin by the definition of the profiles.
 \begin{definition}
 \label{defconcdata}
Let $x_{\infty} \in M$ and $(f,g)\in \mathcal{E}_{x_{\infty}}=(\dot{H}^1\times L^2)(T_{x_{\infty}}M)$. Given $[(f,g),\underline{h},\underline{x}] \in \mathcal{E}_{x_{\infty}}\times (\R_+^* \times M )^{\N}$ such that $lim_n (h_n,x_n)=(0,x_{\infty}) $
We call the associated concentrating data the class of equivalence, modulo sequences convergent to $0$ in $\HutL$, of sequence in $\mathcal{E}$ that take the form 
\bnan
\label{formconcentrdata}
h_n^{-\frac{1}{2}}\Psi_U(x) \left(f,\frac{1}{h_n}g\right)\left(\frac{x-x_n}{h_n}\right)+\petito{1}_{\HutL}
\enan
 in some coordinate patch $U_M\approx U\subset \R^d$ containing $x_{\infty}$ and for some $\Psi_U\in C^{\infty}_0(U)$ such that $\Psi_U(x)=1$ in a neighborhood of $x_{\infty}$. (Here, we have identified $x_n, x_{\infty}$ with its image in $U$).
\end{definition}
We will prove later (Lemma \ref{lemmedefconc}) that this definition does not depend on the coordinate charts and on $\Psi_U$: two sequences defined by (\ref{formconcentrdata}) in different coordinate charts are in the same class. In what follows, we will often call concentrating data associated to $[(f,g),\underline{h},\underline{x}]$ an arbitrary sequence in this class.
\begin{definition}
\label{defconcwave}
Let ($t_n$) a bounded sequence in $\R$ converging to $t_{\infty}$ and $(f_n,g_n)$ a concentrating data associated to $[(f,g),\underline{h},\underline{x}]$. A damped linear concentrating wave is a sequence $v_n$ solution of 
\bnan
\label{eqnadampedlinintro}
\left\lbrace
\begin{array}{rcl}
\Box v_n+v_n+a(x)\partial_t v_n&=&0 \quad \textnormal{on}\quad \R \times M\\
(v_n(t_n),\partial_t v_n(t_n))&=&(f_n,g_n).
\end{array}
\right.
\enan
The associated damped nonlinear concentrating wave is the sequence $u_n$ solution of 
\bnan
\label{eqnadampednonlinintro}
\left\lbrace
\begin{array}{rcl}
\Box u_n+u_n+a(x)\partial_t u_n+|u_n|^4u_n&=&0 \quad \textnormal{on}\quad \R \times M\\
(u_n(0),\partial_t u_n(0))&=&(v_n(0),\partial_t v_n(0)).
\end{array}
\right.
\enan
If $a\equiv 0$, we will only write linear or nonlinear concentrating wave.
\end{definition}
Energy estimates yields that two representants of the same concentrating data have the same associated concentrating wave modulo strong convergence in $L_{loc}^{\infty}(\R, \HutL)$. This is not obvious for the nonlinear evolution but will be a consequence of the study of nonlinear concentrating waves.

It can be easily seen that this kind of nonlinear solutions are not linearizable. Actually, it can be shown that this concentration phenomenon is the only obstacle to linearizability. We begin with the linear decomposition.
 \begin{theorem}
\label{thmdecompositionL}Let ($v_n$) be a sequence of solutions to the damped Klein-Gordon equation (\ref{eqnadampedlinintro}) with initial data at time $t=0$ $(\varphi_n,\psi_n)$ bounded in $\HutL$. Then, up to extraction, there exist a sequence of damped linear concentrating waves $(\underline{p}^{(j)})$, as defined in Definition \ref{defconcwave}, associated to concentrating data $[(\varphi^{(j)},\psi^{(j)}),\underline{h}^{(j)},\underline{x}^{(j)},\underline{t}^{(j)}]$, such that for any $l\in \N^*$,
\bnan
v_n(t,x)=v(t,x)+ \sum_{j=1}^l p_n^{(j)}(t,x)+w_n^{(l)}(t,x),\label{decomlinthm}\\
\forall T>0,\quad \underset{n\to \infty}{\varlimsup} \nor{w_n^{(l)}}{L^{\infty}([-T,T],L^6(M))\cap L^5([-T,T],L^{10})}  \tend{l}{\infty}0\\
\nor{(v_n,\partial_tv_n)}{\HutL}^2=\sum_{j=1}^l \nor{(p_n^{(j)},\partial_t p_n^{(j)})}{\HutL}^2 + \nor{(w_n^{(l)}),\partial_t w_n^{(l)})}{\HutL}^2 +  \petito{1}, \textnormal{ as } n\to \infty, \label{orthlinthm}
\enan
where $\petito{1}$ is uniform for $t\in[-T,T]$.
\end{theorem}
The nonlinear flow map follows this decomposition up to an error term in the strong following norm
\bna
\nort{u}_{I}=\nor{u}{L^{\infty}(I,H^1(M))}+\nor{\partial_t u}{L^{\infty}(I,L^2(M))}+\nor{u}{L^5(I,L^{10}(M))}.
\ena
\begin{theorem}
\label{thmdecompositionNL}
Let $T<T_{focus}/2$. Let $u_n$ be the sequence of solutions to damped non linear Klein-Gordon equation (\ref{eqnadampednonlinintro}) with initial data, at time $0$, $(\varphi_n,\psi_n)$ bounded in $\HutL$. Denote $p_n^{(j)}$ (resp $v$ the weak limit) the linear damped concentrating waves given by Theorem \ref{thmdecompositionL} and $q_n^{(j)}$ the associated nonlinear damped concentrating wave (resp $u$ the associated solution of the nonlinear equation with $(u,\partial_tu)_{t=0}=(v,\partial_tv)_{t=0}$). Then, up to extraction, we have
\bnan
u_n(t,x)= u+\sum_{j=1}^l q_n^{(j)}(t,x)+w_n^{(l)}(t,x)+r_n^{(l)} \label{decomponl}\\
\limsu{n}{\infty} \nort{r_n^{(l)}}_{[-T,T]} \tend{l}{\infty}0 \label{rnpetit}
\enan
where $w_n^{(l)}$ is given by Theorem \ref{thmdecompositionL}.

The same theorem remains true if  $M$ is the sphere $S^3$ and $a\equiv 0$ (undamped equation) without any assumption on the time $T$.
\end{theorem} 
The more precise result we get for the sphere $S^3$ will not be useful for the proof of our controllability result. Yet, we have chosen to give it because it is the only case where we are able to describe what happens when some refocusing occurs.

This profile decomposition has already been proved for the critical wave equation on $\R^3$ by H. Bahouri and P. Gerard \cite{BahouriGerard} and on the exterior of a convex obstacle by I. Gallagher and P. Gérard \cite{PGGall2001}. The same decomposition has also been performed for the Schrödinger equation by S. Keraani \cite{Keraanidefect} and quite recently for the wave maps by Krieger and Schlag \cite{KriegSchlagwavemap}.
Note that such decomposition has proved to be useful in different contexts : the understanding of the precise behavior near the threshold for well-posedness for focusing nonlinear wave see Kenig Merle \cite{KenigMerleNLW} and Duyckaerts Merle\cite{DuyckMerleNLW}, the study of the compactness of Strichartz estimates and maximizers for Strichartz estimates, (see Keraani \cite{Keraanidefect}), global existence for wave maps \cite{KriegSchlagwavemap}... May be our decomposition on manifolds could be useful in one of these contexts. Let us also mention that, this kind of decomposition appears for a long time in the context of Palais-Smale sequences for critical elliptic equation and optimal constant for Sobolev embedding, but with a finite number of profiles, see Brezis Coron \cite{BrezisCoronBubbles}, the book \cite{BookDruetBlow} and the references therein ...

Let us describe quickly the proof of the decomposition. The linear decomposition of Theorem \ref{thmdecompositionL} is made in two steps : first, we decompose our sequence in a sum of an infinite number of sequences oscillating at different rate $h_n^{(j)}$. Then, for each part oscillating at a fixed rate, we extract the possible concentration at certain points. We only have to prove that this process produces a rest $w_n^l$ that gets smaller in the norm $L^{\infty}L^6$ at each stage. Once the linear decomposition is established, Theorem \ref{thmdecompositionNL} says, roughly speaking, that the nonlinear flow map acts almost linearly on the linear decomposition. To establish the nonlinear decomposition we have to prove that each element of the decomposition do not interact with the others. For each element of the linear decomposition, we are able to describe the nonlinear solution arizing from this element as initial data. The linear rest $w_n^l$ is small in $L^{\infty}([-T,T],L^6)$ for $l$ large enough and so the associated nonlinear solution with same initial data is very close to the linear one. The behavior of nonlinear concentrating waves is described in \cite{ibrahim2004gon} (see subsection \ref{subsectbehavNL} for a short review). Before the concentration, linear and nonlinear waves are very close. For times close to the time of concentration, the nonlinear rescaled solution behaves as if the metric was flat and is subject to the scattering of $\R^3$. After concentration, the solution is close to a linear concentrating wave but with a new profile obtained by the scattering operator on $\R^3$.

We finish this introduction by a discussion on the geometric conditions we imposed to get our main theorem. For the linear wave equation, the controllability is known to be equivalent to the so called Geometric Control Condition (Assumption \ref{hypGCC}). This was first proved by Rauch and Taylor \cite{RauchTaylordecay} in the case of a compact manifold and by Bardos Lebeau Rauch \cite{BLR} for boundary control (see Burq Gérard for the necessity \cite{BurqGerardCNS}). For the nonlinear subcritical problem, the result of \cite{HUMDehLeb} only requires the classical Geometric Control Condition. Our assumption  is stronger and we can naturally wonder if it is really necessary. It is actually strongly linked with the critical behavior and nonlinear concentrating waves. Removing this stronger assumption would require a better understanding of the scattering operator of the nonlinear equation on $\R^3$ (see Remark \ref{rmkhyp}). However, we think that the same result could be obtained with the following weaker assumption.
\begin{hypo}\label{hypocontrolfocus} $\omega$ satisfies the Geometric Control Condition. Moreover, for every couple of focus $(x_1,x_2,t)$ at distance $t$, according to Definition \ref{defcouplefocus}, each geodesic starting from $x_1$ in direction $\xi$ such that $exp_{x_1}t\xi =x_2$ meets $\omega$ in a time $0\leq s<t$.
\end{hypo}

Finally, we note that our theorem can easily be extended to the case of $\R^3$ with a metric flat at infinity. In this case, our stabilization term $a(x)$ should fulfill the both assumptions
\begin{itemize}
\item there exist $R>0$ and $\rho>0$ such that $a(x)>\rho$ for $|x|>R$
\item $a(x)>\rho$ for $x\in \omega$ where $\omega$ satisfies Assumption \ref{hypGCCfocus}. 
\end{itemize}
The proof would be very similar. The only difference would come from the fact that the domain is not compact. So the profile decomposition would require the "compactness at infinity" (see property (1.6) of \cite{BahouriGerard}). Moreover, the equirepartition of the energy could not be made only with measures but with an explicit computation (see (3.14) of \cite{DehPGNLW})   
\begin{remarque}
\label{rmkhyp}
In order to know if our stronger Assumptions \ref{hypGCCfocus} or \ref{hypocontrolfocus} are really necessary compared to the classical Geometric Control Condition, we need to prove that the following scenario can not happen. We take the example of $S^3$ with $\omega$ a neighborhood of $\left\{x_4=0,x_3\geq0\right\}$.

 Take some data concentrating on the north pole, with a Fourier transform (on the tangent plane) supported around a direction $\xi_0$. The nonlinear solution will propagate linearly as long as it does not concentrate  : at time $t$ it will be supported in an neighborhood of the point $x(t)$ where $x(t)$ follows the geodesic stemming from the north pole at time $0$ in direction $\xi_0$. Then, if $\xi_0$ is well chosen, it can avoid $\omega$ during that time. Yet, at time $\pi$, the solution will concentrate again in the south pole. According to the description of S. Ibrahim \cite{ibrahim2004gon}, in a short time, the solution will be transformed following the nonlinear scattering operator on $\R^3$. So, at time $\pi+\e$ the solution is close to a linear concentrating wave but it concentrates with a new profile which is obtained with the nonlinear scattering operator on $\R^3$. This operator is strongly nonlinear and we do not know whether the new profile will be supported in Fourier near a new direction $\xi_1$. If it happens, the solution will then be supported near the point $y(t)$ where $y(t)$ follows the geodesic stemming from the south pole at time $\pi$ in direction $\xi_1$. In this situation, it will be possible that the trajectory $y(t)$ still avoids $\omega$. If this phenomenon happens several times, we would have a sequence that concentrates periodically on the north and south pole but always avoiding the region $\omega$ (which in that case satisfies Geometric Control Condition).

\begin{figure}[ht]
$$\ecriture{\includegraphics[width=5cm]{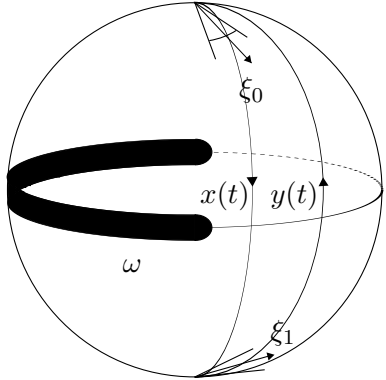}}
{\aat{31}{37}{$\xi_0$}\aat{26}{23}{$x(t)$}\aat{35}{23}{$y(t)$}\aat{35}{5}{$\xi_1$}\aat{16}{14}{$\omega$}}$$ \caption{Possible situation on the sphere}\label{fig.sphere}
\end{figure}
We are led to the following informal question. If $S$ is the scattering operator on $\R^3$, is that possible that for some data $(f,g)\in \dot{H}^1\times L^2$ supported in Fourier near a direction $\xi_0$, the Fourier transform of $S(f,g)$ is supported near another direction $\xi_1$. In other words, can the nonlinear wave operator change the direction of the light?  
\end{remarque}
The structure of the article is as follows. The first section contains some preliminaries that will be used all along the article : the existence theorem for damped nonlinear equation, the description of the main properties of concentrating waves and the useful properties of the scales necessary for the linear decomposition. The second section contains the proof of the profile decomposition of Theorem \ref{thmdecompositionL} and \ref{thmdecompositionNL}. It is naturally divided in two steps corresponding to the linear decomposition and the nonlinear one. We close this section by some useful consequences of the decomposition. The third section contains the proof of the main theorems : the control and stabilization.

Note that the main argument for the proof of stabilization is contained in the last section \ref{sectioncontrol} : in Proposition \ref{proplinear} we apply the linearization argument to get rid of the profiles while Theorem \ref{thminegobserv} contains the proof of the weak observability estimates. We advise the hurried reader to have a first glance at these two proofs in order to understand the global argument.
\subsection{Notation}
\label{defnort}
For $I$ an interval, denote
\bna
\nort{u}_{I}=\nor{u}{L^{\infty}(I,H^1(M))}+\nor{\partial_t u}{L^{\infty}(I,L^2(M))}+\nor{u}{L^5(I,L^{10}(M))}.
\ena
Moreover, when we work in local coordinate, we will need the similar norm (except for $\dot{H}^1$ instead of $H^1$)
\bna
\nort{u}_{I\times \R^3}=\nor{u}{L^{\infty}(I,\dot{H}^1(\R^3))}+\nor{\partial_t u}{L^{\infty}(I,L^2(\R^3))}+\nor{u}{L^5(I,L^{10}(\R^3))}.
\ena
Note that the if $I=\R$, $\nort{u}_{I\times \R^3}$ is invariant by the translation and scaling $u\mapsto \frac{1}{\sqrt{h}} u\left(\frac{t-t_0}{h},\frac{x-x_0}{h}\right)$.

The energy spaces are denoted
\bna
\mathcal{E}&=&H^1(M)\times L^2(M)\\
\mathcal{E}_{x_{\infty}}&=&\dot{H}^1(T_{x_{\infty}}M)\times L^2(T_{x_{\infty}}M)
\ena
with the respective norms 
\bna
\nor{(f,g)}{\HutL}^2&=&\nor{f}{L^2(M)}^2+\nor{\nabla f}{L^2(M)}^2+\nor{g}{L^2(M)}^2\\
\nor{(f,g)}{\HutL_{\infty}}^2&=&\nor{\nabla f}{L^2(T_{x_{\infty}}M)}^2+\nor{g}{L^2(T_{x_{\infty}}M)}^2
\ena
We will denote $\left\langle \cdot,\cdot  \right\rangle_{\HutL}$ and $\left\langle \cdot,\cdot  \right\rangle_{\HutL_{\infty}}$ the associated scalar product.

When dealing with solutions of non linear wave equations on $M$ (or on $T_{x_{\infty}}M$), "the unique strong solution" will mean the unique solution in the Strichartz space $L^5_{loc}(\R,L^{10}(M))$ (or $L^5_{loc}(\R,L^{10}(T_{x_{\infty}}M))$) such that $(u,\partial_t u)\in C(\R,\HutL)$ (or $C(\R,\HutL_{x_{\infty}})$.

All along the article, for a point $x\in M$, we will sometime not distinguish $x$ with its image in a coordinate patch and will write $\R^3$ instead of $T_{x_{\infty}}M$. $M$ will always be smooth, compact and the number of coordinate charts we use is always assumed to be finite. We also assume that all the charts are relatively compact. In all the article, $C$ will denote any constant, possibly depending on the manifold $M$ and the damping function $a$. We will also write $\lesssim$ instead of $\leq C$ for a constant $C$.\\
$B^s_{2,\infty}(M)$ denotes the Besov space on $M$ defined by
\bna
\nor{u}{B^s_{2,\infty}(M)}=\nor{\mathbf{1}_{[0,1[}(\sqrt{-\Delta_M})u}{L^2(M)}+\sup_{k\in \N}\nor{\mathbf{1}_{[2^k,2^{k+1}[}(\sqrt{-\Delta_M})u}{H^s(M)}.
\ena 
We use the same definition for $B^s_{2,\infty}(\R^3)$ with $\Delta_M$ replaced by $\Delta_{\R^3}$ which can be expressed using the Fourier transform and Littlewood-Paley decomposition. Of course, $B^s_{2,\infty}(M)$ is linked with $B^s_{2,\infty}(\R^3)$ by the expression in coordinate charts. This will be precised in Lemma \ref{equivBinft}.

From now on, $a=a(x)$ will always denote a smooth real valued function defined on $M$.
\section{Preliminaries}
\subsection{Existence theorem}
The existence of solutions to our equation is proved using two tools : Strichartz and Morawetz estimates. Strichartz estimates take the following form.  
\begin{prop}[Strichartz and energy estimates]
Let $T>0$ and $(p,q)$ satisfying 
\bna
\frac{1}{p}+\frac{3}{q}=\frac{1}{2}, \quad p>2
\ena
Then, there exists $C>0$ such that any solution $u$ of 
 \bna
 \left\lbrace
\begin{array}{rcl}
\Box v+v+a(x)\partial_tv&=& f\quad \textnormal{on}\quad [-T,T]\times M\\
(v(0),\partial_t u(0))&=&(u_0,u_1) .
\end{array}
\right.
 \ena
satisfies the estimate 
\bna
\nor{(v,\partial_tv)}{L^{\infty}([-T,T],\HutL)}+\nor{v}{L^p([-T,T],L^q(M))}\leq C(\nor{(u_0,u_1)}{\HutL}+\nor{f}{L^1([-T,T],L^2)}).
\ena
\end{prop}
\bnp
The case with $a\equiv 0$ for the wave equation can be found in L.V. Kapitanski \cite{KapitanskiStrichartzvar}. To treat the case of damped Klein-Gordon, we only have to absorb the additional terms and get the desired estimate for $T$ small enough. We can then reiterate the operation to get the result for large times.
\enp
Then, we are going to prove global existence for the equation 
\begin{eqnarray}
\label{eqnexist}
\left\lbrace
\begin{array}{rcl}
\Box u+u+|u|^4u&=&a(x)\partial_tu+g \quad \textnormal{on}\quad [-T,T]\times M\\
(u(0),\partial_t u(0))&=&(u_{0},u_1) \in \HutL
\end{array}
\right.
\end{eqnarray}
with $g\in L^1([-T,T],L^2(M))$ and $a\in C^{\infty}(M)$.

The proof is now very classical, see for example \cite{NLWzuilybourbaki} for a survey of the subject. The critical defocusing nonlinear wave equation on $\R^3$ was proved to be globally well posed by Shatah and Struwe \cite{NLWShatahStruweIMRN,NLWShatahStruweAnnals} using Morawetz estimates. Later, S. Ibrahim an M. Majdoub managed to apply this strategy in the case of variable coefficients in \cite{ibrahim2003existNLW}, but without damping and forcing term. In this subsection, we extend this strategy to the case with these additional terms. We also refer to the appendix of \cite{BahouriGerard} where the computation of Morawetz estimates on $\R^3$ is made with a forcing term. We also mention the result of N. Burq, G. Lebeau and F. Planchon \cite{BLPNLWcritdomain} in the case of 3-D domains.

We only have to check that the two additional terms do not create any trouble. Actually, the main difference is that the energy in the light cones is not decreasing, but it is locally "almost decreasing" (see formula (\ref{presquedecr})) and this will be enough to conclude with the same type of arguments.

As usual in critial problems, the local problem is well understood thanks to Strichartz estimates while we have to prove global existence. We only consider Shatah Struwe solutions, that is satisfying Strichartz estimates and we wave uniqueness for local solutions in this class. We assume that there is a maximal time of existence $t_0$ and we want to prove that it is infinite . 
The solution considered will be limit of smooth solutions of the nonlinear equation with smoothed initial data and nonlinearity. Therefore, the integrations by part are licit by a limiting argument.

We need some notations. To simplify the notations, the space-time point where we want to extend the solution will be $z_0=(t_0,x_0)=(0,0)$. $\varphi$ is the geodesic distance on $M$ to $x_0=0$ defined in a neighborhood $U$ of $0$. Denote for some small $\alpha<\beta<0$
\bna 
K_{\alpha}^{\beta}:=&\left\{z=(t,x)\in [\alpha,\beta] \times U \left|\varphi \leq |t| \right.\right\}& \textnormal{ backward truncated cone}\\
M_{\alpha}^{\beta}:=&\left\{z=(t,x)\in [\alpha,\beta] \times U \left|\varphi = |t| \right.\right\}& \textnormal{ mantle of the truncated cone }\\
D(t):=&\left\{x\in U \left|\varphi \leq |t| \right.\right\}& \textnormal{ spacelike section of the cone at time }t
\ena
In what follows, the gradient, norm, density are computed with respect to the Riemannian metric on $M$ (for example, we have $\left|\nabla \varphi\right|=1$). We also define
\bna
e(u)(t,x):=&\frac{1}{2}\left(\left|\partial_t u \right|^2+ \left|\nabla u \right|^2\right) +\frac{1}{6}\left|u\right|^6& \textnormal{local energy}\\
E(u,D(t)):=&\int_{D(t)}e(u)(t,x)dx& \textnormal{energy at time }t \textnormal{ in the section of the cone}\\
Flux(u,M_{\alpha}^{\beta})):= &\frac{1}{\sqrt{2}}\int_{M_{\alpha}^{\beta}}\frac{1}{2}\left|\partial_tu\nabla \varphi -\nabla u\right|^2+ \frac{1}{6}\left|u\right|^6~d\sigma & \textnormal{flux getting out of the truncated cone}
\ena
\begin{lemme}
\label{lmlimitNRJ}
Let $u$ be a solution  of equation (\ref{eqnexist}). The function $E(u,D(t))$ satisfies for $\alpha< \beta<0$
\bna
 E(u,D(\beta))+Flux (u,M_{\alpha}^{\beta})=E(u,D(\alpha))+ \iint_{K_{\alpha}^{\beta}}a(x) |\partial_tu|^2-\Re\iint_{K_{\alpha}^{\beta}}u\partial_t\bar{u}+\Re\iint_{K_{\alpha}^{\beta}}g\partial_t\bar{u}
\ena
and it has a left limit in $t=0$ as a function of $t$.
\end{lemme}
\bnp
The identity is obtained by multiplying the equation by $\partial_t\overline{u}$ to get $\partial_t e(u)-\Re div (\partial_t\overline{u} \nabla_x u)=a(x)|\partial_t u|^2 -u\partial_t\bar{u}+ \Re g\partial_t\overline{u}$, then, we integrate over the truncated cone $K_{\alpha}^{\beta}$ and use Stokes formula. Denote $f(t)=E(u,D(t))$. Using the positivity of the flux and H\"older inequality, we estimate 
\bna
\nor{f}{L^{\infty}([\alpha,\beta])}\leq f(\alpha)+ C(\beta-\alpha)\nor{f}{L^{\infty}([\alpha,\beta])} + C|\alpha|(\beta-\alpha)\nor{f}{L^{\infty}([\alpha,\beta])}^{2/3}+\nor{g}{L^1([\alpha,\beta],L^2)}\nor{f}{L^{\infty}([\alpha,\beta])}^{1/2}.
\ena
Using $C|\alpha|(\beta-\alpha)\nor{f}{L^{\infty}([\alpha,\beta])}^{2/3}\leq C(\beta-\alpha)(\nor{f}{L^{\infty}([\alpha,\beta])}^{1/2}+\nor{f}{L^{\infty}([\alpha,\beta])})$, we get for $\beta-\alpha$ small enough,
\bnan
\label{presquedecr}
f(\beta)^{1/2}\leq \frac{1}{1-2C(\beta-\alpha)}\left[f(\alpha)^{1/2}+C(\beta-\alpha)+\nor{g}{L^1([\alpha,\beta],L^2)}\right]
\enan
This property will replace the decreasing of the energy that occurs without damping and forcing term in all the rest of the proof. It easily implies that $f$ has a left limit.
\enp
\begin{lemme}
\label{lmMorawetz}
For $u$ and $g$ a strong solution of
\bna
\Box u+|u|^4u=g \quad \textnormal{on}\quad [-T,0[\times M
\ena
we have the estimate
\bna
\int_{D(\alpha)}|u|^6& \leq &C \left(\frac{\beta}{\alpha} \left(f(\beta))+f(\beta)^{1/3}\right)+\left|f(\beta)-f(\alpha)\right|+\nor{g}{L^1L^2(K_{\alpha}^{\beta})}\nor{\partial_t u}{L^{\infty}L^2(K_{\alpha}^{\beta})}\right.\\
&&+\left(\left|f(\beta)-f(\alpha)\right|+\nor{g}{L^1L^2(K_{\alpha}^{\beta})}\nor{\partial_t u}{L^{\infty}L^2(K_{\alpha}^{\beta})}\right)^{1/3}\\
&&+\nor{g}{L^1L^2(K_{\alpha}^{\beta})}\left(\nor{\partial_t u}{L^{\infty}L^2(K_{\alpha}^{\beta})}+\nor{\nabla u}{L^{\infty}L^2(K_{\alpha}^{\beta})}+\nor{u}{L^{\infty}L^6(K_{\alpha}^{\beta})}\right)\\
&&\left.+(\beta-\alpha) \underset{t\in [\alpha,\beta]}{\sup} \left[f(t)+f(t)^{1/3}\right]\right)
\ena 
where we have used the notation $f(t)=E(u,D(t))$.
\end{lemme}
\bnp
It is a consequence of Morawetz estimates. The only difference is the presence of the forcing term $g$ and the metric. The case of flat metric is treated in \cite{BahouriGerard}. The metric leads to the same estimates with an additional term $(\beta-\alpha) \underset{t\in [\alpha,\beta]}{\sup} f(t)+f(t)^{1/3}$ as treated in \cite{ibrahim2003existNLW}. Another minor difference is that in the presence of a forcing term, the energy does not decrease and $f(\beta)+f(\beta)^{1/3}$ have to be replaced by the supremum on the interval. Note also that our estimate is made in the backward cone while the computation is made in the future cone in these references. We leave the easy modifications to the reader.   
\enp
The previous estimates will be the main tools of the proof. It will be enough to prove some non concentration property in the light cone for $L^{\infty}L^6$, $L^5L^{10}$ and finally in energy space. It is the object of the following three corollaries. 
\begin{corollaire}
\label{corL6tendzero}
\bna
\int_{D(\alpha)}|u(\alpha,x)|^6dx \tend{\alpha}{0}0.
\ena
\end{corollaire}
\bnp
We are going to use the previous Lemma \ref{lmMorawetz}, replacing $g$ by $g-u+a(x)\partial_t u$ and with $\beta=\e \alpha$, $0<\e<1$. Denote $L$ the limit of $f(t)$ as $t$ tends to $0$ given by Lemma \ref{lmlimitNRJ}. So for $\alpha$ small enough, we have for a constant $C>0$
\bna
\nor{\partial_tu}{L^{\infty}L^2(K_{\alpha}^{\beta})}+\nor{\nabla u}{L^{\infty}L^2(K_{\alpha}^{\beta})}+\nor{u}{L^{\infty}L^6(K_{\alpha}^{\beta})}\leq 1+C(L^{1/2}+L^{1/6}).
\ena

We also use
\bna 
\nor{g-u+a(x)\partial_tu }{L^1L^2(K_{\alpha}^{\beta})}&\leq &\nor{g}{L^1L^2(K_{\alpha}^{\beta})}+C(\beta-\alpha)\nor{u}{L^{\infty}L^2(K_{\alpha}^{\beta})}+C(\beta-\alpha)\nor{\partial_tu }{L^{\infty}L^2(K_{\alpha}^{\beta})}\\
&\leq& \nor{g}{L^1L^2(K_{\alpha}^{\beta})}+C(\beta-\alpha)(1+L^{1/6}+L^{1/2})
\ena
which tends to $0$ as $\beta$ tends to $0$.
This yields
\bna
\limsu{\alpha}{0} \int_{D(\alpha)}|u(\alpha,x)|^6dx \leq C\e(L+L^{1/3}).
\ena
\enp
\begin{corollaire}
\label{corexistL5L10}
\bna
u\in L^5 L^{10}(K_{-T}^{0}).
\ena
\end{corollaire}
\bnp
Localized Strichartz estimates in cones (see Proposition 4.4 of \cite{ibrahim2003existNLW}) give 
\bna
\nor{u}{L^4 L^{12}(K_s^{0})}&\leq& C E(u,D(s))^{1/2} +\nor{u}{L^5 L^{10}(K_s^{0})}^5+\nor{a(x)\partial_t u-u+g}{L^1 L^{2}(K_s^{0})}\\
&\leq &C E(u,D(s))^{1/2} +\nor{u}{L^{\infty} L^{6}(K_s^{0})}(1+\nor{u}{L^4 L^{12}(K_s^{0})}^4)+\nor{\partial_t u}{L^{\infty} L^{2}(K_s^{0})}+\nor{g}{L^1 L^{2}(K_s^{0})}.
\ena
A boot-strap argument and Corollary \ref{corL6tendzero} give that for $s$ sufficiently close to $0$, $\nor{u}{L^4 L^{12}(K_s^{0})}$ is bounded. We get the announced result by interpolation between $L^4 L^{12}$ and $L^{\infty} L^{6}$.
\enp
\begin{corollaire}
\label{corNRJpasconcentr}
\bna
E(u,D(s)) \tend{s}{0}0.
\ena
\end{corollaire}
\bnp
Let $\e>0$. Corollary \ref{corexistL5L10} allows to fix $s<0$ close to $0$ so that $\nor{u}{L^5 L^{10}(K_s^{0})}\leq \e$. Denote $v_s$ the solution to the linear equation
\bna
\Box v_s+v_s+a(x)\partial_tv_s=0, \quad (v_s,\partial_t v)_{t=s}=(u,\partial_t u)_{t=s}
\ena
then, the difference $w_s=u-v_s$ is solution of
\bna
\Box w_s+w_s+a(x)\partial_tw_s=-|u|^4u, \quad (w_s,\partial_t w_s)_{t=s}=(0,0).
\ena
Then, for $s<t<0$, linear energy estimates give
\bna
E_0(w_s,D(t))^{1/2}\leq  C \nor{u}{L^5 L^{10}(K_s^{0})}^5\leq C\e^5
\ena
where we have set
\bna
E_0(w_s,D(t))=\frac{1}{2}\int_{D(t,z_0)}\left[|\nabla w_s|^2+|\partial_t w_s|^2\right]dx.
\ena
Triangular inequality yields
\bna
E_0(u,D(t))^{1/2}\leq  E_0(v_s,D(t))^{1/2}+C \e^5.
\ena
Since $v_s$ is solution of the free damped linear equation, we have $E_0(v_s,D(t))\tend{t}{0}0$. This yields the result with $E_0$ instead of $E$. The final result is obtained thanks to Corollary \ref{corL6tendzero}.
\enp
We can now finish the proof of global existence.

Let $\e>0$ to be chosen later. By Corollary \ref{corNRJpasconcentr}, $E(u,D(s))\leq \e$ for $s$ close enough to $0$. By dominated convergence, for any $s<0$ close to $0$, there exists $\eta>0$ so that
\bna
\int_{\varphi(x)\leq t_0-s+\eta} e(u)(s)=E(u,D(s,\eta))\leq 2\e
\ena
where $E(u,D(s,\eta))$ is the spacelike energy at time $s$ of the cone centered at $(t_0=\eta,x_0=0)$ (see Figure \ref{fig.cones}).
For $s$ close enough to $0$ and $s<s'<0$, we apply estimate (\ref{presquedecr}) in this cone. It gives 
\bna
E(u,D(s',\eta))^{1/2}\leq C \left(E(u,D(s,\eta))^{1/2}+|s'-s|+\nor{g}{L^1([s,s'],L^2)}\right)\leq C\e^{1/2}.
\ena
 
In particular, $\nor{u}{L^{\infty} L^{6}(K)}\leq C\e^{1/2}$ on the truncated cone
\bna
K=\left\{(s',x)\left|\varphi(x)\leq \eta-s', s<s'<0 \right.\right\}
\ena
Therefore, choosing $\e$ small enough to apply the same proof as Corollary \ref{corexistL5L10}, we get
\bna
\nor{u}{L^{5} L^{10}(K)}<+\infty.
\ena
Since $x_0=0$ is arbitrary, a compactness argument yields one $s<0$ such that $\nor{u}{L^{5}([s,0[ L^{10}(M))}<+\infty$. Therefore, by Duhamel formula, $(u(t),\partial_t u(t))$ has a limit in $\HutL$ as $t$ tends to $0$ and $u$ can be extended for some small $t>0$ using local existence theory.
\begin{figure}[ht]
$$\ecriture{\includegraphics[width=7cm]{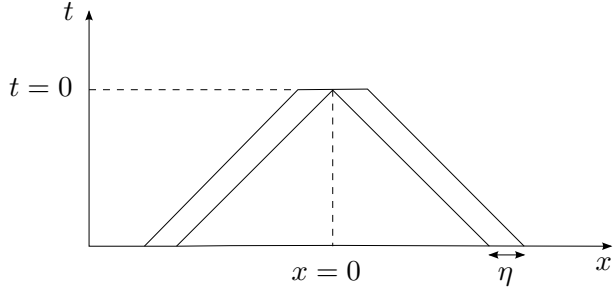}}
{\aat{20}{-2}{$x=0$}\aat{-1}{22}{$t$}\aat{48}{-1}{$x$}\aat{39}{-2}{$\eta$}\aat{-6}{15}{$t=0$}}$$ \caption{The truncated cone $K$}\label{fig.cones}
\end{figure}
\begin{remarque}It is likely that global existence can also be proved using the Kenig-Merle argument \cite{KenigMerleNLW} and the profile decomposition below (assuming only local existence) as is done for example in \cite{KriegSchlagwavemap} for the wave maps.
\end{remarque}
\subsection{Concentration waves}
In this section, we give details about concentrating waves that will be useful in the profile decomposition. The first lemma states that Definition \ref{defconcdata} of concentrating data does not depend on the choice of coordinate patch and cut-off function $\Psi_U$.
\begin{lemme}
\label{lemmedefconc}
Let $[(f,g),\underline{h},\underline{x}] \in \mathcal{E}\times (\R_+^* \times M \times)^{\N}$ such that $lim_n (h_n,x_n)=(0,x_{\infty})$ then, all the sequences defined by formula (\ref{formconcentrdata}) in different coordinates charts and the cut-off function $\Psi_U$ are equivalent, modulo convergence in $\HutL$.
\end{lemme}
\begin{proof}
It is very close to the one of S. Ibrahim \cite{ibrahim2004gon} where the concentrating data are given in geodesic coordinates.  So, let $V_M\approx V$ be another coordinate patch and $\Phi : V \mapsto U$ the associated transition map. Without loss of generality, we can suppose that $x_{\infty}$ is represented by $0$ in $U$ and $V$. We have to prove that the sequences 
$$h_n^{-\frac{1}{2}}\Phi^*\Psi_U(x) (f,\frac{1}{h_n}g)\left(\frac{x-\Phi(x_n)}{h_n}\right)=h_n^{-\frac{1}{2}}\Psi_U(\Phi(x)) (f,\frac{1}{h_n}g)\left(\frac{\Phi(x)-\Phi(x_n)}{h_n}\right)$$ and 
$$h_n^{-1/2}\Psi_V(x) (f\circ D\Phi(0),\frac{1}{h_n}g\circ D\Phi(0))\left(\frac{x-x_n}{h_n}\right)$$ 
are equivalent in the energy space associated to $M$ or $\R^3$ (the volume form and the gradient are not the same but the energies are equivalent). By approximation, we can assume $(f,g)\in (C^{\infty}_0(\R^3))^2$. We make the proof for the $\Hup$ part for $f$, the proof being simpler for $g$.  We remark that the terms coming from derivatives hitting on $\Psi_U(x)$ tend to $0$ in $L^2$. Therefore, we have to prove the convergence to $0$ of
\bna
h_n^{-3}\left\|\Psi_U(\Phi(x))D\Phi(x) \nabla f\left(\frac{\Phi(x)-\Phi(x_n)}{h_n}\right)- \Psi_V(x) D\Phi(0)\nabla f \left(\frac{D\Phi(0)x-D\Phi(0)x_n}{h_n}\right)\right\|^2_{L^2(V)}.
\ena
First, we prove that the cut-off functions $\Psi_U$ and $\Psi_V$ can be replaced by a unique $\Psi$. Let $\delta$ so that $B(0,\delta)\subset V$. Let $\Psi \in C^{\infty}_0(B(0,\delta))$ such that $\Psi \equiv 1$ in a neighborhood of $0$ and has a support included in the set of $x$ such that $\Psi_V(x)=\Psi_U(\Phi(x))=1$, so that $\Psi \Psi_V=\Psi$ and $\Psi (\Psi_U \circ \Phi)=\Psi$.  Then, on the support of $1-\Psi$, we have $\left\|\Phi(x)-\Phi(x_n)\right\|> \varepsilon$ for some $\varepsilon>0$ and some $n$ large enough. Therefore, we have 
\bna
&&h_n^{-3}\left\|(1-\Psi(x))\Psi_U(\Phi(x))D\Phi(x)\nabla f \left(\frac{\Phi(x)-\Phi(x_n)}{h_n}\right)\right\|^2_{L^2(V)}\\ &\leq &C h_n^{-3}\left\|\nabla f \left(\frac{\Phi(x)-\Phi(x_n)}{h_n}\right)\right\|^2_{L^2(\left\|\Phi(x)-\Phi(x_n)\right\|>\varepsilon)}
\ena
which is $0$ for $n$ large enough since $f$ has compact support. Making the same proof for the other term, we are led to prove the convergence to $0$ of
\bnan
&&h_n^{-3}\left\|\Psi(x)D\Phi(x) \nabla f\left(\frac{\Phi(x)-\Phi(x_n)}{h_n}\right)- \Psi(x) D\Phi(0)\nabla f \left(\frac{D\Phi(0)x-D\Phi(0)x_n}{h_n}\right)\right\|^2_{L^2(B(0,\delta))}\nonumber \\
\label{norm1} &\leq &\left\|D\Phi(h_nx+x_n)\nabla f\left(\frac{\Phi(h_nx+x_n)-\Phi(x_n)}{h_n}\right)- D\Phi(0)\nabla f \left(D\Phi(0)x\right)\right\|^2_{L^2(\left\{x:|x_n+h_n x|\leq \delta \right\}}.
\enan
By the fundamental theorem of calculus, there exists $z_n(x) \in [x_n,h_nx+x_n]$ such that\\ $\left|\frac{\Phi(h_nx+x_n)-\Phi(x_n)}{h_n}\right|=\left|D\Phi(z_n)x\right|>C\left|x\right|$ for some uniform $C>0$. As $\nabla f$ is compactly supported, we deduce that for $\left|x\right|$ large enough, the integral is zero. So, we are led with the norm (\ref{norm1}) with $L^2(B(0,C))$ instead of $L^2(\left\{x:|x_n+h_n x|\leq \delta \right\})$. We conclude by dominated convergence.
\end{proof}
Using the previous lemma in geodesic coordinates, we get that our definition of concentrating data is the same as Definition 1.2 of S. Ibrahim \cite{ibrahim2004gon}. 

Remark that for a concentrating data, $x_n-x_{\infty}$ can not be defined invariantly on $T_{x_{\infty}}M$, we can only define the limit of $(x_n-x_{\infty})/h_n$. The change of coordinates must act on $x_n$ as an element of $M$ and not $T_{x_{\infty}}M$ even if it converges to $x_{\infty}$. Yet, the functions $(f,g)$ of a concentrating data "live" on the tangent space. Moreover, the norm in energy of a concentrating data is the one of its data.
\begin{lemme}
\label{lmnormconcwave}
Let $(u_n,v_n)$ a concentrating data associated to $[(\varphi,\psi),\underline{h},\underline{x}]$, then, we have
\bna
\nor{(u_n,v_n)}{\HutL}=\nor{(\varphi,\psi)}{\HutL_{x_{\infty}}}+\petito{1}
\ena
where $\nabla_{x_{\infty}}$ and $L^2(T_{x_\infty}M)$ are computed with respect to the frozen metric.
\end{lemme}
The proof is a direct consequence of Lemma \ref{defweaktrackequiv} and \ref{lmconcweak} below or by a direct computation in coordinates.

The next definition is the tool that will be used to "track" the concentrations. 
\begin{definition}
\label{defweaktrack}
Let $x_{\infty} \in M$ and $(f,g)\in \HutL_{x_{\infty}}$. Given $[(f,g),\underline{h},\underline{x}] \in \HutL_{x_{\infty}}\times (\R_+^* \times M )^{\N}$ such that $lim_n (h_n,x_n)=(0,x_{\infty})$. Let $(f_n,g_n)$ be a sequence bounded in $\HutL$, we set
\bna
D_{h_n} (f_n,g_n)\rightharpoonup (f,g)
\ena
if in some coordinate patch $U_M\approx U\subset \R^d$ containing $x_{\infty}$ and for some $\Psi_U\in C^{\infty}_0(U)$ such that $\Psi_U(x)=1$ in a neighborhood of $x_{\infty}$, we have
\bna
h_n^{\frac{1}{2}} (\Psi_Uf_n,h_n\Psi_Ug_n)\left(x_n+h_nx\right) \rightharpoonup (f,g)\quad \textnormal{ weakly in }\HutL_{x_{\infty}}
\ena
where we have identified $\Psi_U(f_n,g_n)$ with its representation on $T_{x_\infty}M$ in the local trivialisation.

If this holds for one $(U,\Psi_U)$, it holds for any other coordinate chart with the induced transition map.

We denote $D_{h_n}^1 f_n\rightharpoonup f$ if we only consider the first part concerning $\dot{H}^1$ and $D_{h_n}^2 g_n\rightharpoonup g$ for the $L^2$ part convergence.
\end{definition}
Of course, this definition depends on the core of concentration $\underline{h}$ and $\underline{x}$. In the rest of the paper, the rate $\underline{h}$ and $\underline{x}$ will always be implicit. When several rate of concentration $[\underline{h}^{(j)},\underline{x}^{(j)}]$, $j\in \N$, are used in a proof, we use the notation $D_h^{(j)}$ to distinguish them.
  
The fact that this definition is independent of the choice of a coordinate chart can be seen with the following lemma which will also be useful afterward.
\begin{lemme}
\label{defweaktrackequiv}
$D_{h_n} (f_n,g_n)\rightharpoonup (f,g)$ is equivalent to 
\bna
\int_M \nabla_M f_n\cdot \nabla_M u_n&\tend{n}{\infty} &\int_{T_{x_{\infty}}M} \nabla_{x_{\infty}} f \cdot \nabla_{x_{\infty}} \varphi\\
\int_M g_n  v_n&\tend{n}{\infty} &\int_{T_{x_{\infty}}M} g \psi
\ena
where $(u_n,v_n)$ is any concentrating data associated with $[(\varphi,\psi),\underline{h},\underline{x}]$.
\end{lemme}
The $\nabla$ is computed with respect to the metric on $M$ when the integral is over $M$ and with respect to the frozen metric in $x_{\infty}$ when the integral is over $T_{x_{\infty}}M$.
\begin{proof}
We only compute the first term for the $H^1$ norm and assume $\varphi \in C_0^{\infty}(\R^3)$. $d\omega(y)$ denotes the Riemannian volume form at the point $y$, $\cdot_y$ the scalar product at the point $y$ and $\nabla_{h_nx+x_n}= g(h_nx+x_n)^{-1} \nabla $.

We denote $V_h=\frac{V-x_n}{h}$ and $L_{n,V}=h_n^{\frac{1}{2}}\int_{V_h}\nabla_{x_{\infty}} \left[\Psi_Vf_n\left(x_n+h_nx\right)\right]\cdot \nabla_{x_{\infty}} \varphi(x)~d\omega(0)$.
\bna
L_{n,V}&=&h_n^{\frac{1}{2}}\int_{V_h}\nabla_{x_n+h_nx} \left[\Psi_Vf_n\left(x_n+h_nx\right)\right]\cdot_{(x_n+h_nx)} \nabla_{x_n+h_nx} \varphi(x)~d\omega(x_n+h_nx) +\petito{1}\\
&=&h_n^{\frac{3}{2}}\int_{V_h}\Psi_V(x_n+h_nx)(\nabla_{x_n+h_nx} f_n)\left(x_n+h_nx\right)\cdot_{(x_n+h_nx)} \nabla_{x_n+h_nx} \varphi(x)~d\omega(x_n+h_nx) +\petito{1}\\
&=&h_n^{-\frac{3}{2}}\int_{V}\nabla_{y} f_n\left(y\right)\cdot_{y}  \Psi_V(y)\nabla_{y}\varphi\left(\frac{y-x_n}{h_n}\right)~d\omega(y) +\petito{1}\\
&=&h_n^{-\frac{1}{2}}\int_{V}\nabla_{y} f_n\left(y\right)\cdot_{y} \nabla_{y} \left[\Psi_V(y)\varphi\left(\frac{y-x_n}{h_n}\right)\right]~d\omega(y) +\petito{1}=\int_{M}\nabla_{M} f_n\cdot \nabla_{M} u_n~ +\petito{1}.
\ena
Therefore, $L_{n,V}$ tends to $\int \nabla f(x)\cdot \nabla \varphi(x)~d\omega(0)$ if and only if $\int_{M}\nabla_{M} f_n\cdot \nabla_{M} u_n$ has the same limit. 
\end{proof}
An easy consequence of this lemma is the link with concentrating waves. 
\begin{lemme}
\label{lmconcweak}
Let $(f_n,g_n)$ be some concentrating data associated with $[(f,g),\underline{h},\underline{x}]$, then, we have 
\bna
D_{h_n} (f_n,g_n)\rightharpoonup (f,g)
\ena
\end{lemme}
\bnp Lemma \ref{lemmedefconc} permits to work in geodesic coordinates so that the metric $g$ is the identity at the point $x_{\infty}$. In this chart, we have $f_n\left(x_n+h_nx\right)=\Psi_U(x_n+h_nx)h_n^{-\frac{1}{2}}f$.  So, the computation of Lemma \ref{defweaktrackequiv} gives $\int \nabla_{\infty} f\cdot \nabla_{\infty}\varphi d\omega(0)=\int_{M}\nabla_{M} f_n\cdot \nabla_{M} u_n~ +\petito{1}$ which gives the result.
\enp
We conclude this subsection by a definition of orthogonality that will discriminate concentrating data.
\begin{definition}
\label{deforthog}
We say that two sequences $[\underline{h}^{(1)},\underline{x}^{(1)},\underline{t}^{(1)}] $ and $[\underline{h}^{(2)},\underline{x}^{(2)},\underline{t}^{(2)}] $ are orthogonal if either 
\begin{itemize}
\item $\log \left|\frac{h_n^{(1)}}{h_n^{(2)}}\right| \tend{n}{\infty} +\infty $
\item $x_{\infty}^{(1)} \neq x_{\infty}^{(2)}$
\item $h_n^{(1)}=h_n^{(1)}=h$ and $x_{\infty}^{(1)} = x_{\infty}^{(2)}=x_{\infty}$ and in some coordinate chart around $x_{\infty}$, we have
\bna
\frac{\left|t_h^{(1)}-t_h^{(2)}\right|}{h}+\frac{\left|x_h^{(1)}-x_h^{(2)}\right|}{h} \tend{h}{0}+\infty &
\ena
\end{itemize}

We note $[\underline{h}^{(1)},\underline{x}^{(1)},\underline{t}^{(1)}] \perp [\underline{h}^{(2)},\underline{x}^{(2)},\underline{t}^{(2)}]$ and $(\underline{x}^{(1)},\underline{t}^{(1)}) \perp_h (\underline{x}^{(2)},\underline{t}^{(2)})$ if $\underline{h}^{(1)}=\underline{h}^{(2)}=h$.
\end{definition}
This definition does not depend on the coordinate chart. This can be seen because we have the estimate $\frac{1}{C}\left|x_h^{(1)}-x_h^{(2)}\right|\leq \left|\Phi(x_h^{(1)})-\Phi(x_h^{(2)})\right|\leq C\left|x_h^{(1)}-x_h^{(2)}\right|$ if $\Phi$ is the transition map.
\subsection{Scales}
In this subsection, we precise a few facts that will be useful in the first part of the proof of linear profile decomposition which consists of the extration of the scales of oscillation $h_n^j$.

On the Hilbert space $\HutL=\Hu(M) \times L^2(M)$, we define the self-adjoint operator $A_M$ by : 
\bna
&&D(A_M)=H^2_M\times H^1_M\\
&&A_M(u,v)=((-\Delta_M)^{1/2}v,(-\Delta_M)^{1/2} u)
\ena
We define similarly $A_{\R^d}$ with the flat laplacian. We denote $A_{\R^d,N}$ the obvious operator on $(\Hu(\R^d) \times L^2(\R^d))^N$ obtained by applying $A_{\R^d}$ on each "coordinate". 

The following definition is taken from Gallagher-Gérard\cite{PGGall2001}.
\begin{definition}
Let $A$ be a selfadjoint (unbounded) operator on a Hilbert space $H$. Let $(h_n)$ a sequence of positive numbers converging to $0$. A bounded sequence $(u_n)$ in $H$ is said $(h_n)$-oscillatory with respect to $A$ if
\bnan
\label{oscill}
\limsu{n}{\infty} \nor{1_{|A|\geq \frac{R}{h_n}}u_n}{H}\tend{R}{\infty} 0.
\enan
$(u_n)$ is said stricly $(h_n)$-oscillatory with respect to $A$ if it satisfies (\ref{oscill}) and 
\bna
\limsu{n}{\infty} \nor{1_{|A|\leq \frac{\e}{h_n}}u_n}{H}\tend{\e}{0} 0.
\ena
At the contrary, $(u_n)$ is said $(h_n)$-singular with respect to $A$ if we have 
\bna
\nor{1_{\frac{a}{h_n}|A|\leq \frac{b}{h_n}}u_n}{H}\tend{n}{\infty} 0 \quad \textnormal{ for all } 0<a<b.
\ena
\end{definition}
Remark that $1_{|x|\leq 1}$ can easily be replaced by a well chosen function $\varphi \in C_0^{\infty}(\R)$. Moreover, if a sequence $(u_n)$ is stricly $(h_n)$-oscillatory while a second sequence $(v_n)$ is $(h_n)$-singular, then we have the interesting property that $\left\langle u_n,v_n\right\rangle_{H}\tend{n}{\infty}0$. 
\begin{prop}
\label{proposcill}
Let $M=\cup_{i=1}^N U_i$ a finite covering of $M$ with some associated local coordinate patch $\Phi_i : U_i \rightarrow V_i \subset \R^3$.
Let $1=\sum_i \Psi_i$ be an associated partition of the unity of $M$ with $\Psi_i \in C^{\infty}_0(U_i)$. Let $(u_n,v_n)$ a bounded sequence in the $M$ energy space and $h_n$ a sequence converging to $0$. Then $(u_n,v_n)$ is (stricly) $(h_n)$-oscillatory with respect to $A_M$, if and only if all the $\Phi_{i*}\Psi_i (u_n,v_n)$ are (strictly) $(h_n)$-oscillatory with respect to $A_{\R^d}$.
\end{prop}
\begin{proof}First, we remark that a sequence is (strictly) $(h_n)$-oscillatory with respect to $A$ if and only if it is (strictly) $(h_n^2)$-oscillatory with respect to $A^2$. So we can replace $A_M$ and $A_{\R^3}$ by $-(\Delta_M,\Delta_M)$ and $-(\Delta_{\R^3},\Delta_{\R^3})$. We apply a proposition taken from \cite{PGGall2001} that makes the link between oscillation with different operators. 
\begin{prop} [Proposition 2.2.3 of \cite{PGGall2001}]
\label{propPGGalloscill}
Let $\Lambda :H_1 \rightarrow H_2$ be a continuous linear map between Hilbert spaces $H_1$, $H_2$. Let $A_1$ be a selfadjoint operator on $H_1$, $A_2$ be a selfadjoint operator on $H_2$. Assume there exists $C>0$ such that $\Lambda(D(A_1))\subset D(A_2)$, $\Lambda^*(D(A_2))\subset D(A_1)$ and for any $u\in D(A_1)$, $v\in D(A_2)$,
\bnan
\label{inegscale1}\nor{A_2 \Lambda u}{}\leq C(\nor{A_1 u}{}+\nor{u}{})\\
\label{inegscale2}\nor{A_1 \Lambda^* v}{}\leq C(\nor{A_2 v}{}+\nor{v}{}).
\enan
If a bounded sequence $(u_n)$ in $H_1$ is (strictly) ($h_n$)-oscillatory with respect to $A_1$, then $(\Lambda u_n)$ is (strictly) ($h_n$)-oscillatory with respect to $A_2$.
\end{prop}
To prove the first implication, we apply the proposition with\\ $\Lambda (u,v)= (\Phi_{1*}\Psi_1 (u,v),\cdots, \Phi_{N*}\Psi_N (u,v))$. We only prove the necessary estimates, the inclusions of domains being a direct consequence of the inequalities and of the density of smooth functions. To simplify the notation, we denote $(u_i,v_i)=\Phi_{i*}\Psi_i (u,v)$. The proof of (\ref{inegscale1}) mainly uses the equivalent definitions of the $H^s$ norm on a manifold. 
\bna
\nor{A_{\R^3}^2(u_i,v_i)}{H^1_{\R^3}\times L^2_{\R^3}}&=&\nor{\Delta_{\R^3} u_i}{\Hu_{\R^3}}+\nor{\Delta_{\R^3} v_i}{L^2_{\R^3}} \lesssim\nor{u_i}{H^3_{\R^3}}+\nor{v_i}{H^2_{\R^3}}\lesssim \nor{u}{H^3_{M}}+\nor{v}{H^2_{M}}\\
&\lesssim & \nor{u}{H^1_{M}}+\nor{\Delta_M u }{H^1_M }+\nor{v}{L^2_M}+\nor{\Delta_M v}{L^2_M}\\
&\lesssim & \nor{A_{M}^2(u,v)}{H^1_M\times L^2_M}+\nor{(u,v)}{H^1_M\times L^2_M}.
\ena 
Let us prove (\ref{inegscale2}) for the duality $H^1\times L^2$ of the scalar product. Let $(f,g)=(f_i,g_i)_{i=1\cdots N}\in (C_0^{\infty}(\R^3)\times C_0^{\infty}(\R^3))^N$ and $(u,v)\in C^{\infty}(M)$. 
\bna
&&( (u,v),A_M^2\Lambda^*(f,g))_{H^1(M)\times L^2(M)}=(\Lambda A_M^2(u,v),(f,g))_{(H^1(\R^3)\times L^2(\R^3))^N}\\
&=&\sum_i (\Phi_{i*}\Psi_i \Delta_M u,f_i)_{H^1_{\R^3}}+\sum_i(\Phi_{i*}\Psi_i \Delta_M v,g_i)_{L^2}\\
&\lesssim &\sum_i \nor{\Phi_{i*}\Psi_i \Delta_M u}{H^{-1}_{\R^3}}\nor{f_i}{H^3_{\R^3}} +\sum_i \nor{\Phi_{i*}\Psi_i \Delta_M v}{H^{-2}_{\R^3}}\nor{g_i}{H^2_{\R^3}} \\
&\lesssim & \nor{u}{H^{1}_{M}}\sum_i \nor{f_i}{H^3_{\R^3}}+ \nor{\Delta_Mv}{H^{-2}_{M}}\sum_i\nor{g_i}{H^2_{\R^3}}\\
& \lesssim &\nor{(u,v)}{H^1_{M}\times L^2_{M}}\left(\sum_i \nor{(\Delta_{\R^3}f_i,\Delta_{\R^3} g_i)} {H^1_{\R^3}\times L^2_{\R^3}}+ \nor{(f_i,g_i)} {H^1_{\R^3}\times L^2_{\R^3}}\right).
\ena
Therefore, we get $\nor{A_M^2\Lambda^*(f,g)}{H^1_M\times L^2_M} \leq C \left(\nor{A^2_{\R^3,N}(f,g)}{(H^1_{\R^3}\times L^2_{\R^3})^N}+\nor{(f,g)}{(H^1_{\R^3}\times L^2_{\R^3})^N}\right)$ and Proposition \ref{propPGGalloscill} implies that (strict) $(h_n)$-oscillation of $(u_n)$ with respect to $A_M$ implies (strict) $(h_n)$-oscillation of $\Lambda u_n$ with respect to $A_{\R^3,N} $.

To prove the other implication, we use a quite similar operator. Denote $\varphi_i$ some other cut-off functions in $C^{\infty}_0(V_i)\subset C^{\infty}_0(\R^3)$ such that $\varphi_i\equiv 1$ on $Supp (\Phi_{i*}\Psi_i)$. We define $\Gamma$ the bounded operator from $(H^1_{\R^3}\times L^2_{\R^3})^N$ to $H^1_M\times L^2_M$ given by 
\bna
\Gamma (f,g)=\sum_i \Phi_{i*}^{-1}\varphi_i (f_i,g_i)
\ena
Then, we have $\Gamma \circ \Lambda =Id$ and we only have to prove that (strict) $(h_n)$-oscillation of $(f_n,g_n)$ with respect to $A_{\R^3,N}$ implies (strict) $(h_n)$-oscillation of $\Gamma (f_n,g_n)$ with respect to $A_M$. The needed estimates are quite similar and we omit them. 
\end{proof}
\begin{remarque}
Another way to prove Proposition \ref{proposcill} would have been to use the pseudodifferential operators $\varphi(h^2\Delta_M)$ as in \cite{Strichartz}.
\end{remarque}
Now, we will prove that the ($h_n$)-oscillation is conserved by the equation, even with a damping term.
\begin{prop}Let $T>0$. \label{propproposcilamorti}
Let $(\varphi_n,\psi_n)$ a bounded sequence of $\HutL$ that is (stricly) ($h_n$)-oscillatory with respect to $A_M$. Let $u_n$ be the solution of
\begin{eqnarray}
\label{eqnamortioscill}
\left\lbrace
\begin{array}{rcl}
\Box u_n +u_n&=&a(x)\partial_t u_n \quad \textnormal{on}\quad [0,T]\times M\\
(u_n(0),\partial_t u_n(0))&=&(\varphi_n,\psi_n) .
\end{array}
\right.
\end{eqnarray}
Then, $(u_n(t),\partial_t u_n(t))$ are (strictly) ($h_n$)-oscillatory with respect to $A_M$, uniformly on $[0,T]$.\\
At the contrary, if $(\varphi_n,\psi_n)$ is ($h_n$)-singular with respect to $A_M$, $(u_n(t),\partial_t u_n(t))$ is ($h_n$)-singular with respect to $A_M$, uniformly on $[0,T]$.
\end{prop}
\begin{proof}
Let $\chi \in C^{\infty}_0(\R)$ such that $0\leq \chi(s)\leq 1$ and $\chi(s)=1$ for $|s| \leq 1$. The  ($h_n$)-oscillation (resp strict oscillation) is equivalent to $\limsu{n}{\infty} \nor{(1-\chi)(R^2 h_n^2 \Delta)(u_n,\partial_t u_n)}{\HutL} \tend{R}{\infty}0$\\ (resp $\limsu{n}{\infty} \nor{\chi(\frac{ h_n^2 \Delta}{R^2})(u_n,\partial_t u_n) }{\HutL} \tend{R}{\infty}0$).

$v_n=(1-\chi)(R^2 h_n^2 \Delta) u_n$ is solution of  
\begin{eqnarray}
\left\lbrace
\begin{array}{rcl}
\Box v_n +v_n&=&a(x)\partial_t v_n -[\chi(R^2 h_n^2 \Delta),a] \partial_t u_n\quad \textnormal{on}\quad [0,T] \times M\\
(v_n(0),\partial_t v_n(0))&=&(1-\chi)(R^2 h_n^2 \Delta)(\varphi_n,\psi_n).
\end{array}
\right.
\end{eqnarray}
and energy estimates give
\bna
\nor{(v_n(t),\partial_t v_n(t))}{\HutL}&\leq &C_T\nor{(1-\chi)(R^2 h_n^2 \Delta)(\varphi_n,\psi_n)}{\HutL}+C_T \nor{[a,\chi(R^2 h_n^2 \Delta)]\partial_t u_n}{L^1([0,t],L^2)}\\
&\leq&C_T\nor{(1-\chi)(R^2 h_n^2 \Delta)(\varphi_n,\psi_n)}{\HutL}+C_{T} R h_n.
\ena
where the last inequality comes from the fact that $\chi(-h^2\Delta)$ is a semiclassical pseudodifferential operator, as proved in Burq, Gérard and Tzvetkov \cite{Strichartz}, Proposition 2.1 using the Helffer-Sjöstrand formula.

Therefore, passing to the limitsup in $n$ and using the oscillation assumption, we get the expected result uniformly in $t$ for $0\leq t \leq T$. The results for strict oscillation and singularity are proved similarly. 
\end{proof}
\begin{prop}
\label{propBinftamorti}
There exists $C_T>0$ such that for every $(\varphi_n,\psi_n)$ bounded sequence of $\HutL$ weakly convergent to $0$, we have the estimate 
\bna
\underset{n\to \infty}{\varlimsup} \nor{\left(u_n,\partial_t u_n\right)}{L^{\infty}([0,T],B^1_{2,\infty}(M)\times B^0_{2,\infty}(M)}) \leq C_T \underset{n\to \infty}{\varlimsup} \nor{\left(\varphi_n,\psi_n\right)}{B^1_{2,\infty}(M)\times B^0_{2,\infty}(M)} 
\ena
where $u_n$ is the solution of (\ref{eqnamortioscill}) . 
\end{prop}
\bnp Without loss of generality and since the equation is linear, we can assume that $\nor{(\varphi_n,\psi_n)}{\HutL}$ is bounded by $1$. Let $\e>0$. Let $\chi_0, \chi \in C^{\infty}_0(\R)$ so that $1=\chi_0+\sum_{k=1}^{\infty} \chi(2^{-2k}x)$.
We denote $u_n^k=\chi(2^{-2k}\Delta)u_n$. Using the same estimates as in the previous lemma, we get
\bna
\nor{(u_n^k(t),\partial_t u_n^k(t))}{\HutL}\leq C_T \nor{(u_n^k(0),\partial_t u_n^k(0))}{\HutL}+C_{T} 2^{-k}.
\ena
Take $K$ large enough so that $C_{T} 2^{-k}\leq \e$ for $k\geq K$ so that we have .
\bnan
\label{ineg1Bi}
\nor{(u_n^k(t),\partial_t u_n^k(t))}{\HutL}\leq C_T \nor{\left(\varphi_n,\psi_n\right)}{B^1_{2,\infty}(M)\times B^0_{2,\infty}(M)}+\e.
\enan
Then, for $k<K$, we use again some energy estimates for the equation verified by $u_n^k$, we get 
\bna
\nor{(u_n^k(t),\partial_t u_n^k(t))}{\HutL}\leq C_T \nor{(u_n^k(0),\partial_t u_n^k(0))}{\HutL}+C_T \nor{[a,\chi(-2^{-2k} \Delta)]\partial_t u_n}{L^1([0,T],L^2)}.
\ena
Yet, for fixed $k$, $[a,\chi(-2^{-2k} \Delta)]$ is an operator from $L^2$ into $H^1$ (for instance) and we conclude by the Aubin-Lions Lemma that for fixed $k\leq K$
\bnan
\label{ineg2Bi}
\underset{n\to \infty}{\varlimsup} \nor{(u_n^k(t),\partial_t u_n^k(t))}{\HutL}\leq C_T \underset{n\to \infty}{\varlimsup} \nor{(u_n^k(0),\partial_t u_n^k(0))}{\HutL}.
\enan
We get the expected result with an additional $\e$ by combining (\ref{ineg1Bi}) and (\ref{ineg2Bi}). 
\enp
We end this subsection by two lemma that will be useful in the nonlinear decomposition. The first one is lemma 3.2 of \cite{PGGall2001}.
\begin{lemme}
\label{lmorthL3PGGall}
Let $h_n$ and $\tilde{h_n}$ be two orhogonal scales, and let $(f_n)$ and $\tilde{f_n}$ be two sequences such that such $\nabla f_n$ (resp $\nabla \tilde{f}_n)$) is strictly ($h_n$) (resp $\tilde{h}_n$)-oscillatory with respect to $\Delta_{\R^3}$. Then, we have : 
\bna
\limsu{n}{\infty}\nor{f_n\tilde{f}_n}{L^3(\R^3)}=0
\ena
\end{lemme}
Then, we easily deduce the following result.
\begin{lemme}
\label{lmorthL3}
Let $h_n$ and $\tilde{h_n}$ be two orhogonal scales and $v_n$, $\tilde{v}_n$ be two sequences that are strictly $h_n$ (resp $\tilde{h_n}$) oscillatory with respect to $\Delta_M$ (considered on the Hilbert space $H^1$), uniformly on $[-T,T]$. Then, we have
\bna
\nor{v_n\tilde{v}_n}{L^{\infty}([-T,T],L^3(M))} \tend{n}{\infty} 0 \label{resultorthL3}
\ena
Moreover, the same result remains true if $\tilde{v}_n$ is a constant sequence $v\in H^1$ and $\tilde{h}_n=1$. 
\end{lemme}
\begin{proof}
Using a partition of unity $1= \sum_i \Psi_i^2 $ adapted to coordinate charts, we have to compute
\bna
\nor{\Phi_{i*}\Psi_i v_n\Psi_i \tilde{v}_n}{L^{\infty}([-T,T],L^3(R^3))}.
\ena
Using Proposition \ref{proposcill}, we infer that $\Phi_{i*}\Psi_i v_n$ is strictly ($h_n$)-oscillatory with respect to $\Delta_{\R^3}$ (defined on $H^1$) and the same result holds for $\nabla \left(\Phi_{i*}\Psi_i v_n\right)$ with respect to $\Delta_{\R^3}$ defined on $L^2$. We conclude by applying Lemma \ref{lmorthL3PGGall} to $\Phi_{i*}\Psi_i v_n$ and $\Phi_{i*}\Psi_i v_n$. 
\end{proof}
\subsection{Microlocal defect measure and energy}
In this subsection, we state without proof the propagation of the measure for the damped wave equation. We refer to \cite{defectmeasure} for the definition and to \cite{linearisationondePG} Section 4 or \cite{FrancfortMuratoscillation} in the specific context of the wave equation). It will be used several times in the article.
\begin{lemme}
\label{lmmeasuredamped}[Measure for the damped equation and equicontinuity of the energy]
Let $u_n$, $\tilde{u}_n$ be two sequences of solution to 
\bna 
\Box u_n+u_n=a(x)\partial_t u_n,
\ena
 weakly convergent to $0$ in $\HutL$. Then, there exists a subsequence (still denoted $u_n$, $\tilde{u}_n$) such that for any $t\in [0,T]$ there exists a (nonnegative if $u_n=\tilde{u}_n$) Radon measure $\mu^t$ on $S^*M$ such that for any classical pseudodifferential operator $B$ of order $0$, we have with a uniform convergence in $t$
\bnan
\label{cvgcemeasure}
\left(B (-\Delta)^{1/2}u_n(t),(-\Delta)^{1/2}\tilde{u}_n(t)\right)_{L^2(M)}+\left(B \partial_t u_n(t),\partial_t \tilde{u}_n(t)\right)_{L^2(M)} \tend{n}{\infty} \int_{S^*M} \sigma_0(B)d\mu^t.
\enan
 Moreover, one can decompose
 \bna
 \mu^t=\frac{1}{2}(\mu^t_+ +\mu^t_{-})
 \ena
which satisfy the following transport equation 
$$ \partial_t \mu_{\pm}(t) = \pm H_{|\xi|_x} \mu_{\pm}(t)+a(x)\mu_{\pm}(t). $$
Furthermore, if $t_n\tend{n}{\infty}t$, we have the same convergence with $t$ replaced by $t_n$ in (\ref{cvgcemeasure}). 
\end{lemme}
The microlocal defect measure of a concentrating data $[(\varphi,\psi),\underline{h},\underline{x}]$ can be explicitely computed, as follows
\bna
\mu_{\pm}=(2\pi)^{-3} \delta_{x_{\infty}}(x)\otimes \int_0^{+\infty} \left|\hat{\psi}(r\xi) \pm i |r\xi|_{\infty} \hat{\varphi}(r\xi)\right|^2 r^2 dr.  
\ena
This can be easily computed, for instance, with the next lemma.
\begin{lemme}
\label{lmpseudoconc}Let $(\varphi_n,\psi_n)=[(\varphi,\psi),\underline{h},\underline{x}]$ be a concentration data and $A(x,D_x), B(x,D_x)$ two polyhomogeneous pseudodifferential operators of respective order $0$. Then
\bna
\nor{(A(x,D_x)\varphi_n,B(x,D_x)\psi_n)- [(A_0(x_{\infty},D_x)\varphi,B_0(x_{\infty},D_x)\psi),\underline{h},\underline{x}]}{H^1\times L^2} \tend{n}{\infty}0
\ena
where $A_0(x_{\infty},D_x)$ is the Fourier multiplier of homogeneous symbol $a_0(x_{\infty},\xi)$ defined on $T^*_{x_{\infty}}M$
\end{lemme}
\bnp
We only give a sketch of proof for $B(x,D_x)\psi_n$. By approximation, we can assume that $\widehat{\psi}\in C^{\infty}_0(\R^3 \backslash 0)$. In local coordinates centered at $x_{\infty}=0$, we have for a $\petito{1}$ small in $L^2$
$$B(x,D_x)\psi_n=h_n^{-\frac{3}{2}}B(x,D_x)\left[\Psi_U(x) (\chi\psi) \left(\frac{x-x_n}{h_n}\right)\right]+\petito{1}=h_n^{-\frac{3}{2}}\left[B_n(y,D_y)\psi \right]\left(\frac{x-x_n}{h_n}\right)+\petito{1}$$
where $B_n(y,D_y)$ is the operator of symbol $b_n(y,\xi)=b_0(h_ny+x_n,\xi/h_n)$. Here $b_0$ is the principal symbol of $B$, homogeneous for large $\xi$. We write $b_0(h_ny+x_n,\xi/h_n)=b_0(x_n,\xi/h_n)+h_n y \int_0^1 (\partial_y b_0)(x_n+th_n,\xi/h_n)dt$. The first term converges to $b_0(0,\xi)$ by homogeneity while the second produces a term small in $L^2$.
\enp
The previous lemma is made interesting when combined with the propagation of microlocal defect measure. 
\begin{lemme}
\label{lmconserorth}
Let $u_n$ a sequence of solutions of $\Box u_n +u_n=a(x)\partial_t u_n$ weakly convergent to $0$ and $p_n=[(\varphi,\psi),\underline{h},\underline{x},\underline{t}]$ a linear damped concentrating wave. We assume $D_h (u_n,\partial_t u_n) \rightharpoonup 0$. Then, for any classical pseudodifferential operators $A(x,D_x)$ of order $0$, we have uniformly for $t\in [-T,T]$
\bna
\left(A(-\Delta)^{1/2}p_n(t),(-\Delta)^{1/2}u_n(t)\right)_{L^2(M)}+\left(A\partial_t p_n(t),\partial_t u_n(t)\right)_{L^2(M)} \tend{n}{\infty}0.
\ena
In particular, we have
\bna
\nabla p_n\cdot \nabla u_n+\partial_t p_n\partial_t u_n \rightharpoonup 0 \textnormal{ in } \mathcal{D}'(]-T,T[\times M).
\ena
\end{lemme}
\bnp
We first check the property for $t=t_n$. Using Lemma \ref{lmpseudoconc} several times, we are led to estimate
\bna
\left((-\Delta)^{1/2}\varphi_n,(-\Delta)^{1/2}u_n(t_n)\right)_{L^2(M)}+\left(\psi_n,\partial_t u_n(t_n)\right)_{L^2(M)}
\ena
where $(\varphi_n,\psi_n)$ are the concentrating data associated with $[(A(x_{\infty},D_x)\varphi,B(x_{\infty},D_x)\psi),\underline{h},\underline{x}]$. Then, the hypotheses $D_h (u_n,\partial_t u_n) \rightharpoonup 0$ and Lemma \ref{defweaktrackequiv} yields the convergence to $0$ for this particular case $t=t_n$.
We conclude by equicontinuity and by the propagation of joint measures stated in Lemma \ref{lmmeasuredamped}.  
\enp

\section{Profile Decomposition}
\subsection{Linear profile decomposition}
The main purpose of this section is to establish Theorem \ref{thmdecompositionL}. It is completed in two main steps : the first one is the extraction of the scales $h_n^{(j)}$ where we decompose $v_n$ in an infinite sum of sequence $v_n^{(j)}$ which are respectively $h_n^{(j)}$-oscillatory and the second steps consists in decomposing each $v_n^{(j)}$ in an infinite sum of concentrating wave at the rate $h_n^{(j)}$.  Actually, in order to perform the nonlinear decomposition, we will need that, in some sense, each profile of the decomposition do not interact with the other. It is stated in this orthogonality result. 

\medskip

\noindent \textbf{Theorem \ref{thmdecompositionL}'.} \textit{With the notation of Theorem \ref{thmdecompositionL}, we have the additional following properties.\\
If $2T<T_{focus}$, we have $(\underline{h}^{(k)},\underline{x}^{(k)},\underline{t}^{(k)})\perp (\underline{h}^{(j)},\underline{x}^{(j)},\underline{t}^{(j)})$ for any $j\neq k$, according to Definition \ref{deforthog}.\\
If $M=S^3$ and $a\equiv 0$ (undamped solutions), but with $T$ eventually large, we have\\ $(\underline{h}^{(k)},(-1)^m\underline{x}^{(k)},\underline{t}^{(k)}+m\pi)$ orthogonal to $(\underline{h}^{(j)},\underline{x}^{(j)},\underline{t}^{(j)})$ for any $m\in \Z$ and $j\neq k$.}

\medskip
\subsubsection{Extraction of scales}
\begin{prop}\label{propextractscales}
Let $T>0$. Let $(\varphi_n,\psi_n)$ a bounded sequence of $\HutL$ and $v_n$ the solution of
\begin{eqnarray}
\label{soldampedscale}
\left\lbrace
\begin{array}{rcl}
\Box v_n +v_n&=&a(x)\partial_t v_n \quad \textnormal{on}\quad [-T,T] \times M\\
(v_n(0),\partial_t v_n(0))&=&(\varphi_n,\psi_n).
\end{array}
\right.
\end{eqnarray}
Then, up to an extraction, $v_n$ can be decomposed in the following way : for any $l\in \N^*$
\bna
\label{decomplinscale}v_n(t,x)= v(t,x)+\sum_{j=1}^l v_n^{(j)}(t,x)+\rho_n^{(l)}(t,x),
\ena 
where $v_n^{(l)}$ is a strictly ($h_n^{(j)}$)-oscillatory solution of the damped linear wave equation (\ref{soldampedscale}) on $M$. The scales $h_n^{(j)}$ satisfy $h_n^{(j)} \tend{n}{\infty} 0$ and are orthogonal :
\bnan \label{orthpropscale}
\left|\log \frac{h_n^{(k)}}{h_n^{(j)}}\right| \tend{n}{\infty} +\infty \textnormal{ if }j\neq k.
\enan
Moreover, we have 
\bnan
&&\underset{n\to \infty}{\varlimsup} \nor{\rho_n^{(l)}}{L^{\infty}(]-T,T[,L^6(M))} \tend{l}{\infty}0\label{resutlL6} \\
&&\nor{(v_n,\partial_tv_n)(t)}{\HutL}^2=\nor{(v,\partial_tv)(t)}{\HutL}^2+\sum_{j=1}^l \nor{(v_n^{(j)},\partial_tv_n^{(j)})(t)}{\HutL}^2 + \nor{(\rho_n^{(l)},\partial_t\rho_n^{(l)})(t)}{\HutL}^2 +  \petito{1}(t),\label{resutorthen}
\enan
where $\petito{1}(t)\tend{n}{\infty}0 $ uniformly for $t\in [-T,T]$.
\end{prop}
\begin{proof}
We first make this decomposition for the initial data as done in \cite{PGdefautSobolev} (see also \cite{BahouriGerard}). Then, using the propagation of ($h_n$)-oscillation proved in Proposition \ref{propproposcilamorti}, we extend it for all time. 

More precisely, by applying the same procedure as in \cite{PGdefautSobolev}, with the operator $A_M$, we decompose 
\bna
(\varphi_n,\psi_n)=(\varphi,\psi)+\sum_{j}^l(\varphi_n^{(j)},\psi_n^{(j)})+(\Phi_n^{(l)},\Psi_n^{(l)})
\ena
where $(\varphi_n^{(j)},\psi_n^{(j)})$ is $(h_n^{(j)})$-oscillatory for $A_M$, $h_n^{(j)}\tend{n}{\infty}0$, and 
\bnan
\label{cvcebesov}
\underset{n\to \infty}{\varlimsup} \sup_{k\in \N}\nor{\mathbf{1}_{[2^k,2^{k+1}[}(A_M)(\Phi_n^{(l)},\Psi_n^{(l)})}{\HutL}  \tend{l}{\infty}0.
\enan
Moreover, we have the orthogonality property : 
\bna
\nor{(\varphi_n,\psi_n)}{\HutL}^2=\nor{(\varphi,\psi)}{\HutL}^2+\sum_{j}^l\nor{(\varphi_n^{(j)},\psi_n^{(j)})}{\HutL}^2+\nor{(\Phi_n^{(l)},\Psi_n^{(l)})}{\HutL}^2+\petito{1},\quad n\to \infty 
\ena
and the $h_n^{(j)}$ are orthogonal each other as in (\ref{orthpropscale}). Moreover, $(\Phi_n^{(l)},\Psi_n^{(l)})$ is ($h_n^{(j)}$)- singular for $1\leq j\leq l$. 

This decomposition for the initial data can be extended to the solution by 
\bna
v_n(t,x)= v(t,x)+\sum_{j}^l v_n^{(j)}(t,x)+\rho_n^{(l)}(t,x),
\ena
where each $v_n^{(j)}$ is solution of
\begin{eqnarray*}
\left\lbrace
\begin{array}{rcl}
\Box v_n^{(j)}+v_n^{(j)} &=&a(x)\partial_t v_n^{(j)} \quad \textnormal{on}\quad \R_t\times M\\
(v_n^{(j)}(0),\partial_t v_n^{(j)}(0))&=&(\varphi_n^{(j)},\psi_n^{(j)}) .
\end{array}
\right.
\end{eqnarray*}
Thanks to Proposition \ref{propproposcilamorti}, each $(v_n^{(j)}(t),\partial_t v_n^{(j)}(t))$ is strictly ($h_n^{(j)}$)-oscillatory and $(\rho_n^{(l)}(t),\partial_t \rho_n^{(l)}(t))$ is ($h_n^{(j)}$)- singular for $1\leq j\leq l$. 
So, we easily infer for instance that\\ $\left\langle(\rho_n^{(l)}(t),\partial_t \rho_n^{(l)}(t)), (v_n^{(j)}(t),\partial_t v_n^{(j)}(t))\right\rangle_{\HutL}\tend{n}{\infty}0$ uniformly on $[-T,T]$ where $\left\langle ~,~\right\rangle_{\HutL}$ is the scalar product on $\HutL$. This is also true for the product between $v_n^{(j)}$ and $v_n^{(k)}$, $j\neq k$ thanks to the orthogonality (\ref{orthpropscale}). The same convergence holds for the product with $v$ by weak convergence to $0$ of the other terms. Then, we get
\bna
\nor{(v_n,\partial_t v_n)}{\HutL}^2=\nor{(v,\partial_t v)}{\HutL}^2+\sum_{j}^l \nor{(v_n^{(j)},\partial_t v_n^{(j)})}{\HutL}^2+\nor{(\rho_n^{(l)},\partial_t \rho_n^{(l)})}{\HutL}^2+\petito{1},\quad n\to \infty. 
\ena
Let us now prove estimate (\ref{resutlL6}) of the remaining term in $L^{\infty}(L^6)$. (\ref{cvcebesov}) gives the convergence to zero of $(\rho_n^{(l)}(0),\partial_t \rho_n^{(l)}(0))$ in $B^1_{2,\infty}(M)\times B^0_{2,\infty}(M)$. We extend this convergence for all time with Proposition \ref{propBinftamorti} and get
\bna
\sup_{t\in [0,T]}\underset{n\to \infty}{\varlimsup}\nor{\left(\rho_n^{(l)}(t),\partial_t\rho_n^{(l)}(t)\right)}{B^1_{2,\infty}\times B^0_{2,\infty}} \tend{l}{\infty }0.
\ena
The following lemma will transfer this information in local charts.
\begin{lemme}
\label{equivBinft}
There exists $C>0$ such that 
\bna
\frac{1}{C}\nor{\Lambda f}{B^0_{2,\infty}(\R^3)^N}\leq \nor{f}{B^0_{2,\infty}(M)}\leq C\nor{\Lambda f}{B^0_{2,\infty}(\R^3)^N}\\
\frac{1}{C}\nor{\Lambda f}{B^1_{2,\infty}(\R^3)^N}\leq \nor{f}{B^1_{2,\infty}(M)}\leq C\nor{\Lambda f}{B^1_{2,\infty}(\R^3)^N}
\ena  
where $\Lambda$ is the operator described in Proposition \ref{proposcill} of cut-off and transition in $N$ local charts.
\end{lemme}
We postpone the proof of this lemma and continue the proof of the proposition. Using this lemma, we get for every coordinate patch $(U_i,\Phi_i)$ and $\Psi_i\in C^{\infty}_0(U_i)$.
\bna
\underset{n\to \infty}{\varlimsup}\nor{\Phi_{i}^*\Psi_i\rho_n^{(l)}}{L^{\infty}([-T,T], B^1_{2,\infty}(\R^3))}\tend{l}{\infty }0
\ena
The refined Sobolev estimate, Lemma 3.5 of \cite{BahouriGerard}, yields for any $f\in \Hu(\R^3)$
\bna
\nor{f}{L^6(\R^3)}\leq \nor{(-\Delta_{\R^3})^{1/2}f}{L^2}^{1/3}\nor{(-\Delta_{\R^3})^{1/2}f}{\dot{B}^0_{2,\infty}}^{2/3}\leq \nor{f}{H^1(\R^3)}^{1/3}\nor{f}{B^1_{2,\infty}(\R^3)}^{2/3}
\ena
Therefore, we have 
 $$\underset{n\to \infty}{\varlimsup}\nor{\Phi_{i}^*\Psi_i\rho_n^{(l)}}{L^{\infty}([-T,T],L^6(\R^3))}\tend{l}{\infty }0$$ and finally
$$\underset{n\to \infty}{\varlimsup}\nor{\rho_n^{(l)}}{L^{\infty}([-T,T],L^6(M))}\tend{l}{\infty }0$$
This completes the proof of Proposition \ref{propextractscales}, up to the proof of Lemma \ref{equivBinft}.
\end{proof}
\begin{proof}[Proof of Lemma \ref{equivBinft}]We essentially use the following fact : see Lemma 3.1 of \cite{BahouriGerard}. Let $f_n$ be a sequence of $L^2(\R^3)$ weakly convergent to $0$ and compact at infinity
\bna
\underset{n\to +\infty}{\varlimsup} \int_{\left|x\right|>R} \left|f(x)\right|^2~dx \tend{R}{+\infty}0.
\ena
Then, $f_n$ tends to $0$ in $\dot{B}^0_{2,\infty}(\R^3)$ if and only if $f_n$ is $h_n$ singular for every scale $h_n$.

Actually, the same result holds $\Delta_M$, with the same demonstration. The compactness at infinity in $\R^3$ is only assumed to ensure 
\bna
\underset{n\to +\infty}{\varlimsup}\nor{\mathbf{1}_{[-A,A](\Delta_{\R^3})}f_n }{L^2} =0 \textnormal{ for any } A>0
\ena
which is obvious in the case of $\Delta_M$ because of weak convergence and discret spectrum.

Using Proposition \ref{proposcill}, we obtain that $f_n$ is ($h_n$)-singular with respect to $\Delta_M$ if and only if $\Lambda f_n $ is ($h_n$)-singular with respect to $\Delta_{\R^3}$. Combining both previous results, we obtain that the two norms we consider have the same converging sequences and are therefore equivalent. 
\end{proof}
\subsubsection{Description of linear concentrating waves (after S. Ibrahim)}
In this subsection, we describe the assymptotic behavior of linear concentrating waves as described in \cite{ibrahim2004gon} of S. Ibrahim. In \cite{ibrahim2004gon}, it is stated for the linear wave equation without damping. We give some sketch of proof when necessary to emphasize the tiny modifications. 

The following lemma yields that for times close to concentration, the linear damped concentrating wave is close to the solution of the wave equation with flat metric and without damping. It is Lemma 2.2 of \cite{ibrahim2004gon}, except that there is an additional damping term which disappears after rescaling. We do not give the proof and refer to the more complicated nonlinear case (see estimate (\ref{dampingdisappear}) ).
\begin{lemme}
\label{lmlinproche}
Let $v_n= [(\varphi,\psi), \underline{h}, \underline{x}, \underline{t})]$ be a linear damped concentrating wave and $v$ solution of 
\begin{eqnarray}
\left\lbrace
\begin{array}{rcl}
\Box_{\infty} v&=& 0\quad \textnormal{on}\quad \R \times T_{x_{\infty}}M\\
(v,\partial_t v)_{|t=0}&=&(\varphi,\psi)
\end{array}
\right.
\end{eqnarray}
Denote $\tilde{v}_{n}$ the rescaled function associated to $v$, that is $\tilde{v}_{n}=\Phi^*\Psi\frac{1}{\sqrt{h_n}}v\left(\frac{t-t_n}{h_n},\frac{x-x_n}{h_n}\right)$ where $(U,\Phi)$ is a coordinate chart around $x_{\infty}$ and $\Psi\in C^{\infty}_0(U)$ is constant equal to $1$ around $x_{\infty}$. Then, we have
\bna
\limsu{n}{\infty}\nort{\tilde{v}_{n}-v_n}_{[t_n-\Lambda h_n,t_n+\Lambda h_n]\times M}\tend{\Lambda}{\infty}0
\ena
\end{lemme}
\begin{corollaire}
\label{corconclintproche}
With the notation of the Lemma, if $\tilde{t_n}=t_n+(C+\petito{1})h_n$, then $(v_n,\partial_t v_n)_{|t=\tilde{t_n}}$ is a concentrating data associated with $[(v(C),\partial_t v), \underline{h}, \underline{x}]$.
\end{corollaire}
Moreover, Lemma 2.3 of S. Ibrahim \cite{ibrahim2004gon} yields the "non reconcentration" property for linear concentrating waves.
\begin{lemme}
\label{lmnonreconcentr}
Let $\underline{v}=[(\varphi,\psi, \underline{h}, \underline{x}, \underline{t})]$ be a linear (possibly damped) concentrating wave. 
Consider the interval $[-T,T]$ containing $t_{\infty}$, satisfying the following
non-focusing property (see Definition \ref{defcouplefocus})
\bnan
\label{nonfocusing}
mes\left(F_{x,x_{\infty},s}\right)=0 \quad \forall x\in M \textnormal{ and } s\neq 0 \textnormal{ such that } t_{\infty}+s\in [-T,T].
\enan
Then, if we set $I_n^{1,\Lambda}=[-T,t_n-\Lambda h_n]$ and  $I_n^{3,\Lambda}=]t_n+\Lambda h_n,T]$, we have
\bna
\varlimsup_n \nor{v_n}{L^{\infty}(I_n^{1,\Lambda}\cup I_n^{3,\Lambda},L^6(M))} \tend{\Lambda}{\infty} 0\\
\varlimsup_n \nor{v_n}{L^{5}(I_n^{1,\Lambda}\cup I_n^{3,\Lambda},L^{10}(M))} \tend{\Lambda}{\infty} 0.
\ena
\end{lemme}
\begin{proof}[Sketch of the proof of Lemma \ref{lmnonreconcentr} in the damped case]To simplify the notation, we can assume $t_n=0$.
In \cite{ibrahim2004gon}, the proof is made by contradiction, assuming the existence of a subsequence (still denoted $v_n$) such that $\nor{v_{n}(s_n)}{L^6(M)} \rightarrow C>0$ and $\frac{|s_n|}{h_n}\rightarrow \infty$. If $s_n \rightarrow \tau \neq 0$, using Concentration-Compactness principle of \cite{Lionsconcentration}, we are led to prove that the microlocal defect measure $\mu$ associated to $v_n(t_n)$ satisfies $\mu(\left\{y\right\}\times S^2)=0$ for any $y\in M$. We use the same argument for the damped equation except that in that case, the measure $\mu^t$ associated to $v_n(t)$ is not solution of the exact transport equation but of a damped transport equation (see for Lemma \ref{lmmeasuredamped}). Yet, the non focusing assumption (\ref{nonfocusing}) still implies $\mu^t(\left\{y\right\}\times S^2)=0$ for all $y\in M$ and $t\neq 0$, which allows to conclude similarly. 

In the case $\tau=0$, we use in local coordinates the rescaled function $\tilde{v}_n(s,y)=\sqrt{s_n}v_n(s_n s,s_ny+x_n)$. $\tilde{v}_n$ at time $s=0$ is a concentrating data at scale $h_n/s_n$. We prove $\lim \nor{\tilde{v}_n(1,\cdot)}{L^6(\R^3)}=0$. Again by concentration compactness, it is enough to prove that the microlocal defect measure $\mu^s$ of $\tilde{v}_n$ propagates along the curves of the hamiltonian flow with constant coefficient $H_{|\xi|}$. Since $v_n$ is solution of $\Box v_n+v_n+a(x)\partial_t v_n=0$, $\tilde{v}_n$ is solution of $\Box_n \tilde{v}_n+s_n^2\tilde{v}_n+s_n a(s_n\cdot+x_n)\partial_t \tilde{v}_n=0$ where $\Box_n$ is a suitably rescaled D'Alembert operator. Since the additional terms $s_n^2 \tilde{v}_n+s_n a(s_n\cdot+x_n)\partial_t \tilde{v}_n$ converges to $0$ in $L^1L^2$, we can finish the proof as in Lemma 2.3 of  \cite{ibrahim2004gon} by proving that $\mu^s$ propagates as if $\Box_n$ was replaced by $\Box_{\infty}$, that is along the hamiltonian $H_{|\xi|}$. 

The estimate in norm $L^5L^{10}$ is obtained by interpolation of $L^{\infty}L^6$ with another bounded Strichartz norm.
\end{proof}
In the specific case of $S^3$, Lemma 4.2 of \cite{ibrahim2004gon} allows to describe precisely the behavior of concentrating wave for large times, as follows. 
\begin{lemme}
\label{S3perio}
Let $\underline{p}$ be a sequence of solutions of 
\begin{eqnarray*}
\left\lbrace
\begin{array}{rcl}
\Box p_n&=&0 \quad \textnormal{on}\quad [0,+\infty [\times M\\
(p_n(0),\partial_t p_n(0))&=&(\varphi_n,\psi_n)
\end{array}
\right.
\end{eqnarray*}
where $(\varphi_n,\psi_n)$ is weakly convergent to $(0,0)$ in $\HutL$ . Then, we have
$$p_n(t+\pi,x)=-p_n(t,-x)+\petito{1}(t)$$
where the $\petito{1}(t)$ is small in the energy space. The same holds for solutions of $\Box u_n+u_n$.

In particular, if $\underline{p}$ is a concentrating wave associated with data $[(\varphi,\psi, \underline{h}, \underline{x}, \underline{t})]$, then, for any $j\in \N$, $p_n(t+j\pi,x)$ is a linear concentrating wave associated with $[(-1)^j(\varphi,\psi)((-1)^j.), \underline{h},(-1)^j \underline{x}, \underline{t})]$
\end{lemme}
In the previous lemma, $-x$ refers to the embedding of $S^3$ into $\R^4$. Moreover, the notation $(\varphi,\psi)(-.)$ could be written more rigorously $(\varphi,\psi)(D_{\infty}I.)$ where $D_{\infty}I$ is the differential at the point $x_{\infty}$ of the application $I:x\mapsto -x$ defined from $S^3$ into itself. Actually, we are identifying the tangent plane at the south pole with the one on the north pole by the application $x\mapsto -x$ on $\R^4$.

The fact that the result remains true for the equation $\Box u+u=0$ comes from the fact that for initial data weakly convergent to zero, the solutions of $\Box u=0$ and $\Box v+v=0$ with same data are asymptotically close in the energy space. This can be proved by observing that for a weakly convergent sequence of solutions $u_n$ the Aubin-Lions Lemma yields that $u_n$ converges strongly to $0$ in $L^{\infty}([-T,T],L^2)$. So $r_n=u_n-v_n$ is solution of $\Box r_n=u_n$ and converges strongly in $\HutL$. 

\subsubsection{Extraction of times and cores of concentration}
In this section, $h_n$ is a fixed sequence in $\R_+^*$ converging to $0$. For simplicity, we will denote it by $h$ and $u_h$ for sequences of functions. The main purpose of this subsection is the proof of the following proposition, which is the profile decomposition for $h$-oscillatory sequences. It easily implies Theorem \ref{thmdecompositionL} when combined with Proposition \ref{propextractscales}.
\begin{prop}
\label{propextrconc}
Let ($u_h$) be a $h$-oscillatory sequence of solutions to the damped Klein-Gordon equation (\ref{soldampedscale}). Then, up to extraction, there exist damped linear concentrating waves $p_h^{k}$, as defined in Definition \ref{defconcwave}, associated to concentrating data $[(\varphi^{(k)},\psi^{(k)}),\underline{h},\underline{x}^{(k)},\underline{t}^{(k)}]$, such that for any $l\in \N^*$, and up to a subsequence,  
\bnan
&v_h(t,x)= \sum_{j=1}^l p_n^{(j)}(t,x)+w_n^{(l)}(t,x),\label{decomplincore} \\
&\forall T>0,\quad \underset{n\to \infty}{\varlimsup} \nor{w_h^{(l)}}{L^{\infty}(]-T,T[,L^6(M))}  \tend{l}{\infty}0\label{hresultL6} \\
&\nor{(v_h,\partial_tv_h)}{\HutL}^2=\sum_{j=1}^l \nor{(p_h^{(j)},\partial_t p_h^{(j)})}{\HutL}^2 + \nor{(w_h^{(l)},\partial_t w_h^{(l)})}{\HutL}^2 +  \petito{1}, \textnormal{ as } h\to \infty,\label{horthogNRJ}\\
&\qquad\qquad\qquad\qquad\qquad\qquad\qquad\qquad\qquad\qquad\qquad\textnormal{ uniformly for } t\in [-T,T].\nonumber
\enan
Moreover, if $2T<T_{focus}$, for any $j\neq k$, we have $(\underline{x}^{(k)},\underline{t}^{(k)})\perp_h (\underline{x}^{(j)},\underline{t}^{(j)})$ according to Definition \ref{deforthog}.\\
If $M=S^3$ and $a\equiv 0$ (undamped solutions), but with $T$ eventually large, $((-1)^m\underline{x}^{(k)},\underline{t}^{(k)}+m\pi)$ is orthogonal to $(\underline{x}^{(j)},\underline{t}^{(j)})$ for any $m\in \Z$ and $j\neq k$.
\end{prop}
\begin{remarque}
The assumptions to get the orthogonality of the cores of concentration are related to our lack of understanding of the solutions concentrating in a point $x_1$ where $(x_1,x_2,t)$ is a couple of focus at distance $t$. We know that the solution reconcentrates after a time $t$ in the other focus $x_2$ but we do not know precisely how : can it split into several concentrating waves on $x_2$ with different "rate of concentration"? That is to say with some different $x_n$ converging to $x_2$ but which are orthogonal.   
\end{remarque}
Before getting into the proof of the proposition, we state two lemmas that will be useful in the proof.
Using the notation of Definition \ref{defweaktrack}, denote 
$$\delta^x(\underline{v})=\sup_{\underline{x}}\left\{\nor{\nabla \varphi }{L^2(T_{x_{\infty}M})}^2, D_h^1 v_h\rightharpoonup \varphi, \textnormal{ up to a subsequence} \right\} $$
where the supremum is taken over all the sequences $\underline{x}$ in $M$.

If $v_h \in L^{\infty}([-T,T],H^1(M))$, we denote
$$\delta(\underline{v})=\sup_{\underline{x},\underline{t}}\left\{\nor{\nabla \varphi }{L^2(T_{x_{\infty}M})}^2, D_h^1 v_h(t_h)\rightharpoonup \varphi, \textnormal{ up to a subsequence} \right\}=\sup_{\underline{t}} \delta^x(\underline{v}(t_h,\cdot))$$
where the supremum is taken over all the sequences $\underline{x}=(x_h)$ in $M$ and $\underline{t}=(t_h)$ in $[-T,T]$.
\begin{lemme}
\label{lmmulttrack}
Let $\Psi \in C^{\infty}(M)$. Then, there exists $C>0$ such that for any $\underline{v}$, we have the estimate
\bna
\delta^x(\Psi \underline{v})\leq C\delta^x(\underline{v}).
\ena
\end{lemme}
The proof is left to the reader.
\begin{lemme}
\label{lmtrackL6}
There exists $C>0$ such that for any $\underline{v}=(v_h)$ a bounded strictly ($h_n$)-oscillatory sequence in $H^1(M)$ 
\bna
\underset{n\to +\infty}{\varlimsup} \nor{v_h}{L^6} \leq C \delta^x(\underline{v})^{1/3}\underset{n\to +\infty}{\varlimsup} \nor{ v_h}{H^1(M)}^{1/6}.
\ena
\end{lemme}
\begin{proof}
This lemma is already known in the case of $\R^3$ where the definition of $\delta^x_{\R^3}$ is the same except that $D_h^1$ is only considered in the trivial coordinate chart. It is estimate (4.19) of \cite{PGdefautSobolev} in the case of a $1$-oscillatory sequence, which can be easily extended to ($h_n$)- oscillatory sequence by dilation. 

Let $\Psi_i \in C^{\infty}_0(U_i)$ associated to a coordinate patch $\Phi_i$. By Proposition \ref{proposcill}, $\Phi_{i}^*\Psi_i v_h$ is still ($h_n$)-oscillatory and we can apply the estimate on $\R^3$. We get
\bna
&&\underset{n\to +\infty}{\varlimsup} \nor{\Phi_{i}^*\Psi_i v_h}{L^6(\R^3)} \leq C \delta^x_{\R^3}(\Phi_{i*}\Psi_i \underline{v})^{1/3}\underset{n\to +\infty}{\varlimsup} \nor{\Phi_{i}^*\Psi_i v_h}{H^1(\R^3)}^{1/6}\\
&\leq & C \delta^x_{\R^3}(\Phi_{i}^*\Psi_i \underline{v})^{1/3}\underset{n\to +\infty}{\varlimsup} \nor{v_h}{H^1(M)}^{1/6}
\ena
Then, by definition of the convergence $D_h$, we easily get 
$$ \delta^x_{\R^3}(\Phi_{i}^*\Psi_i \underline{v})\leq C\delta^x(\Psi_i \underline{v}).$$
We conclude using Lemma \ref{lmmulttrack} and partition of unity.
\end{proof}
\begin{lemme}
\label{lmtrackLiL6}
Let $T>0$. There exists $C>0$ such that for any sequence $\underline{v}=(v_h)$ ($h_n$)-oscillatory, solution of the damped linear Klein-Gordon equation on $M$ with bounded energy, we have
\bna
\underset{n\to +\infty}{\varlimsup} \nor{v_h}{L^{\infty}([-T,T],L^6(M)} \leq C \delta(\underline{v})^{1/3}\underset{n\to +\infty}{\varlimsup} \nor{(v_h(0),\partial_t v_h(0))}{\HutL}^{1/6}
\ena
\end{lemme}
\begin{proof}
Let $t_h$ be an arbitrary sequence in $[-T,T]$. We apply Lemma \ref{lmtrackL6} to the sequence $v_h(t_h)$ and get
\bna
\underset{n\to +\infty}{\varlimsup} \nor{v_h(t_h,\cdot)}{L^6} &\leq &C \delta^x(\underline{v}(t_h,\cdot))^{1/3}\underset{n\to +\infty}{\varlimsup} \nor{v_h(t_h)}{H^1(M)}^{1/6}\\
&\leq & C \delta(\underline{v})^{1/3}\underset{n\to +\infty}{\varlimsup} \nor{(v_h(0),\partial_t v_h(0))}{\HutL}^{1/6}
\ena
 by definition of $\delta$ and by energy estimates.
\end{proof}
\begin{proof}[Proof of Proposition \ref{propextrconc}]
It is based on the same extraction argument as \cite{BahouriGerard} and \cite{PGGall2001} : the concentration will be tracked using our tool $D_h$ and we will extract concentrating waves so that $\delta(\underline{v})$ decreases. We conclude with Lemma \ref{lmtrackLiL6} to estimate the $L^{\infty}(L^6)$ norm of the remainder term.

More precisely, if $\delta(\underline{v})=0$, Lemma \ref{lmtrackLiL6} shows that there is nothing to be proved. Otherwise, pick $(x_h^{(1)},t_h^{(1)})$ converging to $(x_{\infty}^{(1)},t_{\infty}^{(1)})$ and $(\varphi^{(1)},\psi^{(1)})\in \HutL_{x_{\infty}}$, such that 
\bna
\nor{\nabla \varphi^{(1)}) }{L^2(T_{x_{\infty}}M)}^2+ \nor{\psi^{(1)})}{L^2(T_{x_{\infty}}M)}^2\geq \nor{\nabla \varphi ^{(1)})}{L^2(T_{x_{\infty}}M)}^2 \geq \frac{1}{2} \delta(\underline{v}) 
\ena 
and 
\bna
D_h^{(1)} (v_h,\partial_t v_h)(t_h^{(1)})\underset{h\to 0}{\rightharpoonup} (\varphi^{(1)},\psi^{(1)}).
\ena
The existence of the weak limit $\psi^{(1)}$ (up to a subsequence) is ensured by the boundedness in $L^2(\R^3)$ of $\partial_t v_h$ (considered in a coordinate chart) by conservation of energy.

Then, we choose $p_h^{(1)}$ as the damped linear concentrating profile associated with\\ $[(\varphi^{(1)},\psi^{(1)}),\underline{h},\underline{x}^{(1)},\underline{t}^{(1)}]$ (actually, we pick one representant in the equivalence class modulo sequences converging to $0$ in the energy space as in Definition \ref{defconcdata}). Remark here that the assumption $t_h^{(1)}\in [-T,T]$ ensures $t_{\infty}^{(1)}\in [-T,T]$, which will always be the case for all the concentrating waves we consider. Then, we give a lemma that will be the main step to the orthogonality of energies.
\begin{lemme}
\label{lmorthNRJ}
Let $w_h^{(1)}=v_h-p_h^{(1)}$. Then,
\bna
\nor{(v_h,\partial_t v_h)(t)}{\HutL}^2=\nor{(p_h^{(1)},\partial_t p_h^{(1)})(t)}{\HutL}^2+\nor{(w_h^{(1)},\partial_t w_h^{(1)})(t)}{\HutL}^2+\petito{1}
\ena
where the $\petito{1}$ is uniform for $t$ in bounded intervals.
\end{lemme}
\begin{proof} 
We first compute the energy at time $t_h^{(1)}$. We denote $B$ the bilinear form associated with the energy : 
\bna
B(a,b)=\int_M a ~\overline{b}+\nabla a \cdot \nabla \overline{b}+\partial_t a~ \partial_t\overline{b} 
\ena
We have to prove 
\bna
B\left((p_h^{(1)}(t_h^{(1)}),w_h^{(1)}(t_h^{(1)})\right)=B\left(p_h^{(1)}(t_h^{(1)}),v_h(t_h^{(1)})-p_h^{(1)}(t_h^{(1)})\right)=\petito{1}
\ena
By weak convergence to $0$ in $H^1$ of $v_h$, $p_h^{(1)}$ and $w_h^{(1)}$, we can omit the term $\int_M a ~\overline{b}$ of $B$. By construction and Lemma \ref{lmconcweak}, we have $D_h^{(1)} (v_h,\partial_t v_h)(t_h^{(1)})\underset{h\to 0}{\rightharpoonup} (\varphi^{(1)},\psi^{(1)})$ and $D_h^{(1)} (p_h^{(1)},\partial_t p_h^{(1)})\underset{h\to 0}{\rightharpoonup} (\varphi^{(1)},\psi^{(1)})$. Therefore, $D_h^{(1)} (w_h^{(1)},\partial_t w_h^{(1)})(t_h^{(1)})\underset{h\to 0}{\rightharpoonup} (0,0)$. Lemma \ref{lmconserorth} gives the expected result. Remark that if $a\equiv 0$, this is just a consequence of the conservation of scalar product for solution of linear wave equation.
\end{proof}
We get the expansion of $u_h$ announced in Proposition \ref{propextrconc} by induction iterating the same process.

Let us assume that 
\bnan
&v_h(t,x)= \sum_{j=1}^l p_n^{(j)}(t,x)+w_n^{(l)}(t,x),\nonumber& \\
&\nor{(v_h,\partial_t v_h)}{\HutL}^2=\sum_{j=1}^l \nor{(p_h^{(j)},\partial_t p_h^{(j)})}{\HutL}^2 + \nor{(w_h^{(l)},\partial_t w_h^{(l)})}{\HutL}^2 +  \petito{1}, \textnormal{ uniformly in  } t\textnormal{, as } h\to 0 \label{horthogNRJbis}&
\enan
and where $p_h^{(j)}$ is a linear damped concentrating wave, associated with data $[(\varphi^{(k)},\psi^{(k)}),\underline{h},\underline{x}^{(k)},\underline{t}^{(k)}]$ mutually orthogonal.

We argue as before : we can assume $\delta(\underline{w}^{(l)})>0$ and we can pick $(\varphi^{(l+1)},\psi^{(l+1)}),\underline{x}^{(l+1)},\underline{t}^{(l+1)}$ such that : 
\bnan
\nor{\nabla \varphi^{(l)} }{L^2(T_{x_{\infty}^{(l+1)}}M)}^2+ \nor{\psi^{(l)}}{L^2(T_{x_{\infty}^{(l+1)}}M)}^2\geq \frac{1}{2} \delta(\underline{w}^{(l)})\label{inegdeltasum}\\
D_h^{(l+1)}(w_h^{(l)},\partial_tw_h^{(l)})(t_h^{(l+1))}) \underset{h\to 0}{\rightharpoonup} (\varphi^{(l+1)},\psi^{(l+1)}). \nonumber
\enan
and we define $p_h^{(l+1)}$ as a linear damped concentrating wave, associated with data\\ $[(\varphi^{(l+1)},\psi^{(l+1)}),\underline{h},\underline{x}^{(l+1)},\underline{t}^{(l+1)}]$. Again, Lemma \ref{lmorthNRJ} applied to $w_h^{(l)}$ and $p_h^{(l+1)}$ implies estimates (\ref{horthogNRJ}) with $w_h^{(l+1)}=w_h^{(l)}-p_h^{(l+1)}$. 

Let us now deal with estimate (\ref{hresultL6}). Lemma \ref{lmnormconcwave} combined with energy estimates gives for some $C>0$ only depending on $T$ and $a$
\bna
\nor{\nabla \varphi^{(j)} }{L^2(T_{x_{\infty}^{(j)}}M)}^2+ \nor{\psi^{(j)}}{L^2(T_{x_{\infty}^{(j)}}M)}^2 \leq C\nor{(p_h^{(j)},\partial_t p_h^{(j)})_{t=0}}{\HutL}^2 +\petito{1}.
\ena
From this and estimate (\ref{horthogNRJ}), we infer
\bna
\sum_{j=1}^l \left(\nor{\nabla \varphi^{(j)} }{L^2(T_{x_{\infty}^{(j)}}M)}^2+ \nor{\psi^{(j)}}{L^2(T_{x_{\infty}^{(j)}}M)}^2\right) \leq C\underset{h\to 0}{\varlimsup} \nor{(u_h,\partial_t u_h)}{\HutL}^2 \leq C.
\ena
So, the series of general term $\left(\nor{\nabla \varphi^{(j)} }{L^2(T_{x_{\infty}^{(j)}}M)}^2+ \nor{\psi^{(j)}}{L^2(T_{x_{\infty}^{(j)}}M)}^2\right)$ converges. Using estimate (\ref{inegdeltasum}), we get
\bna
\underset{l\to \infty}{\lim}\delta(\underline{w}^{(l)})=0. 
\ena
Lemma \ref{lmtrackLiL6} yields
\bna
\underset{h\to 0}{\varlimsup} \nor{w_h^{(l)}}{L^{\infty}([-T,T],L^6(M)} \tend{l}{\infty}0.
\ena
This completes the proof of the first part of Proposition \ref{propextrconc}. Let us now deal with the orthogonality result.
We will need the following two lemmas.
\begin{lemme}
\label{lmcvceproche}
Let $(\underline{x}^{(1)},\underline{t}^{(1)})\not\perp_h (\underline{x}^{(2)},\underline{t}^{(2)})$. Let $v_h$ be an h-oscillatory sequence solution of the damped linear wave equation such that 
\bnan
D_{h}^{(1)} (v_h,\partial_t v_h)(t_h^{(1)})\underset{h\to 0}{\rightharpoonup} (\varphi^{(1)},\psi^{(1)})\label{cvcefaible1}.
\enan
Then, there exists $(\varphi^{(2)},\psi^{(2)})$ such that, up to a subsequence
\bnan
D_{h}^{(2)} (v_h,\partial_t v_h)(t_h^{(2)})\underset{h\to 0}{\rightharpoonup} (\varphi^{(2)},\psi^{(2)})\label{cvcefaible2}.
\enan
Moreover, we have 
\bnan
\nor{(\varphi^{(1)},\psi^{(1)})}{\HutL_{x_{\infty}}}=\nor{(\varphi^{(2)},\psi^{(2)})}{\HutL_{x_{\infty}}}.\label{conservnormaprox}
\enan
 
\end{lemme}
\begin{proof}First, we assume $\underline{x}^{(1)}=\underline{x}^{(2)}$. 
By translation in time, we can assume $\underline{t}^{(1)}=0$. The non orthogonality assumption yields, up to extraction, $t_h^{(2)}/h=C+\petito{1}$ with $C$ constant.

 Let $(\varphi,\psi)\in \HutL_{\infty}$ arbitrary and $p_h$ the linear damped concentrating wave associated with\\ $[(\varphi,\psi),\underline{h},\underline{x}^{(1)},0]$. We use the equivalent definition stated in Lemma \ref{defweaktrackequiv} : (\ref{cvcefaible1}) is equivalent to 
\bna
\int_M \nabla v_h(0)\cdot \nabla p_h(0)&\tend{n}{\infty} &\int_{T_{x_{\infty}}M} \nabla \varphi^{(1)} \cdot \nabla \varphi\\
\int_M \partial_t v_h(0) \partial_t p_h(0)&\tend{n}{\infty} &\int_{T_{x_{\infty}}M} \psi^{(1)} \psi.
\ena
As both $v_h$ and $p_h$ are solutions of the damped wave equation on $M$ and $t_h^{(2)}\tend {h}{0}0$, we have by equicontinuity (see Lemma \ref{lmmeasuredamped}).
\bna
\int_M \nabla v_h(t_h^{(2)})\cdot \nabla p_h(t_h^{(2)})+\int_M \partial_t v_h(t_h^{(2)}) \partial_t p_h(t_h^{(2)})\tend{n}{\infty}\int_{T_{x_{\infty}}M} \nabla \varphi^{(1)} \cdot \nabla \varphi+\int_{T_{x_{\infty}}M} \psi^{(1)} \psi.
\ena
Let $v$, $w$ satisfying on $T_{x_{\infty}}M$
\bna
\Box_{\infty}v=0,&(v,\partial_t v)_{\left|t=0\right.} &=(\varphi^{(1)},\psi^{(1)}) \\
\Box_{\infty}w=0,&(w,\partial_t w)_{\left|t=0\right.} &=(\varphi,\psi).
\ena
Conservation of the scalar product yields
\bna
\int_{T_{x_{\infty}}M} \nabla \varphi^{(1)} \cdot \nabla \varphi+\int_{T_{x_{\infty}}M} \psi^{(1)} \psi=\int_{T_{x_{\infty}}M} \nabla v(C) \cdot \nabla w(C)+\int_{T_{x_{\infty}}M}\partial_tv(C) \partial_t w(C).
\ena
But according to Corollary \ref{corconclintproche}, $(p_h,\partial_t p_h)_{|t=t_h^{(2)}}$ is a concentrating data according to\\
$[(w(C),\partial_t w(C)),\underline{h},\underline{x}^{(1)}]$. Since the wave equation is reversible and $(\varphi,\psi)$ is arbitrary, we have proved that for any concentrating data $(f_h,g_h)$ associated with $[(\tilde{\varphi}, \tilde{\psi},\underline{h},\underline{x}^{(1)}]$, we have 
\bna
\int_M \nabla v_h(t_h^{(2)})\cdot \nabla f_h+\int_M \partial_t v_h(t_h^{(2)}) g_h\tend{n}{\infty}\int_{T_{x_{\infty}}M} \nabla v(C) \cdot \nabla \tilde{\varphi}+\int_{T_{x_{\infty}}M}\partial_tv(C) \tilde{\psi}.
\ena
This gives the result for $x_h^{(1)}=x_h^{(2)}$ by taking $(\varphi^{(2)},\psi^{(2)})=(v(C),\partial_t v(C))$ which satisfies (\ref{conservnormaprox}) by conservation of the energy.

 In the general case $\underline{x}^{(1)}\not\perp_h \underline{x}^{(2)}$, we have in a local coordinate chart and up to a subsequence $x_h^{(2)}=x_h^{(1)}+(\vec{D}+\petito{1})h$ where $\vec{D}\in T_{x_{\infty}}M$ is a constant vector. We remark that if a bounded sequence $(f_h,g_h)$ satisfies $D_{h}^{(1)} (f_h,g_h)\underset{h\to 0}{\rightharpoonup} (\varphi,\psi)$, it also fulfills $D_{h}^{(2)} (f_h,g_h)\underset{h\to 0}{\rightharpoonup} (\varphi(.+\vec{D}),\psi(.+\vec{D}))$.
\end{proof}
We will also need the following lemma which is the analog of Lemma 3.7 of \cite{PGGall2001}. We keep the notation of the algorithm of extraction for further use.
\begin{lemme}
\label{lmconcorth}
Let $\left\{j,j'\right\}\in \left\{1,\cdots,K\right\}^2$ be such that
\bna
(\underline{x}^{(j)},\underline{t}^{(j)})\not\perp_h (\underline{x}^{(K+1)},\underline{t}^{(K+1)}) \textnormal{ and }(\underline{x}^{(j)},\underline{t}^{(j)}) \perp_h (\underline{x}^{(j')},\underline{t}^{(j')}).
\ena
Then, $D_h^{(K+1)}(w_h^{(K+1)},\partial_t w_h^{(K+1)})(t_h^{(K+1)}) \rightharpoonup 0$ implies $D_h^{(j)}(w_h^{(K+1)},\partial_t w_h^{(K+1)})(t_h^{(j)}) \rightharpoonup 0$.\\
Moreover, if we assume $\left|t_{\infty}^{(j)}-t_{\infty}^{(j')}\right|<T_{focus}$ (see Definition \ref{defcouplefocus}), then
$D_h^{(j)}(p_h^{(j')},\partial_t p_h^{(j')})(t_h^{(j)}) \rightharpoonup (0,0)$ for any concentrating wave $p_h^{(j')}$ associated with $[(\varphi^{(j')},\psi^{(j')}),\underline{h},\underline{x}^{(j')},\underline{t}^{(j')}]$.
\end{lemme}
\begin{proof}
The first result is a particular case of Lemma \ref{lmcvceproche}. The proof of the second part is very similar to Lemma 3.7 of \cite{PGGall2001}. To simplify the notation, we can assume by translation in time that $t_h^{(j')}=0$. We have to distinguish two cases : time and space orthogonality.

In the case of time orthogonality, that is $\left|\frac{t_h^{(j)}}{h}\right| \tend{h}{0}+\infty$, we first prove  $D_h^{1,(j)}(p_h^{(j')})(t_h^{(j)}) \rightharpoonup 0$ (recall that the exponent $1$ in $D_h^{1,(j)}$ means that we only consider the $H^1$ part of the weak limit). Thanks to the nonfocusing assumption, Lemma \ref{lmnonreconcentr} yields 
\bna
\nor{p_h^{(j')}(t_h^{(j)},.)}{L^6(M)} \tend{h}{0}0
\ena
We choose $(U,\Phi_U)$ some local chart around $x_{\infty}^{(j)}$ and $\Psi_U\in C^{\infty}_0(U)$ equals to $1$ around $x_{\infty}^{(j)}$. Then, $\nor{\Psi_U p_h^{(j')}(t_h^{(j)},.)}{L^6(M)} \tend{h}{0}0$ and $h^{\frac{1}{2}}\nor{\Psi_U p_h^{(j')}(t_h^{(j)},x_h+hx)}{L^6(\R^3)} \tend{h}{0}0$ (here, we have identified $\Psi_U p_h^{(j')}$ with its local representation in $\R^3$). In particular $h^{\frac{1}{2}}\Psi_U p_h^{(j')}(t_h^{(j)},x_h+hx)\rightharpoonup 0$ and $D_h^{1,(j)}(p_h^{(j')})(t_h^{(j)}) \rightharpoonup 0$. Now, we want to prove more precisely $D_h^{(j)}(p_h^{(j')},\partial_t p_h^{(j')})(t_h^{(j)}) \rightharpoonup 0$. Suppose $D_h^{(j)}(p_h^{(j')},\partial_t p_h^{(j')})(t_h^{(j)}) \rightharpoonup (0,\psi)$. Take $s\in \R$ arbitrary. $\tilde{t}_h^{(j)}= t_h^{(j)}+sh$ fulfills the same assumption  $\left|\frac{\tilde{t}_h^{(j)}}{h}\right| \tend{h}{0}+\infty$ and the nonfocusing property $|\tilde{t}_{\infty}^{(j)}|<T_{focus}$. So, we conclude similarly that $D_h^{1,(j)}(p_h^{(j')})(\tilde{t}_h^{(j)}) \rightharpoonup 0$. But the proof of Lemma \ref{lmcvceproche} gives that 
$D_h^{(j)}(p_h^{(j')},\partial_t p_h^{(j')})(\tilde{t}_h^{(j)}) \rightharpoonup (v,\partial_tv)(s)$ where $v$ is solution of
\bna
\Box_{\infty}v=0,&(v,\partial_t v)(0) =(0,\psi)
\ena
So, we have $v(s)=0$ for any $s\in \R$, which gives $\psi=0$ and $D_h^{(j)}(p_h^{(j')},\partial_t p_h^{(j')})(t_h^{(j)}) \rightharpoonup (0,0)$.

In the case of $t_h^{(j)} \not\perp_h  t_h^{(j')}$ and space orthogonality, Lemma \ref{lmcvceproche} allows us to assume that $t_h^{(j)} = t_h^{(j')}=0$. In local coordinates, we have
$$(p_h^{(j')},\partial_t p_h^{(j')})(0)=h^{-\frac{1}{2}}\Psi_U(x) \left(\varphi^{(j')},\frac{1}{h}\psi^{(j')}\right)\left(\frac{x-x_h^{(j')}}{h}\right)$$
If $x_{\infty}^{(j')}\neq x_{\infty}^{(j)}$, the conclusion is obvious. If it is not the case, take $g\in C^{\infty}_0(\R^3)$. For the first part, we have to estimate
\bna
\int_{\R^3}  \Psi_U^2(x_n^{(j)}+hy)\varphi^{(j')}\left(y+\frac{x_h^{(j)}-x_h^{(j')}}{h}\right)g(y)dy
\ena
which goes to $0$ as $h$ tends to $0$ because $g$ is compactly supported. The same result holds for the second part for $\partial_t p_h^{(j')}$.
\end{proof}
Let us come back to the proof of the orthogonality of cores in Proposition \ref{propextrconc}. Define : 
\bna
j_K= \max \left\{ \left.j\in \left\{1,\cdots K\right\}\right| (t_h^{(j)},x_h^{(j)})\not\perp_h (t_h^{(K+1)},x_h^{(K+1)})\right\}
\ena
assuming that such an index exists. 

We list a few consequences of our algorithm
\bnan
D_h^{(l+1)}(w_h^{(l)},\partial_t w_h^{(l)})(t_h^{(l+1)}) &\underset{h\to 0}{\rightharpoonup} &(\varphi^{(l+1)},\psi^{(l+1)}) \textnormal{ with }\varphi^{(l+1)}\neq 0  \textnormal{ if }l\leq K \label{consequ1}\\
w_h^{(l)}&=&p_h^{(l+1)}+w_h^{(l+1)}\label{consequ2}\\
w_h^{(j_K)}&=&\sum_{j=j_K+1}^{K+1}p_h^{(j)}+w_h^{(K+1)} \label{consequ3}
\enan
The definition of $p_h^{(l)}$ and Lemma \ref{lmconcweak} implies $D_h^{(l)} (p_h^{(l)},\partial_t  p_h^{(l)})(t_h^{(l)})\rightharpoonup (\varphi^{(l)},\psi^{(l)})$. Then, we get from (\ref{consequ1}) and (\ref{consequ2}) that $D_h^{(l+1)} (w_h^{(l+1)},\partial_t  w_h^{(l+1)})(t_h^{(l+1)})\rightharpoonup (0,0)$. We apply this to $l+1=j_K$ and it gives $D_h^{(K+1)} (w_h^{(j_K)},\partial_t  w_h^{(j_K)})(t_h^{(K+1)})\rightharpoonup (0,0)$ thanks to the first part of Lemma \ref{lmconcorth} and the definition of $j_K$.

The definition of $j_K$ and the second part of Lemma \ref{lmconcorth} gives $D_h^{1,(K+1)} (p_h^{(l)},\partial_t  p_h^{(l)})(t_h^{(K+1)})\rightharpoonup (0,0)$ for $j_K+1\leq l\leq K$.

To conclude, we "apply" $D_h^{1,(K+1)}$ to equality (\ref{consequ3}) and get $D_h^{1,(K+1)} w_h^{(j_K)}(t_h^{(K+1)})\rightharpoonup \varphi^{(K+1)}$ while we have just proved $D_h^{(K+1)} (w_h^{(j_K)},\partial_t  w_h^{(j_K)})(t_h^{(K+1)})\rightharpoonup (0,0)$ which is a contradiction and complete the proof of the proposition for $2T<T_{focus}$. 

In the case of $S^3$ and large times, the orthogonality result is a consequence of the orthogonality in short times and the almost periodicity. Denote 
\bna
j_K= \max \left\{ \left.j\in \left\{1,\cdots K\right\}\right| \exists  m\in \Z \textnormal{ s.t. }(t_h^{(j)}+m\pi,(-1)^m x_h^{(j)})\not\perp_h (t_h^{(K+1)},x_h^{(K+1)})\right\}.
\ena 
Then, for any $j_K+1\leq j\leq K$, we can find $m^{(j)}\in \Z $ such that
\bna
\left|t_{\infty}^{(j)}+m^{(j)}\pi-t_{\infty}^{(j_K)}\right|\leq \pi/2< T_{focus}\\
(t_h^{(j)}+m^{(j)}\pi,(-1)^{m^{(j)}} x_h^{(j)})\perp_h (t_h^{(K+1)},x_h^{(K+1)}).
\ena 
and we denote $m^{(j_K)}\in \Z$ such that $(t_h^{(j_K)}+m^{(j_K)}\pi,(-1)^{m^{(j_K)}} x_h^{(j_K)})\not\perp_h (t_h^{(K+1)},x_h^{(K+1)})$. \\
We remark that $p_h^{(j)}\left(t_h^{(j)}+m^{(j)}\pi,.\right)$ is still a non zero concentrating data associated with\\ $[(-1)^{m^{(j)}}(\varphi,\psi)((-1)^{m^{(j)}}.), \underline{h},(-1)^j \underline{x}]$ thanks to Lemma \ref{S3perio} (note that it is at this stage that we use $M=S^3$ and $a\equiv 0$ : it is the only case where we are able to describe this phenomenon of reconcentration). So, we are in the same situation as before, and we get a contradition. 

This completes the proof of Proposition \ref{propextrconc}.\end{proof}
\bnp[Proof of Theorem \ref{thmdecompositionL}] We only have to combine the both decompositions we made. Denote $v_n^{j}$ (and the rest $\rho_n^{(l)}$ the $h^{(j)}_n$ oscillatory component obtained by decomposition (\ref{decomplinscale}) and $p_n^{(j,\alpha)}$ the concentrating waves obtained from decomposition (\ref{decomplincore}) (and the rest $w_n^{(j,A_j)}$). We enumerate them by the bijection $\sigma$ from $\N^2$ into $\N$ defined by
\bna
\sigma(j,\alpha)<\sigma(k,\beta) \textnormal{ if } j+\alpha<k+\beta \textnormal{ or  } j+\alpha= k+\beta \textnormal{ and } j<k.
\ena

For $l$ and $A_j$ fixed, $1\leq j\leq l$, the rest can be written 
\bna
w_n^{(l,A_1,\cdots,A_l)}=\rho_n^{(l)}+\sum_{j=1}^l w_n^{(j,A_j)} .
\ena
Let $\varepsilon>0$. To get the result, it suffices to prove that for $l_0$ large enough, $\nor{w_n^{(l,A_1,\cdots,A_l)}}{L^{\infty}(L^6)}\leq \varepsilon $ for all $(l,A_1,\cdots,A_l)$ satisfying $l\geq l_0$ and $\sigma(j,A_j)\geq \sigma(l_0,1)$ .
 
(\ref{orthlinthm}) can easily be deduced from the same orthogonality result in the both other decomposition. In particular, it gives that the series of general term $\sum_{(j,\alpha)}\limsu{n}{\infty}\nor{(p_n^{(j,\alpha)},\partial_t p_n^{(j,\alpha)})_{t=0}}{\HutL}^2$ is convergent. In particular, we can find $l_0$ large enough such that we have 
\bnan
\sum_{\sigma(j,\alpha)>\sigma(l_0,1)}\limsu{n}{\infty}\nor{(p_n^{(j,\alpha)},\partial_t p_n^{(j,\alpha)})_{t=0}}{\HutL}^2 \leq \varepsilon \label{estimsumorteps}.
\enan
Moreover, for $l_0$ large enough, we have for $l\geq l_0$
\bna
\limsu{n}{\infty}\nor{\rho_n^{(l)}}{L^{\infty}(L^{6})} \leq \varepsilon.
\ena
Then, for any $l\geq l_0$, one can find one $B_l$ such that for any $1\leq j\leq l$, $\tilde{A}_j\geq B_l$ implies 
\bna
\limsu{n}{\infty}\nor{w_n^{(j,\tilde{A}_j)}}{L^{\infty}(L^{6})}\leq \varepsilon/l.
\ena
The rest can be decomposed by 
\bna
w_n^{(l,A_1,\cdots,A_l)}=\rho_n^{(l)}+\sum_{j=1}^l w_n^{(j,\max(A_j,B_l))} +S_n^{(j,A_1,\cdots,A_l))}.
\ena
where 
\bna
S_n^{(j,A_1,\cdots,A_l))}&=& \sum_{1\leq j\leq l,A_j<B_l}\left( w_n^{(j,A_j)} -w_n^{(j,B_l)}\right)=\sum_{j=1}^l \sum_{A_j<\alpha\leq B_l} p_n^{j,\alpha}.
\ena
Since $S_n^{(j,A_1,\cdots,A_l))}$ is solution of the damped wave equation, energy estimates and Sobolev embedding give
\bna
\limsu{n}{\infty}\nor{S_n^{(j,A_1,\cdots,A_l))}}{L^{\infty}(L^{6})}^2 &\leq &C\limsu{n}{\infty}\nor{(S_n^{(j,A_1,\cdots,A_l))},\partial_t S_n^{(j,A_1,\cdots,A_l))})_{t=0}}{\HutL}^2\\
&\leq &C\sum_{j=1}^l \sum_{A_j<\alpha\leq B_l} \nor{(p_n^{(j,\alpha)},\partial_t p_n^{j,\alpha})_{t=0}}{\HutL}^2.
\ena
where we have used almost orthogonality in the last estimate. But the sum is restricted to some $(j,\alpha)$ satisfying $\sigma(j,\alpha)> \sigma(j,\alpha_j)>\sigma(l_0,1)$ and is indeed smaller than $C\varepsilon$ thanks to (\ref{estimsumorteps}).

Combining our estimates, we get that $\limsu{n}{\infty}\nor{w_n^{(l,A_1,\cdots,A_l)}}{L^{\infty}(L^6)}$ is smaller than $(2+C)\varepsilon$ for all $(l,A_1,\cdots,A_l)$ satisfying $l\geq l_0$ and $\sigma(j,A_j)\geq \sigma(l_0,1)$. We get the same estimates with the $L^5(L^{10})$ norm by interpolation between $L^{\infty}(L^6)$ and $L^4(L^{12})$. The second norm being bounded by Strichartz estimates and the fact that $w_n^{(l,A_1,\cdots,A_l)}$ is uniformly bounded in the energy space.\enp

We also state a few consequences of the algorithm of Theorem \ref{thmdecompositionL} that will be used below. The following both lemmas use the notation and the assumptions of Theorem \ref{thmdecompositionL} .
\begin{lemme}
\label{lmwnlorth}
Let $2T<T_{focus}$. For any $l\in \N$ and $1\leq j\leq l$, we have, with the notation and assumptions of Theorem \ref{thmdecompositionL} 
\bna
D^{(j)}_n (w_n^{(l)},\partial_t w_n^{(l)})(t_n^{(j)}) \rightharpoonup (0,0).
\ena
\end{lemme}
\bnp
Assume $D^{(j)}_n (w_n^{(l)},\partial_t w_n^{(l)})(t_n^{(j)}) \rightharpoonup (\varphi,\psi)$. We directly use the decomposition of Theorem \ref{thmdecompositionL} to write for $L>l$
\bna
w_n^{(l)}=\sum_{i=l+1}^L p_n^{(i)}+w_n^{(L)}.
\ena
In case of scale orthogonality of $h_n^{(j)}$ and $h_n^{(i)}$, for $l+1\leq i\leq L$, we have directly $D^{(j)}_n (p_n^{(i)},\partial_t p_n^{(i)})(t_n^{(j)})\rightharpoonup (0,0)$. Otherwise, if $h_n^{(j)}=h_n^{(i)}$ and $(\underline{x}^{(j)},\underline{t}^{(j)}) \perp_h (\underline{x}^{(i)},\underline{t}^{(i)})$, Lemma \ref{lmconcorth} gives the same result. Therefore, $D^{(j)}_n (w_n^{(L)},\partial_t w_n^{(L)})(t_n^{(j)}) \rightharpoonup (\varphi,\psi)$. Since $\limsu{n}{\infty}\nor{w_n^{(L)}}{L^{\infty}([-T,T],L^6)}\tend{L}{\infty}0$, we have $\varphi=0$. We finish the proof as in Lemma \ref{lmconcorth}. We use the same argument for times $t_n^{(j)}+sh_n^{(j)}$ and get $\psi\equiv 0$ by the proof of Lemma \ref{lmcvceproche}. Remark that Lemma \ref{lmcvceproche} requires that $w_n^{(l)}$ is strictly $h_n^{(j)}$-oscillatory, but this can be easily avoided by decomposing $w_n^{(l)}=f_n +g_n$ with $f_n$ ($h_n^{(j)}$)-oscillatory and $g_n$ ($h_n^{(j)}$)-singular.
\enp
\begin{lemme}
With the notation and assumptions of Theorem \ref{thmdecompositionL}, we have, for any $j\in \N$
\label{lmnorleq}
\bna
\underset{n\to \infty}{\varlimsup} \nor{p_n^{(j)}}{L^5([-T,T],L^{10})}\leq C \underset{n\to \infty}{\varlimsup}\nor{v_n}{L^5([-T,T],L^{10})}
\ena
where $C$ only depends on  the manifold $M$.
\end{lemme}
\begin{proof}
We first assume $2T<T_{focus}$.
Actually, in the case of $\R^3$, the result is proved using the fact that the $p_n^{(j)}$ are some concentration of some weak limit of a dilation of $v_n$. The proof for a manifold follows the same path with a little more care due to the fact that dilation only have a local meaning.
  
For any $\varepsilon>0$, we prove 
\bna
\underset{n\to \infty}{\varlimsup} \nor{p_n^{(j)}}{L^5([-T,T],L^{10})}\leq C\underset{n\to \infty}{\varlimsup}\nor{v_n}{L^5([-T,T],L^{10})}+C\varepsilon. 
\ena
We use the decomposition of Theorem \ref{thmdecompositionL} and choose $l\geq j$ large enough such that
\bna
\underset{n\to \infty}{\varlimsup} \nor{w_n^l}{L^5([-T,T],L^{10})}\leq \varepsilon. 
\ena
Let $\Psi_U$ be a cut off function related to local charts $(U,\Phi_U)$ such that $\Psi_U(x)=1$ around $x_{\infty}^j$ and $\Psi_U(x)=0$ around any $x_{\infty}^i\neq x_{\infty}^j$. 

For each $1\leq i\leq l$, we decompose $[-T,T]=I_{n,i}^{1,\Lambda}\cup I_{n,i}^{2,\Lambda}\cup I_{n,i}^{3,\Lambda}$ according to Lemma \ref{lmnonreconcentr}. 

For any $i$ such that $x_{\infty}^i= x_{\infty}^j$, for $\Lambda$ large enough, we have 
\bnan
\label{linpetit}
\underset{n\to \infty}{\varlimsup}\nor{p_n^{(i)}}{L^5(I_{n,i}^{1,\Lambda}\cup I_{n,i}^{3,\Lambda},L^{10})}\leq \varepsilon/l. 
\enan
Moreover, Lemma \ref{lmlinproche} yields for $\Lambda$ large enough
\bnan
\label{solconc}
\underset{n\to \infty}{\varlimsup}\nor{p_n^{(i)}-v_n^{(i)}}{L^5(I_{n,i}^{2,\Lambda},L^{10})}\leq \varepsilon/l
\enan
where $v_n^{(i)}(t,x)=\frac{1}{\sqrt{h_n^{(i)}}}\Phi_U^{*}\Psi_U (x)v^{(i)}\left(\frac{t-t_n^{(i)}}{h_n^{(i)}},\frac{x-x_n^{(i)}}{h_n^{(i)}}\right)$ on a coordinate patch and $v^{(i)}$ solution of 
\begin{eqnarray}
\left\lbrace
\begin{array}{rcl}
\Box_{x_\infty^j} v^{(i)}&=&0 \quad \textnormal{on}\quad \R \times T_{x_{\infty}^j}M\\
(v^{(i)}(0),\partial_t v^{(i)}(0))&=&(\varphi^{(i)},\psi^{(i)}) .
\end{array}
\right.
\end{eqnarray}
Thanks to (\ref{linpetit}) and (\ref{solconc}), the conclusion of the lemma will be obtained if we prove
\bna
\nor{v^{(j)}}{L^5(\R,L^{10}(T_{x_{\infty}^j}M))} \leq \underset{n\to \infty}{\varlimsup}\nor{v_n}{L^5([-T,T],L^{10})}+C\varepsilon.
\ena 
We argue by duality. Take $f \in C^{\infty}_0(\R \times T_{x_{\infty}^j}M)$ with $\nor{f}{L^{5/4}(\R,L^{10/9})}=1$. 

From now on, we work in local coordinates around $x_{\infty}^{(j)}$ and we will not distinguish a function defined on $U\subset M$ with its representant in $\R^3\approx T_{x_{\infty}^j}M$. Denote $W^j$ the operator defined on functions on $\R_t\times \R^3$ by
$$W^j g(s,y):=\sqrt{h_n^{j}}g(t_n^j+h_n^js,t_n^j+h_n^js).$$
The definition of $v_n^{(j)}$ in local coordinates yields
 \bna
 \int_{\R\times \R^3}(W^j 1_{[-T,T]}v_n^{(j)} )f \tend{n}{\infty} \int_{\R\times \R^3}v^{(j)} f.
 \ena
On the other hand 
 \bna
 \int_{\R\times \R^3}(W^j \Psi_U 1_{[-T,T]}p_n^j )f = \int_{\R\times \R^3}W^j \left[\Psi_U 1_{[-T,T]}(v_n -\sum_{x_{\infty}^{(i)}\neq x_{\infty}^{(j)}}p_n^i- \sum_{x_{\infty}^{(i)}= x_{\infty}^{(j)},i\neq j}p_n^i -w_n^l)\right]f.
 \ena
For any $1\leq i\leq l$, with $x_{\infty}^{(i)}\neq x_{\infty}^{(j)}$, using again Lemma \ref{lmnonreconcentr} and \ref{lmlinproche} and the fact that we can choose $\Psi_U$ with $\Psi_U(x_{\infty}^{(i)})=0$, we easily get
\bna
\underset{n\to \infty}{\varlimsup} \nor{\Psi_U p_n^{(i)}}{L^5([-T,T],L^{10})}=0. 
\ena
So for $n$ large enough
  \bna
 \left|\int_{\R\times \R^3}(W^j \Psi_U p_n^{(j)} )f \right|\leq C\left(\nor{v_n}{L^{5}([-T,T],L^{10})}+2\varepsilon \right)+ \left|\int_{\R\times \R^3}W^j \left[\Psi_U 1_{[-T,T]}\sum_{x_{\infty}^{(i)}= x_{\infty}^{(j)},i\neq j}p_n^{(i)}\right]f\right|
 \ena
 But for $i\neq j$, $x_{\infty}^{(i)}= x_{\infty}^{(j)}$, using (\ref{linpetit}) and then (\ref{solconc}), we have for $\Lambda$ and $n$ large enough 
 \bna
\left|\int_{\R\times \R^3}W^j \left[\Psi_U 1_{[-T,T]} p_n^{(i)}\right]f\right|& \leq & \left|\int_{\R\times \R^3}W^j \left[\Psi_U 1_{I_{n,i}^{2,\Lambda}} p_n^{(i)}\right]f\right|+\varepsilon/l\\
&\leq & \left|\int_{\R\times \R^3}W^j \left[\Psi_U 1_{I_{n,i}^{2,\Lambda}} v_n^{(i)}\right]f\right|+2\varepsilon/l.
 \ena
 These terms are actually
 \bna
 &\left|\int_{\R\times \R^3}W^j \left[\Psi_U 1_{I_{n,i}^{2,\Lambda}} v_n^{(i)}\right]f\right|=\\
 &\sqrt{\frac{h_n^{j}}{h_n^{i}}}\left|\int_{\R\times \R^3} \left[\Psi_U^2(h_n^j x+x_n^j) 1_{[\frac{t_n^i-t_n^j-\Lambda h_n^i}{h_n^j},\frac{t_n^i-t_n^j+\Lambda h_n^i}{h_n^j}]} v^{(i)}\left(\frac{th_n^j+t_n^j-t_n^i}{h_n^i},\frac{xh_n^j+x_n^j-x_n^i}{h_n^i}\right)\right]f\right|.
 \ena
Since this expression is uniformly continuous in $v^i\in L^5(\R,L^{10}(\R^3))$, we may assume $v^i$ in $C^{\infty}_0(\R\times \R^3)$. Then, if $\frac{h_n^{j}}{h_n^{i}}\tend{n}{\infty}0$, we have 
 \bna
 \left|\int_{\R\times \R^3}W^j \left[\Psi_U 1_{I_{n,i}^{2,\Lambda}} v_n^{(i)}\right]f\right|=\grando{\sqrt{\frac{h_n^{j}}{h_n^{i}}}}.
 \ena
 If $\frac{h_n^{j}}{h_n^{i}}\tend{n}{\infty}\infty$, the change of variable $s=\frac{th_n^j+t_n^j-t_n^i}{h_n^j}$, $y=\frac{xh_n^j+x_n^j-x_n^i}{h_n^i}$ gives 
 \bna
 \left|\int_{\R\times \R^3}W^j \left[\Psi_U 1_{I_{n,i}^{2,\Lambda}} v_n^i\right]f\right|=\grando{\left(\frac{h_n^{j}}{h_n^{i}}\right)^{-7/2}}.
 \ena
 If $h_n^{j}=h_n^{i}$, the space or time orthogonality yields that the integral is zero for $n$ large enough.
 
 In conclusion, for any $f \in C^{\infty}_0(\R \times \R^3)$ with $\nor{f}{L^{5/4}(\R,L^{10/9})}=1$, we have proved : 
 \bna
\left| \int_{\R\times \R^3}v^j f\right| \leq C\underset{n\to \infty}{\varlimsup}\nor{v_n}{L^{5}([-T,T],L^{10})} + C\varepsilon
 \ena
 This gives the expected result by duality.

The case of $S^3$ is proved by considering subintervals of length smaller than $T_{focus}$ where the former result can be applied.
\end{proof}
\subsection{Nonlinear profile decomposition}
\subsubsection{Behavior of nonlinear concentrating waves (after S. Ibrahim)}
\label{subsectbehavNL}
In this subsection, we recall the description of nonlinear concentrating waves. As explained in the introduction, the behavior for times close to concentration is rulled by the scattering operator on $\R^3$ with a flat metric. So, we first state the existence of the wave operator on $\R^3$, following the notation of \cite{BahouriGerard}. We state it for any constant metric on the tangent plane $T_{x_{\infty}}M\approx \R^3$.
\begin{prop}[Scattering operators on $\R^3$]
Let $x_{\infty}\in M$ and $\Box_{\infty}$ the d'Alembertian operator (constant) on $T_{x_{\infty}}M\approx \R^3$ induced by the metric on $M$. To every solution of 
\bna
\left\lbrace
\begin{array}{rcl}
\Box_{\infty} v&=&0 \quad \textnormal{on}\quad \R \times T_{x_{\infty}}M\\
(v(0),\partial_t v(0))&=&(\varphi,\psi) \in \HutL _{x_{\infty}}.
\end{array}
\right.
\ena
there exists a unique strong solution $u_{\pm}$ of
\bna
\left\lbrace
\begin{array}{rcl}
&\Box_{\infty} u_{\pm}=-|u_{\pm}|^4u_{\pm} \quad \textnormal{on}\quad \R \times T_{x_{\infty}}M\\
&\limvar{t}{\pm\infty} \nor{(v-u_{\pm},\partial_t (v-u_{\pm}))(t)}{\HutL _{x_{\infty}}}=0.
\end{array}
\right.
\ena
The wave operators
\bna
\Omega_{\pm} : (v,\partial_t v)_{t=0} \mapsto (u_{\pm},\partial_t u_{\pm})_{t=0}
\ena
are bijective from $\HutL_{x_{\infty}}$ onto itself. 

The scattering operator $S$ is defined as $S=(\Omega_+)^{-1}\circ \Omega_{-}$.
\label{propscattering} 
\end{prop}
The analysis of nonlinear concentrating waves computed by S. Ibrahim in \cite{ibrahim2004gon} shows that there are three different periods to be considered : before, during and after the time of concentration. Roughly speaking, for times close to the concentrating time, the solution is closed to nonlinear concentrating waves on $\R^3$ with flat metric and without damping, as described in Bahouri-G\'erard \cite{BahouriGerard}:  in the fast time $h_nt$, it follows the scattering on $\R^3$. Before and after the time of concentration, the nonlinear concentrating wave is "close" to some linear damped concentrating waves as defined in Table \ref{tableauconc} below. This is precised in the following theorem whose proof can be found in S. Ibrahim \cite{ibrahim2004gon}. Yet, in \cite{ibrahim2004gon}, the result is stated for an equation without damping and we give a sketch of the proof in the damped case in Section \ref{sectmodifibrahim}.
\begin{theorem}
\label{thmdescriptionNLW}
Let $\underline{v} = [(\varphi,\psi), \underline{h}, \underline{x}, \underline{t})]$ be a linear damped concentrating wave. We
denote by $\underline{u}$ its associated nonlinear damped concentrating wave (same data at $t=0$). There exist three
linear damped concentrating waves denoted by $[(\varphi_{i},\psi_{i}), \underline{h}, \underline{x}, \underline{t})]$, $i=1,2$ or $3$ such that :
for all interval $[-T,T]$ containing $t_{\infty}$, satisfying the following
non-focusing property (see Definition \ref{defcouplefocus})
\bnan
\label{nonrefocus}
mes\left(F_{x,x_{\infty},s}\right)=0 \quad \forall x\in M \textnormal{ and } s\neq 0 \textnormal{ such that } t_{\infty}+s\in [-T,T]
\enan
we have
\bnan
\label{estimbefore}
\varlimsup_n \nort{ u_n - [(\varphi_{1},\psi_{1}), \underline{h}, \underline{x}, \underline{t})]}_{I_n^{1,\Lambda}} \tend{\Lambda}{+\infty}0\\
\label{estimafter}\varlimsup_n \nort{ u_n - [(\varphi_{3},\psi_{3}), \underline{h}, \underline{x}, \underline{t})]}_{I_n^{3,\Lambda}} \tend{\Lambda}{+\infty}0
\enan
where, $I_n^{1,\Lambda}=[-T,t_n-\Lambda h_n]$ and $I_n^{3,\Lambda}=]t_n+\Lambda h_n,T]$.

Moreover, for times close to concentration $I_n^{2,\Lambda}=[t_n-\Lambda h_n,t_n+\Lambda h_n]$, we have 
\bnan
\label{estimclose}
\forall \Lambda>0,\quad \lim_n \nort{u_n - w_{n}}_{I_n^{2,\Lambda}} =0
\enan
where $w_{n}(t,x)=\Psi_U (x)\frac{1}{\sqrt{h_n}}w\left(\frac{t-t_n}{h_n},\frac{x-x_n}{h_n}\right)$ on a coordinate patch and $w$ solution of 
\begin{eqnarray}
\label{eqnnonlinIN2}
\left\lbrace
\begin{array}{rcl}
\Box_{\infty} w&=&-|w|^4w \quad \textnormal{on}\quad \R\times T_{x_{\infty}}M\\
(w(0),\partial_t w(0))&=&(\varphi_{2},\psi_{2}) .
\end{array}
\right.
\end{eqnarray}
where $\Box_{\infty}$ corresponds to the frozen metric on $T_{x_{\infty}}M$.

The different functions $(\varphi_i,\psi_i)$ are defined according to Table \ref{tableauconc}, following the notation of Proposition \ref{propscattering}. 
\begin{table}[h]
\label{tableauconc}
\begin{center}
\begin{tabular}{|c|c|c|c|}
\hline
$\lim \frac{t_h}{h} $ & $(\varphi_{1},\psi_{1})$&$(\varphi_{2},\psi_{2})$&$(\varphi_{3},\psi_{3})$.\\\hline\hline 
$-\infty $ & $\Omega_{-}^{-1}\circ\Omega_{+}(\varphi,\psi)$&$\Omega_{+}(\varphi,\psi)$& $(\varphi,\psi)$\\ \hline
$0$ & $\Omega_{-}^{-1}(\varphi,\psi)$&$(\varphi,\psi)$& $\Omega_{+}^{-1}(\varphi,\psi)$\\ \hline
$\infty $& $(\varphi,\psi)$&$\Omega_{-}(\varphi,\psi)$& $\Omega_{+}^{-1}\circ\Omega_{-}(\varphi,\psi)$\\ \hline
\end{tabular}
\end{center}
\caption{Transformation of the profile through a focus}
\end{table}
\end{theorem}

\bigskip
 
\begin{remarque}
Note that the transition from the first column to the third one represents the modification of profile due to the concentration and the concentrating functions are modified according to the scattering operator $S$. To go from the first column to the second one, we apply the operator $\Omega_-$ while we apply $\Omega_{+}^{-1}$ to get from the second to the third one.
\end{remarque}
\begin{remarque}
The behavior for times close to concentration is not written this way in the article \cite{ibrahim2004gon} of S. Ibrahim, but is a byproduct of its proof. We refer to the next section which contains a sketch of the proof.
\end{remarque}
\begin{corollaire}
\label{cornonlinoscill}
A nonlinear damped concentrating wave $q_h$ is strictly ($h$)-oscillatory with respect to $A_M$ and bounded in all Strichartz norms, uniformly on any bounded interval.
\end{corollaire}
\begin{proof}[Proof of Corollary \ref{cornonlinoscill}] The boundedness of all Strichartz norms is a counsequence of Duhamel formula and Strichartz estimates once the result is known in the case of $L^5L^{10}$. On the intervals $I_n^{1,\Lambda}$ and $I_n^{3,\Lambda}$ when $q_h$ is closed to a linear concentrating wave, the result follows from Proposition \ref{propproposcilamorti} and linear Strichartz estimates. On $I_n^{2,\Lambda}$, $q_h$ behaves like a concentration of a nonlinear solution on $T_{x_{\infty}}M$. The strict ($h$)-oscillation is obvious and the Strichartz estimates follow from global estimates on $\R^3$.
\end{proof}
In the case of $S^3$, thanks to a better knowledge of the behavior of nonlinear concentrating waves we can avoid assumption (\ref{nonrefocus}). This is Theorem 1.8 from \cite{ibrahim2004gon}. It will allow us to perform the profile decomposition for large times. 
\begin{theorem}
\label{thmdescriptionNLWS3}
Let $\underline{v} = [(\varphi,\psi), \underline{h}, \underline{x}, \underline{t})]$ be a linear (not damped, that is $a(x)\equiv0$) concentrating wave on $S^3$. We
denote by $\underline{u}$ its nonlinear associated concentrating wave (same data at $t=0$). We assume that $t_{\infty}\in ]0,\pi[$. Then, for all $j\in \Z$, we have
\bna
\varlimsup_n \nort{ u_n - [\tilde{S}^{(j)}S(\varphi,\psi), \underline{h},(-1)^j \underline{x}, \underline{t})]}_{]t_n+j\pi+\Lambda h_n,t_n+(j+1)\pi-\Lambda h_n]} \tend{\Lambda}{+\infty}0
\ena
where, $\tilde{S}=S\circ A$, $\tilde{S}^{(j)}=\tilde{S}\circ\tilde{S}\circ \dots \circ\tilde{S}$, $j$ times and $A(\varphi,\psi)(x)=-(\varphi,\psi)(-x)$.
\end{theorem}
Moreover, the cases $t_{\infty}\in ]-\pi,0[$ and $t_{\infty}=0$ can be deduced similarly to Theorem \ref{thmdescriptionNLW} with some changes on the concentration data in the same spirit as Table \ref{tableauconc}.

\subsubsection{Modification of the proof of S. Ibrahim for Theorem \ref{thmdescriptionNLW} in the case of damped equation}
\label{sectmodifibrahim}
In this subsection, we give some sketch of proof for the behavior of nonlinear damped concentrating waves announced in subsection \ref{subsectbehavNL}. These results are proved in \cite{ibrahim2004gon} in the undamped case $a(x)\equiv 0$ and so we only briefly emphasize the main necessary modifications of proof. To simplify, we only treat the case $\frac{t_n}{h_n}\tend {n}{+\infty}\infty $.
\begin{proof}[Sketch of the proof of estimate (\ref{estimbefore}) of Theorem \ref{thmdescriptionNLW} : Behavior before concentration]
The proof is exactly the same as Corollary 3.2 of \cite{ibrahim2004gon}. $w_n=u_n-v_n$ is solution of
\bna
\begin{array}{rcl}
\Box w_n+w_n+ a(x)\partial_tw_n&=&-|w_n+v_n|^4(w_n+v_n) \quad \textnormal{on}\quad I_n^{1,\Lambda}\times M\\
(w_n,\partial_t w_n)_{|t=0}&=&(0,0) .
\end{array}
\ena
Using Strichartz and energy estimates, we are able to use a bootstrap argument if $\limsu{n}{\infty}\nor{v_n}{L^5({I_n^{1,\Lambda}},L^{10})}$ is small enough. This can be achieved thanks to Lemma \ref{lmnonreconcentr} and gives the result. 
\end{proof}
\begin{proof}[Sketch of the proof of estimate (\ref{estimclose}) of Theorem \ref{thmdescriptionNLW} : Behavior for times close to concentration]
By\\ definition of $v_n$ and finite propagation speed, the main energy part of $v_n$ is concentrated near $x_{\infty}$ for times close to $t_{\infty}$. By estimate (\ref{estimbefore}), it is also the case for $u_n$. Therefore, for times $t\in [t_n-\Lambda h_n,t_n+\Lambda h_n]$, we can neglect the energy outside of a fixed open set and work in local coordinates. Moreover, in that case, we can use the norm $\nort{\cdot}_{I\times \R^3}$ instead of $\nort{\cdot}_{I}$ and use the fact that is is invariant by translation and scaling  up to a modification of the interval of time.

Denote $\tilde{u}_{n}$ (resp $\tilde{v}_n$) the rescaled function associated to $u_n$ (resp $v_n$), so that\\ $u_n(t,x)=\frac{1}{\sqrt{h_n}}\tilde{u}_{n}\left(\frac{t-t_n}{h_n},\frac{x-x_n}{h_n}\right)$. We need to prove $\limsu{n}{\infty}\nort{\tilde{u}_{n}-w}_{[-\Lambda,\Lambda]\times \R^3}\tend{\Lambda}{\infty}0$ where $w$ is solution of 
\begin{eqnarray*}
\left\lbrace
\begin{array}{rcl}
\Box_{\infty} w&=&-|w|^4w \quad \textnormal{on}\quad \R\times \R^3\\
(w,\partial_t w)_{|t=0}&=&(\varphi_{2},\psi_{2}) =\Omega_{-}(\varphi,\psi).
\end{array}
\right.
\end{eqnarray*}
By definition of $\Omega_{-}$, $w$ satisfies $\nor{(w-v,\partial_t(w-v))(t)}{\dot{H}^1\times L^2} \tend{t}{-\infty}0$ where $v$ is solution of
\begin{eqnarray}
\label{eqnsollin}
\left\lbrace
\begin{array}{rcl}
\Box_{\infty} v&=& 0\quad \textnormal{on}\quad \R \times \R^3\\
(v,\partial_t v)_{|t=0}&=&(\varphi,\psi).
\end{array}
\right.
\end{eqnarray}

Moreover, it is known that $\Omega_- (\varphi,\psi)=\limvar{s}{-\infty} U(-s)U_0(s)(\varphi,\psi)$ where $U$ and $U_0$ are the nonlinear and linear flow map. More precisely, by Lemma 3.4 of \cite{ibrahim2004gon}, we have $\nort{w_{\Lambda}-w}_{[-\Lambda,\Lambda]\times \R^3} \tend{\Lambda}{\infty}0$ where $w_{\Lambda}$ is the smooth solution of 
\begin{eqnarray*}
\left\lbrace
\begin{array}{rcl}
\Box_{\infty} w_{\Lambda}+|w_{\Lambda}|^4w_{\Lambda}&=& 0\quad \textnormal{on}\quad [-\Lambda,\Lambda]\times \R^3\\
(w_{\Lambda},\partial_t w_{\Lambda})_{|t=-\Lambda}&=&\chi_{\Lambda}(v,\partial_t v)_{|t=-\Lambda}.
\end{array}
\right.
\end{eqnarray*}
where $\chi_{\Lambda}$ is an appropriate familly of smoothing operator. So, we are left to prove \\$\limsu{n}{\infty} \nort{\tilde{u}_{n}-w_{\Lambda}}_{[-\Lambda,\Lambda]\times \R^3}\tend{\Lambda}{\infty}0$.

We introduce the auxiliary family of functions $\tilde{u}_{n}^{\Lambda}$ solution of
\begin{eqnarray*}
\left\lbrace
\begin{array}{rcl}
\Box_n \tilde{u}_{n}^{\Lambda}+h_n^2\tilde{u}_{n}^{\Lambda}+|\tilde{u}_{n}^{\Lambda}|^4\tilde{u}_{n}^{\Lambda} &=&-h_n a(h_nx+x_n)\partial_t\tilde{u}_{n}^{\Lambda}\quad \textnormal{on}\quad [-\Lambda,\Lambda]\times \R^3\\
(\tilde{u}_{n}^{\Lambda},\partial_t \tilde{u}_{n}^{\Lambda})_{|t=-\Lambda}&=& (\tilde{v}_{n},\partial_t \tilde{v}_{n})_{|t=-\Lambda}.
\end{array}
\right.
\end{eqnarray*}
where we have denoted $\Box_n$ the dilation of the operator $\Box$. So it can be written $\Box_n =\partial_t^2-\sum_{i,j}g^{ij}(h_nx+x_n) \partial_{ij}+h_nV(h_nx+x_n)\cdot \nabla $ where $V$ is a smooth vector field (note that it is only defined in an open set of size $\grando{h_n^{-1}}$ but it is also the case for $\tilde{u}_{n}$, $\tilde{u}_{n}^{\Lambda}$ and $\tilde{v}_{n}$, we omit the details). 
The proof is complete if we prove
\bnan
\limsu{n}{\infty}\nort{\tilde{u}_{n}^{\Lambda}-w^{\Lambda}}_{[-\Lambda,\Lambda]\times \R^3}\tend{\Lambda}{\infty}0 \label{cvgeNL1}
\enan
and 
\bnan
\limsu{n}{\infty}\nort{\tilde{u}_{n}^{\Lambda}-\tilde{u}_{n}}_{[-\Lambda,\Lambda]\times \R^3}\tend{\Lambda}{\infty}0\label{cvgeNL2}.
\enan 
We begin with (\ref{cvgeNL1}). $r_{n,\Lambda}=\tilde{u}_{n}^{\Lambda}-w^{\Lambda}$ is solution of 
\begin{eqnarray*}
\left\lbrace
\begin{array}{rcl}
\Box_n r_{n,\Lambda}+h_n^2r_{n,\Lambda}+h_n a(h_nx+x_n)\partial_tr_{n,\Lambda}&=&|w_{\Lambda}|^4w_{\Lambda}-|r_{n,\Lambda}+w_{\Lambda}|^4(r_{n,\Lambda}+w_{\Lambda})\\
&&-h_n^2w_{\Lambda}-h_n a(h_nx+x_n)\partial_tw_{\Lambda} +(\Box_{\infty}-\Box_n)w_{\Lambda}\\
(r_{n,\Lambda},\partial_t r_{n,\Lambda})_{|t=-\Lambda}&=&(\tilde{v}_n-\chi_{\Lambda}v,\partial_t (\tilde{v}_n-\chi_{\Lambda}v))_{|t=-\Lambda} .
\end{array}
\right.
\end{eqnarray*}
A quick scaling analysis easily yields that the operator $\Box_n +h_n^2+h_n a(h_nx+x_n)\partial_t$ satisfies the same Strichartz and energy estimates as $\Box +1+a(x)\partial_t$ for some times of order $\Lambda$. Moreover, following the same argument as Lemma 2.1 of \cite{ibrahim2004gon}, we get that for fixed $\Lambda$
\bna
\label{dampingdisappear}\limsu{n}{\infty}\nor{-h_n^2w_{\Lambda}-h_n a(h_nx+x_n)\partial_tw_{\Lambda} +(\Box_{\infty}-\Box_n)w_{\Lambda}}{L^1([-\Lambda,\Lambda],L^2)} =0.
\ena
Thanks to Lemma \ref{lmlinproche}, we know that $\limsu{n}{\infty}\left\|(\tilde{v}_n-\chi_{\Lambda}v,\partial_t (\tilde{v}_n-\chi_{\Lambda}v))(-\Lambda)\right\|_{\dot{H^1}\times L^2}$ can be made arbitrary small for large $\Lambda$.
Strichartz and energy estimates give for any $\eta>-\Lambda$
\bna
\nort{r_{n,\Lambda}}_{[-\Lambda,\eta]\times \R^3}&\leq &\left\|(\tilde{v}_n-\chi_{\Lambda}v,\partial_t (\tilde{v}_n-\chi_{\Lambda}v))(-\Lambda)\right\|_{\dot{H^1}\times L^2}\\
&&+\nor{-h_n^2w_{\Lambda}-h_n a(h_nx+x_n)\partial_tw_{\Lambda} +(\Box_{\infty}-\Box_n)w_{\Lambda}}{L^1([-\Lambda,\eta],L^2)}\\
&&+\nor{r_{n,\Lambda}}{L^5([-\Lambda,\eta],L^{10})}^5+\nor{r_{n,\Lambda}}{L^5([-\Lambda,\eta],L^{10})}\nor{w_{\Lambda}}{L^5([-\Lambda,\eta],L^{10})}^4.
\ena 
If $\nor{w_{\Lambda}}{L^5([-\Lambda,\eta],L^{10})}$ is small enough, a bootstrap gives (\ref{cvgeNL1}) on $[-\Lambda,\eta]$. We can iterate the process by dividing $[-\Lambda,\Lambda]$ in a finite number of intervals where the bootstrap can be performed. 

For (\ref{cvgeNL2}), we observe that $\tilde{u}_{n}^{\Lambda}$ and $\tilde{u}_{n}$ are solutions of the same equation but with different initial data which satisfy thanks to estimate (\ref{estimbefore})
$$\limsu{n}{\infty} \nor{(\tilde{u}_{n}^{\Lambda}-\tilde{u}_{n},\partial_t(\tilde{u}_{n}^{\Lambda}-\tilde{u}_{n}))(-\Lambda)}{\HutL}=\limsu{n}{\infty} \nor{(\tilde{v}_{n}-\tilde{u}_{n},\partial_t(\tilde{v}_{n}-\tilde{u}_{n}))(-\Lambda)}{\HutL}\tend{\Lambda}{\infty}0.$$
Then, Strichartz and energy estimates allow us to use a boot strap argument on subintervals $I$ such that $\nor{\tilde{u}_{n}^{\Lambda}}{L^5(I,L^{10})}$ is small. (\ref{cvgeNL1}) allows to complete the proof.  
\end{proof}
\subsubsection{Proof of the decomposition}
This subsection is devoted to the proof of Theorem \ref{thmdecompositionNL}.

Let us define the function $\beta$ in the following way : 
\bna
\forall \omega \in \C, \quad \beta(\omega)\overset{def}:=|\omega|^4\omega.
\ena
\begin{prop}
\label{proporthogL3}
Let $0<2T<T_{focus}$ (see Definition \ref{defcouplefocus}). Let $p_n^{(j)}$, $1\leq j\leq l$, linear damped concentrating waves, associated with data $[(\varphi^{(j)},\psi^{(j)}),\underline{h}^{(j)},\underline{x}^{(j)},\underline{t}^{(j)}]$ (we can have $h^{(j)}=1$ for one of it), which are orthogonal according to Definition \ref{deforthog} and such that $t_{\infty}^{(j)} \in [-T,T]$. Denote $q_n^{(j)}$ the associated nonlinear damped concentrating waves (same data at $t=0$).

Then, we have 
\bnan
\varlimsup_{n \to \infty} \nor{\beta\left(\sum_{j=1}^l q_n^{(j)}\right)-\sum_{j=1}^l \beta(q_n^{(j)})}{L^1([-T,T],L^2)} =0.
\enan
\end{prop}
\begin{proof}
We follow closely Lemma 4.2 of \cite{PGGall2001}.
\bna
\nor{\beta\left(\sum_{j=1}^l q_n^{(j)}\right)-\sum_{j=1}^l \beta(q_n^{(j)})}{L^1([-T,T],L^2)}\leq \sum_{1\leq j_1,\cdots,j_5\leq l}\nor{\prod_{k=1}^5 q_n^{(j_k)}}{L^1([-T,T],L^2)}
\ena
where at least two $q_n^{(j_k)}$ are different.
In the case of othogonality of scales, we use H\"older inequality
\bna
\nor{\prod_{k=1}^5 q_n^{(j_k)}}{L^1([-T,T],L^2)}\leq C \nor{q_n^{1}q_n^{2}}{L^{\infty}([-T,T],L^3)}\prod_{k=3}^5\nor{ q_n^{(j_k)}}{L^3([-T,T],L^{18})}.
\ena
Then, Corollary \ref{cornonlinoscill} and Lemma \ref{lmorthL3} yield the result (note that $L^3L^{18}$ is a pair of Strichartz norm). So now, we can assume $h_n^1=h_n^2=h_n$.
By H\"older and Corollary \ref{cornonlinoscill}, we get
\bna
\nor{\prod_{k=1}^5 q_n^{(j_k)}}{L^1([-T,T],L^2)}\leq C \nor{q_n^{1}q_n^{2}}{L^{5/2}([-T,T],L^{5})}\prod_{k=3}^5\nor{ q_n^{(j_k)}}{L^5([-T,T],L^{10})}\leq C \nor{q_n^{1}q_n^{2}}{L^{5/2}([-T,T],L^{5})}.
\ena
We apply Theorem \ref{thmdescriptionNLW} to $q_n^{1}$. We obtain three couples $(\varphi^i,\psi^i)$, $i=1,2,3$ and split the interval $[-T,T]=\cup_{j=1}^3 I_n^{j,\Lambda}$. We first deal with the interval $I_{n}^{1,\Lambda}$. Denote $\underline{v_1}= [(\varphi_{1},\psi_{1}), \underline{h}, \underline{x}, \underline{t})]$ so that 
\bna
\nor{q_n^{1}q_n^{2}}{L^{5/2}(I_n^{1,\Lambda},L^{5})}\leq \nor{q_n^{1}}{L^{5}(I_n^{1,\Lambda},L^{10})}\leq C \nor{q_n^{1}-v_{1,n}}{L^{5}(I_n^{1,\Lambda},L^{10})}+\nor{v_{1,n}}{L^{5}(I_n^{1,\Lambda},L^{10})}
\ena
So, combining Theorem \ref{thmdescriptionNLW} and Lemma \ref{lmnonreconcentr} yields 
\bna
\varlimsup_n \nor{q_n^{1}q_n^{2}}{L^{5/2}(I_n^{1,\Lambda},L^{5})} \tend{\Lambda}{\infty}0.
\ena
The same result holds for $I_n^{3,\Lambda}$ and we are led with the interval $I_n^{2,\Lambda}$. In the case of time orthogonality, say $\frac{|t_n^2-t_n^1|}{h_n}\tend {n}{\infty}+\infty$, the two intervals $[t_n^1-\Lambda h_n,t_n^1+\Lambda h_n]$ and $[t_n^2-\Lambda h_n,t_n^2+\Lambda h_n]$ have empty intersection for fixed $\Lambda$ and $n$ large enough, which yields the result by the same estimates applied to $q_n^{2}$, once $\Lambda$ is chosen large enough. 

We can now assume, up to a translation in time, that $t_n^1=t_n^2$. On $I_n^{2,\Lambda}$, Theorem \ref{thmdescriptionNLW} allows us to replace $q_n^{1}$ by 
$w_{n}^1(t,x)=\Psi_U^1 (x)w^1\left(\frac{t-t_n^1}{h_n},\frac{x-x_n^1}{h_n}\right)$ on a coordinate patch where $w^1$ is solution of a nonlinear wave equation on the tangent plane $T_{x_{\infty}^{1}}M$ and similarly for $q_n^{2}$. In the first case of space orthogonality, that is $x_{\infty}^1\neq x_{\infty}^2$, the result is obvious on the interval $I_n^{2,\Lambda}$ by taking $\Psi_U^1$ and $\Psi_U^2$ with empty intersection. In the case $x_{\infty}^1=x_{\infty}^2$, we are left with the estimate of
\bna
\int_{I_n^2}\left(\int_{\R^3}\left|w_{n}^1(t,x) w_{n}^2(t,x)\right|^5\right)^{1/2}ds\leq \int_{[-\Lambda,\Lambda]}\left(\int_{\R^3}\left|w^1(t,x) w^2(t,x+\frac{x_n^1-x_n^2}{h_n})\right|^5\right)^{1/2}ds
\ena
This yields the result in the last case of space orthogonality by approximating $w^1$ and $w^2$ by compactly supported functions.
\end{proof}
In the case of the sphere, we are able to state the same result without any restriction on the time.
\begin{corollaire}
\label{proporthogL3S3}
Let $M=S^3$ and $T>0$ (eventually large). We make the same assumptions as Proposition \ref{proporthogL3}, except for the time $T$, with the additional hypothesis : \\
$[\underline{h}^{(i)},(-1)^m\underline{x}^{(i)},\underline{t}^{(i)}+m\pi]$ is orthogonal to $[\underline{h}^{(j)},\underline{x}^{(j)},\underline{t}^{(j)}]$ for any $m\in \Z$ and $i\neq j$. Moreover, we assume $a(x)\equiv 0$ (undamped equation).\\
Then, the same conclusion as Proposition \ref{proporthogL3} is true.
\end{corollaire}
\begin{proof}
We build a covering of the interval $[-T,T]$ with a finite number of intervals of length stricly less than $T_{focus}=\pi$ so that on each of this interval $I=[\alpha,\beta]$ and for any $1\leq i \leq l$, there exists at most one $m^{(i)} \in \Z$ such that $t_{\infty}^{(i)}+m^{(i)}\pi\in I$. Moreover, one can also impose $\alpha \neq t_{\infty}^{(i)}+m^{(i)}\pi$.

Therefore, $\alpha\in ]t_n^{(i)}+(m^i-1)\pi+\Lambda h_n^{(i)},t_n^{(i)}+m^{(i)}\pi-\Lambda h_n^{(i)}]$ for large fixed $\Lambda$ and $n$ large enough. Theorem \ref{thmdescriptionNLWS3} yields $\nor{ (q_n^{(i)} - v_n^{(i)},\partial_t (q_n^{(i)} - v_n^{(i)}))_{t=\alpha}}{\HutL} \tend{n}{\infty}0$ for a linear concentrating wave $v_n^{(i)}=[\tilde{S}^{m^{(i)}}S(\varphi^{(i)},\psi^{(i)}), \underline{h}^{(i)}, \underline{x}^{(i)}, \underline{t}^{(i)})]$. In each interval, we are in the same situation as in Proposition \ref{proporthogL3} which yields the desired result. 
\end{proof}
Now, we are ready for the proof of the nonlinear profile decomposition. We give it in a quite sketchy way since it is very similar to the one of \cite{BahouriGerard} or \cite{PGGall2001}. First, we obtain it in the particular case where the linear solution is small in Strichartz norm.
\begin{lemme}
\label{lmdecompNLpetit}
There exists $\delta_1>0$ such that if
\bna
\underset{n\to \infty}{\varlimsup}\nor{v_n}{L^5([-T,T],L^{10})}\leq \delta_1
\ena
then the conclusion of Theorem \ref{thmdecompositionNL} is true.
\end{lemme}
\begin{proof}
The proof is essentially the same as Lemma 4.3 of \cite{BahouriGerard}. We have to estimate the rest $r_n^{(l)}$ solution of
\bna
\left\lbrace
\begin{array}{rcl}
\Box r_n^{(l)} +r_n^{(l)}+a(x)\partial_t r_n^{(l)}&=&\beta(u)+\sum_{j=1}^l \beta(q_n^{(j)}) -\beta\left(u+\sum_{j=1}^l q_n^{(j)}+w_n^{(l)}+r_n^{(l)}\right)\\
(r_n^{(l)},\partial_t r_n^{(l)})_{t=0}&=&(0,0) .
\end{array}
\right.
\ena
We conclude as in \cite{BahouriGerard} using Proposition \ref{proporthogL3} and Lemma \ref{lmnorleq} which is not immediate on a manifold. In the case of $S^3$ and $a\equiv 0$ for large $T$, we use Corollary \ref{proporthogL3S3} instead of Proposition \ref{proporthogL3} .
\end{proof}
Once the result is obtained when Strichartz norms are small, we divide $[-T,T]$ in a finite number of intervals where the Strichartz norms are small enough. This is done in the following lemma.
\begin{lemme}
\label{lmdivisioninterval}
Let $2T<T_{focus}$. Let $\delta>0$ and $\tilde{q}_n$ be a sequence in $L^5([-T,T],L^{10}(M))$, such that 
\bna
\limsu{n}{\infty} \nor{\tilde{q}_n}{L^5([-T,T],L^{10})}\leq \delta.
\ena
Fix also $l\in \N$ and $l$ sequences of nonlinear concentrating wave $q_n^{(j)}$, $j=1,...,l$.
  
Then, for any $\delta'>\delta$, there exists $L\in \N$ such that for any $n\in \N$, we have the decomposition of $[-T,T]$ in closed intervals $I_n^{(i)}$
\bna
[-T,T]=\bigcup_{i=1}^L I_n^{(j)}, 
\ena
such that the sequence 
\bna
\Gamma_n= \sum_{j=1}^l q_n^{(j)}+\tilde{q}_n
\ena
satisfies on each interval $I_n^{(i)}$
\bna
\limsu{n}{\infty}\nor{\Gamma_n}{L^5(I_n^{(j)},L^{10})} \leq \delta'.
\ena
\end{lemme}
\bnp
We first treat the case $l=1$. We divide $[-T,T]=I_n^{1,\Lambda}\cup I_n^{2,\Lambda} \cup I_n^{3,\Lambda}$ according to Theorem \ref{thmdescriptionNLW} (one of these intervals being possibly empty). Then, a combination of estimate (\ref{estimbefore}) of Theorem \ref{thmdescriptionNLW} (comparison with linear concentrating wave) and Lemma \ref{lmnonreconcentr} (non reconcentration) gives for $\Lambda$ large enough
\bna
\limsu{n}{\infty}\nor{q_n^{(1)}}{L^5(I_n^{1,\Lambda},L^{10})} \leq \delta'-\delta.
\ena
The same result holds for $I_n^{3,\Lambda}$ and we are left with the interval $I_n^{2,\Lambda}$. Once $\Lambda$ is fixed, we can divide $[-\Lambda,\Lambda]$ in a finite number of intervals $I_{}^{(i),\Lambda}$ such that $\nor{w}{L^5(I_{}^{(i),\Lambda},L^{10}} \leq \delta-\delta'$ where $w$ is the function defined by equation (\ref{eqnnonlinIN2}) of Theorem \ref{thmdescriptionNLW}. Then, we replace each $I_{}^{(i),\Lambda}$ by $I_{n}^{(i),\Lambda}$ obtained by translation dilation. We conclude by the approximation (\ref{estimclose}) of $q_n^{(1)}$ by translation dilation of $w$ on the interval $I_n^{2,\Lambda}$.
\enp
Note that the previous lemma also applies for large times on $S^3$ with $a\equiv 0$ by doing a first decomposition of $[-T,T]$ in a finite number of intervals of length strictly less than $\pi$.
\bnp[End of the proof of Theorem \ref{thmdecompositionNL} in the general case] We choose $l\in \N$ such that $\nor{w_n^{(l)}}{}\leq \delta_1$ and use Lemma \ref{lmdivisioninterval} in order to be able to apply Lemma \ref{lmdecompNLpetit} on each interval $I_n^{(j)}$. See \cite{BahouriGerard} or, in the different context of Schrödinger equation, \cite{Keraanidefect} . \enp
\subsection{Applications}
\subsubsection{Strichartz estimates and Lipschitz bounds for the nonlinear evolution group}
\begin{prop}
\label{propapriori}
Let $T>0$ be fixed. There exist a non-decreasing function, $A:[0,\infty[ \rightarrow [0,\infty[$, such that any solution of  
\begin{eqnarray}
\label{eqnStrichartzglob}
\left\lbrace
\begin{array}{rcl}
\Box u+u+a(x)\partial_t u&=&-|u|^4u \quad \textnormal{on}\quad [-T,T]\times M\\
(u(0),\partial_t u(0))&=&(u_{0},u_1) \in \HutL\\
\end{array}
\right.
\end{eqnarray}
fulfills 
\bna
\nor{u}{L^8([-T,T],L^{8}(M)}+\nor{u}{L^5([-T,T],L^{10}(M)}+\nor{u}{L^4([-T,T],L^{12}(M)}&\leq& A(\nor{(u_{0},u_1)}{\HutL}).
\ena
\end{prop}
\bnp
The proof is exactly the same as Corollary 2 of \cite{BahouriGerard}. Using Strichartz estimates, it is enough to get the result for $L^5L^{10}$. We argue by contradiction and suppose that there exists a sequence $u_n$ of strong solutions of equation (\ref{eqnStrichartzglob}) satisfying
\bna
\sup_n \nor{(u_{0,n},u_{1,n})}{\HutL} <+\infty , \quad \nor{u_n}{L^5([-T,T],L^{10}(M)} \tend{n}{\infty}+\infty.
\ena
We apply the profile decomposition of Theorem \ref{thmdecompositionNL} to our sequence. We get a contradiction by the fact that the $L^5([-T,T],L^{10}(M))$ norm of a nonlinear concentrating wave is uniformly bounded thanks to Corollary \ref{cornonlinoscill}. This argument works for times $2T<T_{focus}$ and can be reiterated since the nonlinear energy at times $T$ can be bounded with respect to the one at time $0$ thanks to almost conservation (we can also use energy estimates once we know $u$ is uniformly bounded in $L^5L^{10}$). \enp
\begin{lemme}
\label{lmestimL2}
Let $R_0>0$ and $T>0$.
Then, there exists $C>0$ such any solution $u$ satisfying 
\begin{eqnarray}
\label{eqnlinlmL2}
\left\lbrace
\begin{array}{rcl}
\Box u+u+a(x)\partial_t u+|u|^4u&=& 0\quad \textnormal{on}\quad [-T,T]\times M\\
(u(0),\partial_t u(0))&=&(u_{0},u_1) \in \HutL\\
\nor{(u_{0},u_1)}{\HutL}&\leq &R_0.
\end{array}
\right.
\end{eqnarray}
fulfills 
\bna
\nor{(u(t),\partial_t u(t))}{L^2\times H^{-1}}\leq C \nor{(u(0),\partial_t u(0))}{L^2\times H^{-1}} \quad \forall t\in [-T,T].
\ena
\end{lemme}
\begin{proof}
Proposition \ref{propapriori} yields a uniform bound for $u$ in $L^4([-T,T],L^{12}(M))$ and so for $V=|u|^4$ in $L^1([0,T],L^{3}(M))$. We prove uniform estimates for some solutions of the linear equation
 \begin{eqnarray}
\left\lbrace
\begin{array}{rcl}
\Box u+u+a(x)\partial_t u=Vu \quad \textnormal{on}\quad [-T,T]\times M\\
(u(0),\partial_t u(0))=(u_{0},u_1) \in  L^2\times H^{-1}\\
\end{array}
\right.
\end{eqnarray}
where $V$ satisfies $\nor{V}{L^1([-T,T],L^{3}(M))}\leq A(R_0)^4$. The product of functions in $L^{\infty}([-T,T],L^2)$ and $L^1([-T,T],L^{3})$ is in $L^1([-T,T],L^{6/5})$ and so in $L^1([-T,T],H^{-1})$ by Sobolev embedding. Standard estimates yields
\bna
\nor{(u,\partial_t u)}{L^{\infty}([0,t],L^2\times H^{-1})}\leq C\nor{(u(0),\partial_t u(0))}{L^2\times H^{-1}}+C(t+\nor{V}{L^1([0,t],L^{3})})\nor{(u,\partial_t u)}{L^{\infty}([0,t],L^2\times H^{-1})}.
\ena
We can divide the interval $[-T,T]$ into a finite number of intervals $[a_i,a_{i+1}]_{i=1\cdots N}$ such that $C(t+\nor{V}{L^1([a_i,a_{i+1}],L^{3}(M))})<1/2$. $N$ depends only on $R_0$ and $T$ (not on $V$).

Then, on each of these intervals, we have
\bna
\nor{(u,\partial_t u)}{L^{\infty}([a_i,a_{i+1}],L^2\times H^{-1})}\leq 2C\nor{(u(a_i),\partial_t u(a_i))}{L^2\times H^{-1}}.
\ena
We obtain the expected result by iteration. The final constant $C$ only depends on $R_0$ and $T$ since it is also the case for $N$.
\end{proof}
\begin{corollaire}
\label{corL2choixpetit}
Let $R_0>0$ and $T>0$.
For any $\varepsilon>0$, there exists $\delta>0$ such that any solution $u$ satisfying (\ref{eqnlinlmL2}) and $\nor{(u_{0},u_1)}{L^2\times H^{-1}}\leq \delta$ satisfies
\bna
\nor{(u(T),\partial_t u(T))}{L^2\times H^{-1}}\leq \varepsilon.
\ena
\end{corollaire}
We will also need the following lemma which states the local uniform continuity of the flow map. Note that it can be proved to be locally Lipschitz with a slightly more complicated argument (see Corollary 2 of \cite{PGGall2001}). We will note need this for our purpose.
\begin{lemme}
\label{approxNL}
Let $u_n$, $\tilde{u}_n$  be two sequences of solutions of
\begin{eqnarray*}
\left\lbrace
\begin{array}{rcl}
\Box u_n+u_n+|u_n|^4u_n&=&g_n \quad \textnormal{on}\quad [-T,T]\times M\\
(u_n,\partial_t u_n)_{t=0}&=&(u_{n,0},u_{n,1}) \textnormal{ bounded in } \HutL,
\end{array}
\right.
\end{eqnarray*}
with
$\nor{(u_{n,0}-\tilde{u}_{n,0},u_{n,1}-\tilde{u}_{n,1})}{\HutL}+\nor{g_n-\tilde{g}_n}{L^{1}([-T,T],L^2)} \tend{n}{\infty}0$. Then, we have
$$\nort{u_n-\tilde{u}_n}_{[-T,T]} \tend{n}{\infty}0.$$ 
\end{lemme}
\bnp
$r_n=u_n-\tilde{u}_n$ is solution of 
\begin{eqnarray*}
\left\lbrace
\begin{array}{rcl}
\Box r_n+r_n+|u_n|^4u_n-|\tilde{u}_n|^4\tilde{u}_n&=&g_n-\tilde{g}_n \quad \textnormal{on}\quad [-T,T]\times M\\
(r_n,\partial_t r_n)_{t=0}&=&(u_{n,0}-\tilde{u}_{n,0},u_{n,1}-\tilde{u}_{n,1}).
\end{array}
\right.
\end{eqnarray*}
Using energy and Strichartz estimates, we get
\bna
\nort{r_n}_{[-T,T]}&\leq &C\nor{(u_{n,0}-\tilde{u}_{n,0},u_{n,1}-\tilde{u}_{n,1})}{\HutL}+C\nor{g_n-\tilde{g}_n}{L^{1}([-T,T],L^2)}\\
&&+C\nor{r_n}{L^5([-T,T],L^{10})}\left(\nor{u_n}{L^5([-T,T],L^{10})}^4+\nor{\tilde{u}_n}{L^5([-T,T],L^{10})}^4\right).
\ena
Using Proposition \ref{propapriori}, we can divide the interval $[-T,T]$ in a finite number of intervals $I_{i,n}=[a_{i,n},a_{i+1,n}]$, $1\leq i\leq  N$, such that $C\left(\nor{u_n}{L^5(I_{i,n},L^{10})}^4+\nor{\tilde{u}_n}{L^5(I_{i,n},L^{10})}^4\right)<1/2$ so that the third term can be absorbed. We iterate this estimate $N$ times, which gives the result.
\enp
\subsubsection{Profile decomposition of the limit energy}
For $u$ solution of the nonlinear wave equation, we denote its nonlinear energy density 
\bna
e(u)(t,x)=\frac{1}{2}\left[|\partial_t u(t,x)|^2+|\nabla u(t,x)|^2+|u(t,x)|^2\right]+\frac{1}{6}|u(t,x)|^6.
\ena
For a sequence $u_n$ of solution with initial data bounded in $\HutL$, the corresponding non linear energy density is bounded in $L^{\infty}([-T,T],L^1)$ and so in the space of bounded measures on $[-T,T]\times M$. This allows to consider, up to a subsequence, its weak$*$ limit. 

The following theorem is the equivalent of Theorem 7 in \cite{DehPGNLW}. It proves that the energy limit follows the same profile decomposition as $u_n$. It will be the crucial argument that will allow to use microlocal defect measure on each profile and then to apply the linearization argument.
\begin{theorem}
\label{thmprofilNRJ}
Assume $2T<T_{focus}$.

Let $u_n$ be a sequence of solutions of 
\bna
\Box u_n+u_n+ |u_n|^4u_n=0
\ena 
with $(u_n,\partial_t u_n) (0)$  weakly convergent to $0$ in $\HutL$.

The nonlinear energy density limit of $u_n$ (up to subsequence) reads
\bna
e(t,x)=\sum_{j=1}^{+\infty} e^{(j)}(t,x)+e_f(t,x)
\ena
where $e^{(j)}$ is the limit energy limit density of $q_n^{(j)}$ (following the notation of Theorem \ref{thmdecompositionNL}) and 
$$e_f=\limvar{l}{\infty}~\limvar{n}{\infty}e(w_n^{(l)}) $$ 
where both limit are considered up to a subsequence and in the weak $*$ sense.

In particular, $e_f$ can be written 
\bna
e_f(t,x)= \int_{\xi \in S^2_x} \mu(t,x,d\xi)
\ena
with
\bna
\mu(t,x,\xi)= \mu_{-}(G_t(x,\xi))+\mu_{+}(G_{-t}(x,\xi))
\ena
where $G_t$ is the flow map of the vector field $H_{|\xi|_x}$ on $S^*M$, that is the Hamiltonian of the Riemannian metric.  

Moreover, $e$ is also the limit of the linear energy density
\bna
e_{lin}(u_n)(t,x)=\frac{1}{2}\left[|\partial_t u_n(t,x)|^2+|\nabla u_n(t,x)|^2\right].
\ena
\end{theorem}
\begin{proof}Proposition \ref{propapriori} yields $\nor{u_n}{L^{8}([-T,T]\times M)}\leq C$. Then, compact embedding and Lemma \ref{lmestimL2} yields $\nor{u_n}{L^{2}([-T,T]\times M)} \tend{n}{\infty}0$ and so $\nor{u_n}{L^{6}([-T,T]\times M)} \tend{n}{\infty}0$ by interpolation. Therefore, $e$ is the limit of $b(u_n,u_n)$, with
\bna
b(f,g)=\partial_t f(t,x) \overline{\partial_tg}(t,x) + \nabla f(t,x)\cdot \overline{\nabla g}(t,x)
\ena
Now, we have to compute the limit of $b(u_n,u_n)$ using decomposition (\ref{decomponl}) of Theorem \ref{thmdecompositionNL}. We set for any $l\in \N$
\bna
s_n^{(l)}=\sum_{j=1}^l q_n^{(j)}
\ena
and so
\bna
b(u_n,u_n)=b(s_n^{(l)},s_n^{(l)})+b(w_n^{(l)},w_n^{(l)})+2b(s_n^{(l)},w_n^{(l)})+2b(u_n,r_n^{(l)})-b(r_n^{(l)},r_n^{(l)}).
\ena
Because of (\ref{rnpetit}), $\limsu{n}{\infty}\nor{2b(u_n,r_n^{(l)})-b(r_n^{(l)},r_n^{(l)})}{L^{1}([-T,T]\times M)}$ converges to zero as $l$ tends to infinity. So, if we define $e_r^{(l)}=w^*\limvar{n}{\infty}\left(2b(u_n,r_n^{(l)})-b(r_n^{(l)},r_n^{(l)})\right)$, we have
\bna
\nor{e_r^{(l)}}{TV} \tend{l}{\infty}0.
\ena
Let $\varphi \in C^{\infty}_0(]-T,T[\times M)$. For fixed $l$, it remains to estimate 
\bna
\iint_{]-T,T[\times M} \varphi~ b(s_n^{(l)},w_n^{(l)})=\sum_{j=1}^l \iint_{]-T,T[\times M} \varphi ~b(q_n^{(j)},w_n^{(l)}).
\ena
Since $b(q_n^{(j)},w_n^{(l)})$ is bounded in $L^{\infty}(]-T,T[,L^1)$, we can assume, up to an error arbitrary small, that $\varphi$ is supported in $\left\{t<t_{\infty}^{(j)}\right\}$ or $\left\{t>t_{\infty}^{(j)}\right\}$ (replace $\varphi$ by $(1-\Psi)(t)\varphi$ with $\Psi(t_{\infty}^{(j)})=1$ and $\nor{\Psi}{L^1(]-T,T[)}$ small). On each interval, Theorem \ref{thmdescriptionNLW} allows to replace $q_n^{(j)}$ by a linear concentrating wave. Then, we combine Lemma \ref{lmwnlorth} and Lemma \ref{lmconserorth} to get the weak convergence to zero of $b(s_n^{(l)},w_n^{(l)})$ for fixed $l$.

Lemma \ref{lmconcorth} and the orthogonality of the cores of concentration give $D_h^{(j)}(p_h^{(j')},\partial_t p_h^{(j')})(t_h^{(j)}) \rightharpoonup (0,0)$ for $j\neq j'$ and $p_h^{(j')}$ a concentrating wave at rate $[h^{(j')},t^{(j')},x^{(j')}]$. Then, the same argument as before yields
\bna
b(s_n^{(l)},s_n^{(l)})\tendweak{n}{\infty} \sum_{j=1}^l e^{(j)}.
\ena
So we have proved that for any $l\in \N$
 \bna
b(u_n,u_n)\tendweak{n}{\infty} e=\sum_{j=1}^l e^{(j)}+e_{w}^{(l)}+e_{r}^{(l)}
\ena
where $e_w^{(l)}$ is the weak* limit of $b(w_n^{(l)},w_n^{(l)})$ and $e_{r}^{(l)}$ satisfies $\nor{e_r^{(l)}}{TV} \tend{l}{\infty}0$. $e_w^{(l)}$ is the weak* limit of a sequence of solutions of the linear wave equation weakly convergent to $0$ in energy space. Therefore, it has the announced form using the link with microlocal defect measure (see Lemma \ref{lmmeasuredamped}).

We get the final result by letting $l$ tend to infinity. \end{proof}
\begin{remarque}
The fact that $|u_n|^6$ is weakly convergent to $0$ is false if we consider the limit in $\mathcal{D}'(M)$ time by time. For example, for a nonlinear concentrating wave with $t_n=0$, the weak limit in $\mathcal{D}'(]-T,T[\times M)$ of $|u_n|^6$ is of course still zero but the weak limit of $|u_n|^6(t)$ in $\mathcal{D}'(M)$ is zero if $t\neq 0$ and a multiple of a Dirac function if $t=0$. So the limit in $\mathcal{D}'(M)$ of ${e_n}_{|t=0}$ is not the same as the one of $b(u_n,u_n)_{|t=0}$. This comes from the fact that the limit of $b(u_n,u_n)(t)$ is not equicontinuous as a function of $t$ while it is the case for the nonlinear energy. Yet, in the proof, we are only interested in its limit in the space-time distributional sense which will be continuous. Actually, the discontinuity at $t=0$ of the limit of $b(u_n,u_n)(t)$ can be described explicitely from the scattering operator. At the contrary, the fact that the nonlinear energy density $e(t)$ is continuous in time can, in this case, be seen as a consequence of the conservation of the nonlinear energy of the scattering operator. 
\end{remarque}
\section{Control and stabilization}
\label{sectioncontrol}
\subsection{Weak observability estimates, stabilization}
\subsubsection{Why Klein-Gordon and not the wave?}
\label{subsectcontrexwave}
In this subsection, we prove that the expected observability estimate 
\bna
E(u)(0) \leq C\iint_{[0,T]\times M}  \left|a\partial_t u\right|^2 ~dtdx .
\ena
does not hold for the nonlinear damped wave equation $\Box u+\partial_t u+u^5=0$ (in the simpler case $a\equiv 1$), even for small data. It explains why we have chosen the Klein-Gordon equation instead.
The main point is that for small data, the nonlinear solution is close to the linear one which has the constants (in space-time) as undamped solutions (which is obviously false for $\Box u+u=0$).

We take $a\equiv 1$ and initial data constant equal to $(\varepsilon,0)$. The nonlinear wave equation takes the form of the following ODE
\bna
\left\lbrace
\begin{array}{rcl}
\ddot{u}+\dot{u}+u^5 &=&0\quad \textnormal{on}\quad [0,T]\\
(u(0),\dot{u}(0))&=&(\varepsilon,0). 
\end{array}
\right.
\ena
Decreasing of energy yields for any $t\geq 0$
\bna
E(t)=\frac{1}{2}\dot{u}^2 + \frac{1}{6}u^6(t) \leq E(0) =\frac{1}{6}\varepsilon^6
\ena
and so
\bna
|u(t)|\leq \varepsilon \quad \forall t\geq 0.
\ena
Then, $c=\dot{u}$ is solution of
\bna
\left\lbrace
\begin{array}{rcl}
\dot{c}+c+u^5&=&0 \quad \textnormal{on}\quad [0,T]\\
c(0)&=&0 
\end{array}
\right.
\ena
Therefore,
\bna
&c(t)=-\int_0^t e^{-(t-s)}u^5(s)~ds\\
\textnormal{and}&|\dot{u}(t)|=|c(t)|\leq \varepsilon^5.
\ena
For any $T>0$, we have
\bna
\int_0^T |\dot{u}(s)|^2 \leq T \varepsilon^{10}.
\ena 
Therefore, the observability estimate
\bna
 T \varepsilon^{10}\geq \int_0^T |\dot{u}(s)|^2 \geq C E(0)=C \frac{1}{6}\varepsilon^6 
\ena
can not hold if $\varepsilon$ is taken small enough.
\subsubsection{Weak observability estimate}
As explained in the introduction, the proof of stabilization consists in the analysis of possible sequences contradicting an observability estimate. The first step is to prove that such sequence is linearizable in the sense that its behavior is close to solutions of the linear equation.
\begin{prop}
\label{proplinear}
Let $\omega$ satisfying Assumption \ref{hypGCCfocus} and $a \in C^{\infty}(M)$ satisfying $a(x)>\eta>0$ for all $x\in \omega$. Let $T>T_{0}$ and $u_n$ be a sequence of solutions of 
\begin{eqnarray}
\left\lbrace
\begin{array}{rcl}
\Box u_n+u_n+|u_n|^4u_n+a(x)^2\partial_t u_n&=& 0\quad \textnormal{on}\quad [0,T]\times M\\
(u_n,\partial_t u_n)_{t=0}&=&(u_{0,n},u_{1,n}) \in \HutL.
\end{array}
\right.
\end{eqnarray}
satisfying 
\bnan
(u_{0,n},u_{1,n}) \tendweak{n}{\infty}0 \textnormal{ weakly in } \HutL \nonumber\\
\iint_{[0,T]\times M} \left| a(x)\partial_t u_n\right|^2 ~dtdx \tend{n}{\infty}0 \label{cvgeomega}
\enan
Then, $u_n$ is linearizable on $[0,t]$ for any $t<T-T_{0}$, that is 
\bna
\nort{u_n-v_n}_{[0,t]}\tend{n}{\infty} 0
\ena
where $v_n$ is the solution of 
\begin{eqnarray*}
\left\lbrace
\begin{array}{rcl}
\Box v_n&=& 0\quad \textnormal{on}\quad [0,T]\times M\\
(v_n,\partial_t v_n)_{t=0}&=&(u_{0,n},u_{1,n}) .
\end{array}
\right.
\end{eqnarray*}
\end{prop}
\bnp
Denote $t_*=\sup \left\{s\in [0,T] \left| \limvar{n}{\infty} \nort{u_n-v_n}_{[0,s]}=0\right.\right\}$ and we have to prove $t_*\geq T-T_{0}$. If it is not the case, we can find an interval $[t_*-\varepsilon,t_*-\varepsilon+L]\subset [0,T]$ with $T_0<L<T_{focus}$ and $0<2\varepsilon< L-T_0$ (if $t_*=0$, take the interval $[0,L]\subset [0,T]$). Then, Lemma \ref{propprofilnul} below gives that $u_n$ is linearizable on $[t_*-\varepsilon,t_*+\varepsilon]$. We postpone the proof of Lemma \ref{propprofilnul} and finish the proof of the proposition. The definition of $t_*$ gives $\limvar{n}{\infty} \nort{u_n-v_n}_{[0,t_*-\varepsilon]}=0$ and we have proved that  $\limvar{n}{\infty}\nort{u_n-\tilde{v}_n}_{[t_*-\varepsilon,t_*+\varepsilon]}=0$ where $\tilde{v}_n$ is solution of 
\bna
\Box \tilde{v}_n= 0\quad;\quad (\tilde{v}_n,\partial_t \tilde{v}_n)_{t=t_*-\varepsilon}=(u_n,\partial_t u_n)_{t=t_*-\varepsilon}. 
\ena
Since the norm $\nort{\cdot}$ controls the energy norm, this easily yields $\limvar{n}{\infty} \nort{u_n-v_n}_{[0,t_*+\varepsilon]}=0$ which is a contradiction to the definition of $t_*$.
\enp
\begin{lemme}
\label{propprofilnul}
With the assumptions of Proposition \ref{proplinear}. Consider the profile decomposition according to Theorem \ref{thmdecompositionNL} of $u_n$ on a subinterval $[t_0,t_0+L]\subset[0,T]$ with $T_0<L<T_{focus}$.
 
Then, for any $0<\varepsilon< L-T_0$, this decomposition does not contain any non linear concentrating wave with $t_{\infty}^{(j)}\in [t_0,t_0+\varepsilon]$ and $u_n$ is linearizable on $[t_0,t_0+\varepsilon]$.
\end{lemme}
\bnp
To simplify the notation, we work on the interval $[0,L]$. Moreover, since $a(x)\partial_t u_n$ tends to $0$ in $L^1L^2$, Lemma \ref{approxNL} allows to assume with the same assumptions that $u_n$ is solution of the nonlinear equation without damping.  Proposition \ref{propapriori} and Lemma \ref{lmestimL2} (with Rellich Theorem) give that $u_n$ is bounded in $L^8([0,T]\times M)$ and convergent to $0$ in $L^2([0,T]\times M)$. Therefore, $u_n$ tends to $0$ in $L^7([0,T]\times M)$ and so $|u_n|^4u_n$ is convergent to $0$ in $L^{7/5}([0,T]\times M)\hookrightarrow L^{4/3}([0,T]\times M)\hookrightarrow H_{loc}^{-1}(]0,l[\times M)$. Then, if we consider the (space-time) microlocal defect measure of $u_n$, the elliptic regularity and the equation verified by $u_n$ gives that $\mu$ is supported in $\left\{\tau^2= \left|\xi\right|_x^2\right\}$ as in the linearizable case. So, combining this with (\ref{cvgeomega}), we get 
\bna
u_n \tend{n}{\infty}0 \textnormal{ in } H^1_{loc}(]0,L[\times \omega ). 
\ena
Using the notation of Theorem \ref{thmprofilNRJ}, this gives $e=0$ on $]0,L[\times \omega $. Since all the measures in the decomposition of $e$ are positive, we get the same result for any nonlinear concentrating wave in the decomposition of $u_n$, that is
\bna
q_n^{(j)} \tend{n}{\infty}0 \textnormal{ in } H^1_{loc}(]0,L[\times \omega ) 
\ena
and if $\mu^{(j)}$ is the microlocal defect measure of $q_n^{(j)}$, we have 
\bnan
\label{mujnul}
\mu^{(j)}\equiv 0 \textnormal{ in }S^*(]0,L[\times \omega).
\enan

Assume that $t_{\infty}^{(j)}\in [0,\varepsilon]$ for one $j\in \N$, so that the interval $[t_{\infty}^{(j)},L]$ has length greater that $T_0$. Denote $p_n^{(j)}$ the linear concentrating wave approaching $q_n^{(j)}$ in the interval $I_n^{3,\Lambda}$ according to the notation of Theorem \ref{thmdescriptionNLW}, so that for any $t_{\infty}^{(j)}<t<L$ (here, we use the fact that $L<T_{focus}$), we have
\bna
\nort{q_n^{(j)}-p_n^{(j)}}_{[t,L]}\tend{n}{\infty}0.
\ena
In particular, $\mu^{(j)}$ is also attached to $p_n^{(j)}$ on the time interval $]t_{\infty}^{(j)},L]$. Since $p_n^{(j)}$ is solution of the linear wave equation, its measure propagates along the hamiltonian flow. Assumption \ref{hypGCCfocus} and $\left|L-t_{\infty}^{(j)}\right|> T_0$ ensure that the geometric control condition is still verified on the interval $[t_{\infty}^{(j)},L]$ which gives $\mu^{(j)}\equiv 0$ when combined with (\ref{mujnul}). That means $p_n^{(j)}\equiv 0$ and so $q_n^{(j)}\equiv 0$ as expected.

Then, for the profile decomposition of $u_n$ on the interval $[0,L]$ (here the weak limit $u$ is necessarily zero)
\bna
u_n= \sum_{j=1}^l q_n^{(j)}+w_n^{(l)}+r_n^{(l)},
\ena
we have proved that $t_n^{(j)} \in ]\varepsilon,L]$. Then Theorem \ref{thmdescriptionNLW} and $L<T_{focus}$ provides a linear concentrating wave $p_n^{(j)}$ such that $\limsu{n}{\infty}\nort{q_n^{(j)}-p_n^{(j)}}_{[0,\varepsilon]}=0$ while Lemma \ref{lmnonreconcentr} give $\limsu{n}{\infty}\nor{p_n^{(j)}}{L^5([0,\varepsilon],L^{10})} =0$. Moreover, the conclusion of Theorem \ref{thmdescriptionNLW} give  $\limsu{n}{\infty}\nor{w_n^{(l)}+r_n^{(l)}}{L^5([0,\varepsilon],L^{10})} \tend{l}{\infty}0$. This finally yields $\limsu{n}{\infty}\nor{u_n}{L^5([0,\varepsilon],L^{10})} =0$ and therefore
\bna
 \nor{|u_n|^4u_n}{L^{1}([0,\varepsilon],L^2)}\tend{n}{\infty}0.
\ena
This gives exactly that $u_n$ is linearizable on $[0,\varepsilon]$.
\enp
We are now ready for the proof of some weak observability estimates. We recall the notation $E(u)$ for the nonlinear energy defined in (\ref{defnonmlinearNRJ}).
\begin{theorem}
\label{thminegobserv}
Let $\omega$ satisfying Assumption \ref{hypGCCfocus} with $T_0$ and $a \in C^{\infty}(M)$ satisfying $a(x)>\eta>0$ for all $x\in \omega$. Let $T>2T_{0}$ and $R_0>0$. Then, there exists $C>0$ such that for any $u$ solution of 
\begin{eqnarray}
\label{solutiondamped}
\left\lbrace
\begin{array}{rcl}
\Box u+u+|u|^4u+a^2(x)\partial_t u&=& 0\quad \textnormal{on}\quad [0,T]\times M\\
(u,\partial_t u)_{t=0}&=&(u_{0},u_{1}) \in \HutL\\
\nor{(u_{0},u_{1})}{\HutL} &\leq& R_0
\end{array}
\right.
\end{eqnarray}
satisfies 
\bna
E(u)(0) \leq C\left(\iint_{[0,T]\times M}  \left|a(x)\partial_t u\right|^2 ~dtdx  +\nor{(u_{0},u_{1})}{L^2\times H^{-1}}E(u)(0)\right).
\ena
\end{theorem}
\bnp
We argue by contradiction : we suppose that there exists a sequence $u_n$ of solutions of (\ref{solutiondamped}) such that
\bna
\left(\iint_{[0,T]\times M}  \left|a(x)\partial_t u_n\right|^2 ~dtdx  +\nor{(u_{0,n},u_{1,n})}{L^2\times H^{-1}}E(u_n)(0)\right)\leq \frac{1}{n}E(u_n)(0).
\ena
Denote $\alpha_n=(E(u_n)(0))^{1/2}$. By Sobolev embedding for the $L^6$ norm, we have $\alpha_n \leq C (R_0)$. So, up to extraction, we can assume that $\alpha_n \longrightarrow \alpha\geq 0$.

We will distinguish two cases : $\alpha>0$ and $\alpha=0$.
\bigskip

\textbullet \quad First case : $\alpha_n \longrightarrow \alpha>0$\\
The second part of the estimate gives $\nor{(u_{0,n},u_{1,n})}{L^2\times H^{-1}}\tend{n}{\infty}0$ and so $(u_{0,n},u_{1,n}) \tendweak{n}{\infty}0$ in $H^1\times L^2$. Therefore, we are in position to apply Proposition \ref{proplinear} and get that $u_n$ is linearizable on an interval $[0,L]$ with $L>T_0$. We get a contradiction to $\alpha>0$ by applying the following classical linear proposition, which can be easily proved using microlocal defect measure as in Lemma \ref{propprofilnul}.
\begin{prop}
\label{propcontrollin}
Let $\omega$ satisfying Assumption \ref{hypGCCfocus} with $T_0$. Let $T>T_0$ and $v_n$ be a sequence of solutions of 
\begin{eqnarray*}
\left\lbrace
\begin{array}{rcl}
\Box v_n&=& 0\quad \textnormal{on}\quad [0,T]\times M\\
(v_n(0),\partial_t v_n(0))&\tendweak{n}{\infty}& 0 \textnormal{ in } \HutL
\end{array}
\right.
\end{eqnarray*}
satisfying
\bna
\iint_{[0,T]\times M}  \left|a(x)\partial_t v_n\right|^2 ~dtdx \tend{n}{\infty}0
\ena
Then, $(v_n(0),\partial_t v_n(0))\tend{n}{\infty} 0$ for the strong topology of $H^1\times L^2$. The same result holds with $\Box u_n$ replaced by $\Box u_n+u_n$. 
\end{prop}
\textbullet \quad Second case : $\alpha_n \longrightarrow 0$\\
Let us make the change of unknown $w_n=u_n/\alpha_n$. $w_n$ is solution of the system
\bnan
\label{eqnalphan}
\Box w_n +a^2(x)\partial_t w_n+ w_n+\alpha_n^4|w_n|^4w_n=0
\enan
and
\begin{eqnarray*}
\iint_{[0,T]\times M}  \left|a(x)\partial_t w_n\right|^2 ~dtdx \leq \frac{1}{n}. 
\end{eqnarray*}
We have for a large constant $C>0$ depending on $R_0$ and for all $t\in [0,T]$,
 $$\frac{1}{C}\left\|(u_n,\partial_tu_n)\right\|^2_{\HutL}\leq E(u_n)\leq  C\left\|(u_n,\partial_tu_n)\right\|^2_{\HutL}.$$
 Therefore, we have
\bnan
\left\|(w_n(t),\partial_tw_n(t))\right\|_{\HutL}=\frac{ \left\|(u_n(t),\partial_tu_n(t))\right\|_{\HutL} }{ \sqrt{E(u_n(0))} }\leq C\frac{\sqrt{E(u_n(t))}  }{\sqrt{E(u_n(0))}}\leq C \nonumber \\
\label{H1tend1}
\left\|(w_n(0),\partial_tw_n(0))\right\|_{\HutL}=\frac{ \left\|(u_n(0),\partial_tu_n(0))\right\|_{\HutL} }{ \sqrt{E(u_n(0))} }\geq \frac{1}{\sqrt{C}}>0.
\enan
Thus, we have $\left\|(w_n(0),\partial_tw_n(0))\right\|_{\HutL}\approx 1$ and $(w_n,\partial_t w_n)$ is bounded in $L^{\infty}([0,T],\HutL)$.

Applying Strichartz estimates to equation (\ref{eqnalphan}), we get for $C=C(R_0)>0$
\bna
\nor{w_n}{L^{5}([0,T],L^{10})}\leq C(1+\alpha_n^4 \nor{w_n}{L^{5}([0,T],L^{10})}^5)
\ena
Then, using a bootstrap argument, we deduce that $\nor{w_n}{L^{5}([0,T],L^{10})}$ is bounded and therefore
\bna
\Box w_n +w_n\tend{n}{\infty} 0 \textnormal{ in } L^{1}([0,T],L^{2}).
\ena 
Proposition \ref{propcontrollin} yields that $w_n$ converges strongly to some $w$ solution of 
\bna
\Box w+w=0;\quad \partial_t w\equiv 0 \textnormal{ on } \omega.
\ena
We deduce as in J. Rauch and M. Taylor \cite{RauchTaylordecay} or C. Bardos, G. Lebeau, J. Rauch \cite{BLR} that the set of such solutions is finite dimensional and admits an eigenvector $w$ for $\Delta$. By unique continuation for second order elliptic operator, we get $\partial_t w\equiv 0$. Multiplying the equation by $\bar{w}$ and integrating, we obtain $w\equiv 0$ (note that, at this stage, the choice of the Klein-Gordon equation instead of the wave equation is crucial to avoid the constant solutions). 
We conclude that $(w_n(0),\partial_t w_n(0))$ tends to $0$ strongly in $\HutL$ which gives a contradiction to (\ref{H1tend1}). 
\enp
\subsection{Controllability}
\subsubsection{Linear control}
In this section, we recall some well known results about linear control theory and HUM method. Let $(\Phi_0,\Phi_1)\in L^2\times H^{-1}$. We solve the system
\bnan
\label{eqnHUMlinphi}
\left\lbrace
\begin{array}{rcl}
\Box \Phi+\Phi&=& 0\quad \textnormal{on}\quad [0,T]\times M\\
(\Phi,\partial_t \Phi)_{|t=0}&=&(\Phi_0,\Phi_1) .
\end{array}
\right.
\enan
and
\begin{eqnarray}
\label{eqnHUMlinv}
\left\lbrace
\begin{array}{rcl}
\Box v+v&=& a^2\Phi \quad \textnormal{on}\quad [0,T]\times M\\
(v,\partial_t v)_{|t=T}&=&(0,0) .
\end{array}
\right.
\end{eqnarray}
The HUM operator $S$ from $L^2\times H^{-1}$ to $L^2\times H^1$ is defined by
\bna
S(\Phi_0,\Phi_1)=(-\partial_t v(0),v(0)).
\ena
\begin{lemme}
If $\omega$ satisfies the geometric control Assumption \ref{hypGCC}, then S is an isomorphism.
\end{lemme}
\bnp
Multiplying equation (\ref{eqnHUMlinv}) by $\bar{\Phi}$, integrating over $[0,T]\times M$ and integrating by part, we get the formula
\bna
\int_0^T\int_M |a\Phi|^2=-\int_M \partial_tv(0) \bar{\phi}(0)+\int_M v(0) \partial_t\bar{\phi}(0)=\left\langle S(\Phi_0,\Phi_1),(\Phi_0,\Phi_1)\right\rangle
\ena
where $\left\langle .,.\right\rangle$ denotes the duality between $L^2\times H^1$ and $L^2\times H^{-1}$. We get the conclusion thanks to the following observability estimate which can be proved by the same techniques used in the nonlinear problem
\bna
\nor{(\Phi_0,\Phi_1)}{L^2\times H^{-1}}^2\leq \int_0^T\int_M |a\Phi|^2.
\ena
\enp
\subsubsection{Controllability for small data}
\begin{theorem}
\label{thmcontrolpetit}
Let $\omega$ satisfying Assumption \ref{hypGCC} and $T>T_0$. Then, there exists $\delta>0$ such that for any $(u_0,u_1)$ and $(\tilde{u}_0,\tilde{u}_1)$ in $H^1\times L^2$, with
\bna
\nor{(u_0,u_1)}{\HutL} \leq \delta; \quad \nor{(\tilde{u}_0,\tilde{u}_1)}{\HutL}\leq \delta
\ena
there exists $g\in L^{\infty}([0,2T],L^2)$ supported in $[0,2T]\times \omega$ such that the unique strong solution of 
\begin{eqnarray*}
\left\lbrace
\begin{array}{rcl}
\Box u+u+|u|^4u&=& g\quad \textnormal{on}\quad [0,2T]\times M\\
(u(0),\partial_t u(0))&=&(u_0,u_1) .
\end{array}
\right.
\end{eqnarray*}
satisfies $(u(2T),\partial_t u(2T))=(\tilde{u}_0,\tilde{u}_1)$
\end{theorem}
\bnp
The proof is very similar to \cite{DLZstabNLW} except that the critical exponent do not allow to use compactness argument and we use the classical Picard fixed point instead of Schauder, as done in \cite{control-nl} or \cite{LaurentNLSdim1}, \cite{LaurentNLSdim3} for NLS. 
By a compactness argument, we can select $a \in C^{\infty}_0(\omega)$ with $a(x)>\eta>0$ for $x$ in $\tilde{\omega}$ where $\tilde{\omega}$ satisfies Assumption \ref{hypGCC}. Since the equation is reversible, we can assume $(\tilde{u}_0,\tilde{u}_1)\equiv (0,0)$ and take the time $T$ instead of $2T$. We seek $g$ of the form $a^2(x)\Phi$ where $\Phi$ is solution of the free wave equation as in linear control theory with initial datum $(\Phi_0,\Phi_1)\in L^2\times H^{-1}$. The purpose will be to choose the right $(\Phi_0,\Phi_1)\in L^2\times H^{-1}$ to get the expected data. We consider the solutions of the two systems
 \bna
\left\lbrace
\begin{array}{rcl}
\Box \Phi+\Phi&=& 0\quad \textnormal{on}\quad [0,T]\times M\\
(\Phi,\partial_t \Phi)_{|t=0}&=&(\Phi_0,\Phi_1)
\end{array}
\right.
\ena
and
\begin{eqnarray}
 \label{eqncontrolu}
\left\lbrace
\begin{array}{rcl}
\Box u+u+|u|^4u&=& a^2\Phi \quad \textnormal{on}\quad [0,T]\times M\\
(u,\partial_t u)_{|t=T}&=&(0,0) .
\end{array}
\right.
\end{eqnarray}

Let us define the operator
\begin{eqnarray}
\begin{array}{rrcl}
L:&L^2\times H^{-1} &\rightarrow &H^1\times L^2\\
& (\Phi_0,\Phi_1)&\mapsto &L (\Phi_0,\Phi_1)=(u,\partial_t u)_{|t=0}.
\end{array}
\end{eqnarray}
We split $u=v+\Psi$ with $\Psi$ solution of 
 \bnan
  \label{eqncontrolpsi}
\left\lbrace
\begin{array}{rcl}
\Box \Psi+\Psi&=& a^2\Phi\quad \textnormal{on}\quad [0,T]\times M\\
(\Psi,\partial_t \Psi)_{|t=T}&=&(0,0) .
\end{array}
\right.
\enan
This corresponds to the linear control, and $(-\partial_t \Psi,\Psi)_{|t=0}=S(\Phi_0,\Phi_1)$. As for function $v$, it is solution of
 \bnan
 \label{eqncontrolv}
\left\lbrace
\begin{array}{rcl}
\Box v+v&=&-|u|^4u \quad \textnormal{on}\quad [0,T]\times M\\
(v,\partial_t v)_{|t=T}&=&(0,0) .
\end{array}
\right.
\enan
$\Phi$ belongs to $C([0,T],L^2)$. So, $u$, $v$ and $\Psi$ belong to $C([0,T],H^1) \cap C^1([0,T],L^2)\cap L^5([0,T],L^{10})$. We can write
\bna
L (\Phi_0,\Phi_1)=K(\Phi_0,\Phi_1)+S (\Phi_0,\Phi_1)
\ena
where $K(\Phi_0,\Phi_1)=(-\partial_t v,v)_{|t=0}$. $L (\Phi_0,\Phi_1)=(-u_1,u_0)$ is equivalent to $(\Phi_0,\Phi_1)=-S^{-1}K(\Phi_0,\Phi_1)+S^{-1}(-u_1,u_0)$. Defining the operator $B:L^2\times H^{-1}\rightarrow L^2\times H^{-1}$ by
\bna
B(\Phi_0,\Phi_1)=-S^{-1}K(\Phi_0,\Phi_1)+S^{-1}(-u_1,u_0),
\ena
the problem $L (\Phi_0,\Phi_1)=(-u_1,u_0)$ is equivalent to finding a fixed point of $B$. We will prove that if $\nor{(u_0,u_1)}{\HutL}$ is small enough, $B$ is a contraction and reproduces a small ball $B_R$ of $L^2\times H^{-1}$.

Since $S$ is an isomorphism, we have
\bna
\nor{B(\Phi_0,\Phi_1)}{L^2\times H^{-1}}\leq C (\nor{K(\Phi_0,\Phi_1)}{L^2\times H^1}+\nor{(u_0,u_1)}{\HutL})
\ena
So we are led to estimate $\nor{K(\Phi_0,\Phi_1)}{L^2\times H^1}=\nor{(v,\partial_t v)_{|t=0}}{\HutL}$. Energy estimates applied to equation (\ref{eqncontrolv}) and Hölder inequality give
\bna
\nor{(v,\partial_t v)_{|t=0}}{\HutL}\leq C\nor{|u|^4u}{L^1([0,T],L^2)}\leq C \nor{u}{L^5([0,T],L^{10})}^5.
\ena
But Strichartz estimates applied to equation (\ref{eqncontrolu}) give
\bna
\nor{u}{L^5([0,T],L^{10})}\leq C\left(\nor{a^2\Phi}{L^1([0,T],L^2)}+\nor{u}{L^5([0,T],L^{10})}^5\right)\leq C\left(\nor{(\Phi_0,\Phi_1)}{L^2\times H^{-1}}+\nor{u}{L^5([0,T],L^{10})}^5\right).
\ena
Using a bootstrap argument, we get that for $\nor{(\Phi_0,\Phi_1)}{L^2\times H^{-1}}\leq R$ small enough, we have
\bnan
\label{estimuL5L10}
\nor{u}{L^5([0,T],L^{10})}\leq C\nor{(\Phi_0,\Phi_1)}{L^2\times H^{-1}}.
\enan
We finally obtain 
\bna
\nor{B(\Phi_0,\Phi_1)}{L^2\times H^{-1}}\leq C\left(\nor{(\Phi_0,\Phi_1)}{L^2\times H^{-1}}^5+\nor{(u_0,u_1)}{\HutL}\right).
\ena
Choosing $R$ small enough and $\nor{(u_0,u_1)}{H^1\times L^2}\leq R/2C$, we obtain $\nor{B(\Phi_0,\Phi_1)}{L^2\times H^{-1}}\leq R$ and $B$ reproduces the ball $B_R$. Let us now prove that $B$ is contracting. We examine the system
\bnan
 \label{eqncontroluutilde}
\left\lbrace
\begin{array}{rcl}
\Box (u-\tilde{u})+(u-\tilde{u})+|u|^4u-|\tilde{u}|^4\tilde{u}&=& a^2(\Phi-\tilde{\Phi}) \quad \textnormal{on}\quad [0,T]\times M\\
(u-\tilde{u},\partial_t (u-\tilde{u}))_{|t=T}&=&(0,0) .
\end{array}
\right.
\enan
\bnan
 \label{eqncontrolvvtilde}
\left\lbrace
\begin{array}{rcl}
\Box (v-\tilde{v})+(v-\tilde{v})+|u|^4u-|\tilde{u}|^4\tilde{u}&=& 0 \quad \textnormal{on}\quad [0,T]\times M\\
(v-\tilde{v},\partial_t (v-\tilde{v}))_{|t=T}&=&(0,0) .
\end{array}
\right.
\enan
We obtain similarly
\bnan
\label{estimdiifB}
\nor{B(\Phi_0,\Phi_1)-B(\tilde{\Phi}_0,\tilde{\Phi}_1)}{L^2\times H^{-1}}&\leq&  C\nor{(v-\tilde{v},\partial_t (v-\tilde{v}))_{|t=0}}{\HutL}\nonumber\\
&\leq &C\nor{|u|^4u-|\tilde{u}|^4\tilde{u}}{L^1([0,T],L^2)}\nonumber\\
&\leq &C \nor{u-\tilde{u}}{L^5([0,T],L^{10})}(\nor{u}{L^5([0,T],L^{10})}^4+\nor{\tilde{u}}{L^5([0,T],L^{10})}^4)\nonumber\\
&\leq &CR^4 \nor{u-\tilde{u}}{L^5([0,T],L^{10})}
\enan
where we have used estimate (\ref{estimuL5L10}) for the last inequality. Applying Strichartz estimates to equation (\ref{eqncontroluutilde}), we get
\bna
\nor{u-\tilde{u}}{L^5([0,T],L^{10})}&\leq &C(\nor{|u|^4u-|\tilde{u}|^4\tilde{u}}{L^1([0,T],L^2)}+\nor{a^2(\Phi-\tilde{\Phi}) }{L^1([0,T],L^2)})\\
&\leq & CR^4 \nor{u-\tilde{u}}{L^5([0,T],L^{10})} + C \nor{(\Phi_0,\Phi_1)-(\tilde{\Phi}_0,\tilde{\Phi}_1)}{L^2\times H^{-1}}
\ena
If $R$ is taken small enough, it yields
\bnan
\label{estimdifu}
\nor{u-\tilde{u}}{L^5([0,T],L^{10})}\leq C \nor{(\Phi_0,\Phi_1)-(\tilde{\Phi}_0,\tilde{\Phi}_1)}{L^2\times H^{-1}}.
\enan
Combining (\ref{estimdiifB}) and (\ref{estimdifu}), we finally obtain for $R$ small enough
\bna
\nor{B(\Phi_0,\Phi_1)-B(\tilde{\Phi}_0,\tilde{\Phi}_1)}{L^2\times H^{-1}}\leq CR^4 \nor{(\Phi_0,\Phi_1)-(\tilde{\Phi}_0,\tilde{\Phi}_1)}{L^2\times H^{-1}}
\ena
and $B$ is a contraction for $R$ small enough, which completes the proof of Theorem \ref{thmcontrolpetit}.
\enp
\subsubsection{Controllability of high frequency data}
This subsection is devoted to the proof of the both main theorem of the article : Theorem \ref{thmdecresexp} and \ref {mainthm}.
\bnp[Proof of Theorem \ref{thmdecresexp}]
First, by decreasing of the energy and Sobolev embedding, there exists some constant $C(R_0)$ such that the assumption $\nor{(u_0,u_1)}{\HutL}\leq R_0$ implies 
\bnan
\label{normaHubornee}
E(u)(t) \leq C(R_0) \textnormal{ and }\nor{(u,\partial_t u)(t)}{\HutL}\leq C(R_0); \quad\forall t\geq 0.
\enan
Fix $T$ such that Theorem \ref{thminegobserv} applies. Then, there exists $\e>0$ such that for any $(u_0,u_1)$ satisfying 
\bnan
\nor{(u_0,u_1)}{\HutL}\leq C(R_0);\quad \nor{(u_0,u_1)}{L^2\times H^{-1}} \leq \e , \label{taille}
\enan
we have the strong observability estimate 
\bna
E(u)(0) \leq C\iint_{[0,T]\times M}  \left|a(x)\partial_t u\right|^2 ~dtdx .
\ena
for any solution of the damped equation (\ref{solutiondampedbis}). That means that there exists $0<C$ such that any solution of the damped equation satisfying  (\ref{taille}) fulfills
\bnan
\label{decay}
E(u)(T) \leq (1-C) E(u)(0).
\enan
Pick $N\in \N$ large enough  such that $(1-C)^N C(R_0)\leq \e^2/2$. 

Corollary \ref{corL2choixpetit} and (\ref{normaHubornee}) allow us to choose $\delta$ small enough such that the assumption
\bna
\nor{(u_0,u_1)}{\HutL}\leq R_0; &\quad &\nor{(u_0,u_1)}{L^2\times H^{-1}} \leq \delta
\ena
implies 
\bnan
 \nor{(u(nT),\partial_tu(nT))}{L^2\times H^{-1}} \leq\e, \quad 0\leq n\leq N. \label{petitepsilon}
\enan
So, with that choice, we have $E(u)(NT) \leq (1-C)^N E(u)(0)$. Then, by the energy decreasing, for any $t\geq NT$, we have
\bna
\nor{(u,\partial_tu)(t)}{L^2\times H^{-1}}^2\leq 2 E(u)(t)\leq 2E(u)(NT)\leq \e^2.
\ena

Therefore, the decay estimate (\ref{decay}) is true on each interval $[nT,(n+1)T]$, $n\in \N$ and we have 
\bna
E(u)(nT) \leq (1-C)^n E(u)(0)
\ena
which yields the result.
\enp
\bnp[Proof of Theorem \ref{mainthm}]
Since the equation is reversible, we can assume $(\tilde{u}_0,\tilde{u}_1)=(0,0)$. By a compactness argument, we can select $a \in C^{\infty}_0(\omega)$ with $a(x)>\eta>0$ for $x$ in $\tilde{\omega}$ where $\tilde{\omega}$ satisfies Assumption \ref{hypGCCfocus}. We will first use the damping term $a(x)^2\partial_tu$ as a term of control. We apply Theorem \ref{thmdecresexp} and Theorem \ref{thmcontrolpetit} once the energy is small enough.
\enp

{\it Acknowledgements.} The author deeply thanks his adviser Patrick G\'erard for drawing his attention to this problem and for helpful discussions and encouragements.
\bibliographystyle{plain} 
\bibliography{biblio}

\begin{thebibliography}{10}

\bibitem{dampingcriticAlIbrNak}
L.~Aloui, S.~Ibrahim, and K.~Nakanishi.
\newblock {Exponential energy decay for damped Klein-Gordon Equation with
  nonlinearities of arbitrary growth}.
\newblock {\em private communication}, 2009.

\bibitem{BahouriGerard}
H.~Bahouri and P.~G\'erard.
\newblock {High frequency approximation of critical nonlinear wave equations}.
\newblock {\em American J. Math.}, 121:131--175, 1999.

\bibitem{BLR}
C.~Bardos, G.~Lebeau, and J.~Rauch.
\newblock {Sharp sufficient conditions for the observation, control and
  stabilization of waves from the boundary}.
\newblock {\em SIAM J. Control Optim.}, 305:1024--1065, 1992.

\bibitem{BrezisCoronBubbles}
H.~Brezis and J-M Coron.
\newblock {Convergence of solutions of H-systems or how to blow bubbles}.
\newblock {\em Archive for Rational Mechanics and Analysis}, 89(1):21--56,
  1985.

\bibitem{BurqGerardCNS}
N.~Burq and P.~G\'erard.
\newblock {Condition n\'ec\'essaire et suffisante pour la contr\^olabilite
  exacte des ondes}.
\newblock {\em Comptes rendus de l'Acad{\'e}mie des sciences. S{\'e}rie 1,
  Math{\'e}matique}, 325(7):749--752, 1997.

\bibitem{Strichartz}
N.~Burq, P.~G\'erard, and N.~Tzvetkov.
\newblock {Strichartz Inequalities and the nonlinear Schr\"odinger equation on
  compact manifolds}.
\newblock {\em American Journal of Mathematics.}, 126:569--605, 2004.

\bibitem{BLPNLWcritdomain}
N.~Burq, G.~Lebeau, and F.~Planchon.
\newblock {Global existence for energy critical waves in 3-D domains}.
\newblock {\em J. of American Math. Soc}, 21(3):831, 2008.

\bibitem{DehPGNLW}
B.~Dehman and P.~G\'erard.
\newblock {Stabilization for the Nonlinear Klein Gordon Equation with critical
  Exponent}.
\newblock {\em Pr\'epublication de l'Universit\'e Paris-Sud, available at
  http://www.math.u-psud.fr/\char`\~biblio/saisie/fichiers/ppo\_2002\_35.ps},
  2002.

\bibitem{control-nl}
B.~Dehman, P.~G\'erard, and G.~Lebeau.
\newblock {Stabilization and control for the nonlinear Schr\"odinger equation
  on a compact surface}.
\newblock {\em Mathematische Zeitschrift}, 254(4):729--749, 2006.

\bibitem{HUMDehLeb}
B.~Dehman and G.~Lebeau.
\newblock {Analysis of the HUM Control Operator and Exact Controllability for
  Semilinear Waves in Uniform Time}.
\newblock {\em SIAM Journal on Control and Optimization}, 48(2):521--550, 2009.

\bibitem{DLZstabNLW}
B.~Dehman, G.~Lebeau, and E.~Zuazua.
\newblock {Stabilization and control for the subcritical semilinear wave
  equation}.
\newblock {\em Annales scientifiques de l'Ecole normale sup{\'e}rieure},
  36(4):525--551, 2003.

\bibitem{BookDruetBlow}
O.~Druet, E.~Hebey, and F.~Robert.
\newblock {\em {Blow-up theory for elliptic PDEs in Riemannian geometry}}.
\newblock Princeton Univ Pr, 2004.

\bibitem{DuyckMerleNLW}
T.~Duyckaerts and F.~Merle.
\newblock {Dynamic of threshold solutions for energy-critical wave equation}.
\newblock {\em Int. Math. Res. Papers.}, 2008(2), 2008.

\bibitem{FernandezGuerBurger}
E.~Fern{\'a}ndez-Cara and S.~Guerrero.
\newblock {Null controllability of the Burgers system with distributed
  controls}.
\newblock {\em Systems \& Control Letters}, 56(5):366--372, 2007.

\bibitem{FernandZuazheat}
E.~Fern{\'a}ndez-Cara and E.~Zuazua.
\newblock {Null and approximate controllability for weakly blowing up
  semilinear heat equations}.
\newblock {\em Annales de l'Institut Henri Poincare/Analyse non lineaire},
  17(5):583--616, 2000.

\bibitem{FrancfortMuratoscillation}
G.~Francfort and F.~Murat.
\newblock {Oscillations and energy densities in the wave equation}.
\newblock {\em Comm. in Partial Diff. equations}, 17(11):1785--1865, 1992.

\bibitem{PGGall2001}
I.~Gallagher and P.~G{\'e}rard.
\newblock {Profile decomposition for the wave equation outside a convex
  obstacle}.
\newblock {\em Journal de math{\'e}matiques pures et appliqu{\'e}es},
  80(1):1--49, 2001.

\bibitem{defectmeasure}
P.~G\'erard.
\newblock {Microlocal Defect Measures}.
\newblock {\em Comm. Partial Diff. eq.}, 16:1762--1794, 1991.

\bibitem{linearisationondePG}
P.~G\'erard.
\newblock {Oscillations and Concentration Effects in Semilinear Dispersive Wave
  Equations}.
\newblock {\em Journal of Functional Analysis}, 141:60--98, 1996.

\bibitem{PGdefautSobolev}
P.~G\'erard.
\newblock {Description du d\'efaut de compacit\'e de l'injection de Sobolev}.
\newblock {\em ESAIM: Control, Optimisation and Calculus of Variations},
  3:213--233, 1998.

\bibitem{ibrahim2004gon}
S.~Ibrahim.
\newblock {Geometric-Optics for Nonlinear Concentrating Waves in Focusing and
  Non-Focusing Two Geometries}.
\newblock {\em Communications in Contemporary Mathematics}, 6(1):1--24, 2004.

\bibitem{ibrahim2003existNLW}
S.~Ibrahim and M.~Majdoub.
\newblock {Solutions globales de l'equation des ondes semi-lineaire critique a
  coefficients variables}.
\newblock {\em Bulletin de la Soci{\'e}t{\'e} Math{\'e}matique de France},
  131(1):1--22, 2003.

\bibitem{KapitanskiStrichartzvar}
L.V. Kapitanski.
\newblock {Some generalizations of the Strichartz-Brenner inequality}.
\newblock {\em Leningrad Math. J.}, 1(10):693--726, 1990.

\bibitem{KenigMerleNLW}
C.E. Kenig and F.~Merle.
\newblock {Global well-posedness, scattering and blow-up for the
  energy-critical focusing non-linear wave equation}.
\newblock {\em Acta Mathematica}, 201(2):147--212, 2008.

\bibitem{Keraanidefect}
S.~Keraani.
\newblock {On the defect of compactness for the Strichartz estimates of the
  Schrodinger equations}.
\newblock {\em Journal of Differential Equations}, 175(2):353--392, 2001.

\bibitem{KochTataruCarlLp}
H.~Koch and D.~Tataru.
\newblock {Dispersive estimates for principally normal pseudodifferential
  operators}.
\newblock {\em Comm. Pure and Appl. Math.}, 58(2):217--284, 2005.

\bibitem{KriegSchlagwavemap}
J.~Krieger and W.~Schlag.
\newblock {Concentration compactness for critical wave maps}.
\newblock {\em Arxiv preprint arXiv:0908.2474}, 2009.

\bibitem{LaurentNLSdim1}
C.~Laurent.
\newblock {Global controllability and stabilization for the nonlinear
  Schr\"odinger equation on an interval}.
\newblock {\em ESAIM Control Optimisation and Calculus of Variations},
  16(2):356--379, 2010.

\bibitem{LaurentNLSdim3}
C.~Laurent.
\newblock {Global controllability and stabilization for the nonlinear
  Schr\"odinger equation on some compact manifolds of dimension 3}.
\newblock {\em SIAM Journal on Mathematical Analysis}, 42(2):785--832, 2010.

\bibitem{Lionsconcentration}
P.L. Lions.
\newblock {The Concentration-Compactness Principle in the Calculus of
  Variations.(The limit case, Part I.)}.
\newblock {\em Revista matem{\'a}tica iberoamericana}, 1(1):145, 1985.

\bibitem{NakanishiScattering}
K.~Nakanishi.
\newblock {Scattering theory for the nonlinear Klein-Gordon equation with
  Sobolev critical power}.
\newblock {\em International Mathematics Research Notices}, 1999(1):31--60,
  1999.

\bibitem{RauchTaylordecay}
J.~Rauch and M.~Taylor.
\newblock {Exponential Decay of Solutions to Hyperbolic Equations in Bounded
  Domains}.
\newblock {\em Indiana Univ. Math. J.}, 24(1):79--86, 1975.

\bibitem{NLWShatahStruweAnnals}
J.~Shatah and M.~Struwe.
\newblock {Regularity results for nonlinear wave equations}.
\newblock {\em The Annals of Mathematics}, 138(3):503--518, 1993.

\bibitem{NLWShatahStruweIMRN}
J.~Shatah and M.~Struwe.
\newblock {Well-posedness in the energy space for semilinear wave equations
  with critical growth}.
\newblock {\em International Mathematics Research Notices}, 1994(7):303--309,
  1994.

\bibitem{tartarh-measure}
L.~Tartar.
\newblock {H-measures, a new approach for studying homogenisation, oscillations
  and concentration effects in partial differential equations}.
\newblock {\em Proceedings of the Royal Society of Edinburgh. Section A.
  Mathematics}, 115(3-4):193--230, 1990.

\bibitem{NLWzuilybourbaki}
C.~Zuily.
\newblock {Solutions en grand temps d'\'equations d'ondes non lin\'eaires.
  S\'eminaire Bourbaki, Vol. 1993/94}.
\newblock {\em Asterisque}, 227:107--144, 1995.

\end{thebibliography}
\end{document}